\documentclass[11pt,reqno,twoside]{amsart}

\usepackage{amsmath}
\usepackage{amsfonts}
\usepackage{amssymb}

\usepackage{epsfig}

\usepackage{xcolor}

\numberwithin{equation}{section}

\newif\ifdraft\drafttrue
\draftfalse

\newcommand{\br}{{\mathbb{R}}}
\newcommand{\bz}{{\mathbb{Z}}}
\newcommand{\m}{{m}}
\newcommand{\n}{{n}}

\newcommand{\vt}{{\bf{t}}}

\newcommand\hd{Hausdorff dimension}
\newcommand\ba{badly approximable}

\newcommand\da{Diophantine approximation}

\newcommand\ssm{\smallsetminus}

\newcommand\eq[2]{{\ifdraft{\ \tt [#1]}\else\ignorespaces\fi}\begin{equation}\label{eq:#1}{#2}\end{equation}}
\newcommand {\equ}[1]     {\eqref{eq:#1}}

\newcommand{\phid}{\overset{ \ \scriptscriptstyle\wedge_k}{ \varphi\ }}
\newcommand{\phione}{\overset{ \ \scriptscriptstyle\wedge_1}{ \varphi\ }}
\newcommand{\psitwo}{\overset{ \ \scriptscriptstyle\wedge_2}{ \psi\ }}

\newcommand{\R}{{\mathbb {R}}}
\newcommand{\Z}{{\mathbb {Z}}}
\newcommand{\N}{{\mathbb {N}}}

\newcommand{\Lip}{{\operatorname{Lip}}}

\newcommand{\Ad}{{\operatorname{Ad}}}
\newcommand{\ad}{{\operatorname{ad}}}

\newcommand{\SL}{\operatorname{SL}}

\newcommand{\ggm}{G/\Gamma}

\newcommand{\diam}{\operatorname{diam}}
\newcommand{\dist}{\operatorname{dist}}
\newcommand{\diag}{\operatorname{diag}}
\newcommand{\supp}{\operatorname{supp}}

\newcommand{\Lie}{\operatorname{Lie}}

\newcommand {\ignore}[1]  {}

\newcommand\hs{homogeneous space}

\newcommand{\ggplus}{{\underset{+}{\gg}}}

\newcommand\cic{C^\infty_{comp}}

\newcommand{\df}{{\, \stackrel{\mathrm{def}}{=}\, }}

\newcommand{\x}{{\mathbf{x}}}
\newcommand{\0}{{\mathbf{0}}}
\newcommand{\vv}{{\bf{v}}}
\newcommand{\vs}{{\bf{j}}}

\newcommand{\vy}{{\bf y}}

\newcommand{\vr}{{\bf i}}
\newcommand{\vm}{\mathbf{m}}

\newcommand{\Bad}{\mathrm{Bad}}

\newcommand{\vn}{{\bf n}}
\newcommand{\vw}{{\bf w}}

\newcommand{\p}{{\bf p}}

\newcommand{\q}{{\mathbf{q}}}

\newcommand{\vq}{{\bf q}}

\newcommand {\comm}[1]   {\textcolor{red}{#1}}
\newcommand {\gr}[1]     {\textcolor{green}{#1}}

\DeclareMathOperator{\codim}{codim}
\DeclareMathOperator{\spn}{span}

\newcommand{\fa}{\frak a}

\newcommand{\vre}{\varepsilon}

\newcommand\nz{\smallsetminus \{0\}}

\newtheorem{thm}{Theorem}[section]
\newtheorem{lem}[thm]{Lemma}
\newtheorem{prop}[thm]{Proposition}
\newtheorem{cor}[thm]{Corollary}

\newtheorem{remark}[thm]{Remark}
\newtheorem{defn}[thm]{Definition}

\begin{document}
\title[Dimension estimates
for  non-dense orbits in homogeneous spaces]
{Dimension estimates for the set of points with non-dense orbit in
homogeneous spaces}
\author{Dmitry Kleinbock}
\address{Department of Mathematics, Brandeis University, Waltham MA}
\email{kleinboc@brandeis.edu}

\author{Shahriar Mirzadeh}
\address{Department of Mathematics, Michigan State University, East Lansing MI}

 \email{mirzade2@msu.edu}

\begin{abstract} %\comm{Need to change it accordingly.}
Let $X = \ggm$, where $G$ is a Lie group and $\Gamma$ is a lattice in $G$, and let $U$ be a subset of $X$ whose complement is compact. We use the exponential mixing results for
diagonalizable flows on  $X$ to 
give upper estimates for  the Hausdorff
dimension of the set of points whose trajectories miss $U$. 
%\[\dim X - C\frac{{\mu ({\sigma _r}U)}}{{\log (1/r) + \log (1/\mu ({\sigma _r}U))}},\]
%where $C > 0$ is a constant independent of $r > 0$ and only
%dependent on $X$ and the flow $(g_t)$ acting on $X$, and $\mu$ is the
%Haar measure of $X$. In this way, we 
This extends a recent result of
Kadyrov \cite{K} and produces new applications to \da, such as 
% where $X$ was compact and $U$  which states that for sufficiently small
%$r>0$ the Hausdorff dimension of set of points that lie on
%trajectories missing the open ball $B(x_0;r)$ in $X$ is at most:
%\[\dim X - C\frac{{{r^{\dim X}}}}{{\log (\frac{1}{r})}},\]   \\
%where $x_0$ is a point in $X$ and $C$ is a constant independent of
%$r$. \\This result yields several corollaries. For example, we use
%this result to get 
an upper bound for the Hausdorff dimension of the
set of weighted uniformly badly approximable systems of linear forms, generalizing 
an estimate due to  Broderick and  Kleinbock  \cite{BK}.
% for badly
%approximable systems of linear forms with approximation constant $c$.
%\\
%, and use our techniques to derive an analogous lower bound for so-called `escaping sets.'
\end{abstract}

\thanks{Supported in part by NSF grant  DMS-1600814.}

%}
\date{July 13, 2019}

\maketitle

\section{Introduction}

{Throughout the paper, we let  $G$ be a Lie group and $\Gamma $   a %uniform L
lattice in $G$,
%Consider the %compact
denote by $X$ the homogeneous space  $G/\Gamma$ and by $\mu$   the
$G$-invariant probability measure on $X$.   The notation $${A\gg B} \quad{(\text{resp., }A\,\ggplus\, B)},$$ where  $A$ and $B$ are quantities depending on certain parameters, will mean ${A \ge  CB}$  {(resp., $A \ge  CB + D$)}, where {$C,D$ are constants} %in
dependent  only on $X$ and $F$.
%of those parameter
}
Let ${F^+:= }{({g_t})_{t
\ge 0}}$ be a one-parameter %{$\Ad$-}diagonalizable
{sub}semigroup of $G$.
% acting on
%{$X$} by left translations.
%
%
%We know that
{Following \cite{K}, for any subset $U$ of $X$  define the set %$E(F^+,U)$ the set
%It means that for
%any open subset $U$ of $X$ the set,
\eq{set} {E(F^+,U) :=  \{ x \in X: \overline{F^+x}
%{g_t}x
%\notin
 \cap U = \varnothing\}}}
of points in $X$ whose $F^+$-orbits stay away from $U$.
If the flow $(X,{\mu,}\,g_t)$ is ergodic, then the orbit ${\{ {g_t}x\}
_{t \ge 0}}$ is dense for {$\mu$-}almost all $x \in X$; hence $\mu\big(E(F^+,U)\big) = 0
$ whenever $U$ is non-empty.
%It follows
%has measure zero.
%{It is also known that whenever the flow is not quasiunipotent, the set of points $x$ with a non-dense orbit ${\{ {g_t}x\}
%_{t \ge 0}}$ has full \hd. }

%we may
{{A natural} question one can  ask is: how large can this set of  
measure zero be?  If the semigroup $F^+$ is \emph{quasiunipotemt}, that is, all eigenvalues of $\Ad\, g_1$ have absolute value $1$, then, whenever the action is ergodic and $U$ is non-empty, the set \equ{set} is contained in a countable union of proper submanifolds of $X$ -- this follows from Ratner's Measure Classification Theorem and the work of Dani and Margulis,   {see \cite[Lemma 21.2]{St} and  \cite[Proposition 2.1]{DM}}. On the other hand, if $F^+$ is not quasiunipotemt and $U = \{z\}$ for some $z\in X$, it is shown in \cite{K1} that the set \equ{set} has full \hd.

Fix a {right-invariant} Riemannian {structure  on $G$, and denote by
`${\dist}$' the corresponding Riemannian metric},
  {using} the same notation for the induced metric on $X$. {Also} denote by $B(r)$  the open ball of radius $r$ centered at the identity element of $G$, and  by $B(z,r)$   the open ball of radius $r$ centered at $z\in X$.  The aforementioned result of \cite{K1} can thus be stated as
\eq{dimlimit}{
\dim E\big(F^+,B(z,r)\big) \to \dim X \text{ as }  {r\to 0} .}
Here and hereafter $\dim E$ means the \hd\ of the set $E$, and $\codim E$ will stand for its Hausdorff codimension, i.e.\ the difference between the dimension of the ambient set and the \hd\ of $E$. 
{Until recently a {problem of estimating} the left hand side of  \equ{dimlimit}, or more generally,  the quantity $\dim E(F^+,U)$ where $U$ is a non-empty open subset of $X$, has not been addressed. In \cite{BK} Broderick and the first named author considered the case
\eq{slmn}{G = \SL_{m+n }(
\R), \ \Gamma = \SL_{m+n }( \Z), \ {X = \ggm},
}
with the action of $F^+ = (g_t)_{t\ge 0}$ where
\eq{gt}{g_t = \diag(e^{t/\m }, \ldots, e^{ t/\m }, e^{-t/\n },
\ldots,
e^{-t/\n })\,,%\in G
}
This action is important because of its Diophantine applications. In particular, a system of linear forms is badly approximable if and only if  (see \cite{dani}) the $g_t$-trajectory of a certain element of $X$ does not enter the set
% In \cite{BK} a specific family of complements of compact subsets of $X$ was considered, namely
 %\eq{ue}
 \eq{uepsilon}{ U(\vre) := \big\{g\Gamma \in  X: {\delta(g\Gamma)}
 %\|g\vv\| 
 < \vre %\text{ for some } \vv\in\Z^{m+n}\nz
 \big\}}
 for some $\vre > 0$, where 
 \eq{defdelta}{\delta(g\Gamma) := \inf_{\vv\in\Z^{m+n}\nz}\|g\vv\|.}
It was essentially\footnote{\cite[Theorem 1.3]{BK} is stated in a number-theoretic language; however it readily implies \equ{slmnbound}.in view of  \cite[Lemma 3.1]{BK}.
%  and an argument similar to the  proof of Theorem \ref{Main Theorem1} assuming Theorem \ref{Main Theorem} in \S\ref{main Section} of the present paper. 
%See \S\ref{weighted badly} for detailed explanations and a generalization.}
Note that
recently a precise asymptotic formula for the left hand side of  \equ{slmnbound} was obtained by  Simmons \cite{Si}: 
namely, that as $\vre \to 0$, the ratio $\frac{\codim 
 %X - \dim  
 E\left(F^+,U(\vre)\right)}{\vre^{m+n}}$ tends to a constant depending only on $m,n$.
 }
  shown there  that for
  %some constant $C$ depending only on $m,n$ and
  all $\vre > 0$ one has
 \eq{slmnbound}{{\codim 
 %X - \dim  
 E\big(F^+,U(\vre)\big)  {\,\gg}\ }\frac{{\vre^{m+n}}}{{\log (1/\vre)}}.}
The main ingredient of the proof in \cite{BK} was the exponential mixing of the $g_t$-action on $X$ (see \S\ref{h} for the definition).
This theme was continued by Kadyrov in \cite{K}, where an estimate similar to \equ{slmnbound} was proved for the \hd\ of $E\big(F^+,B(z,r)\big) $ under the assumptions that $X= \ggm$ is compact and the $F^+$-action on $X$ is exponentially mixing. Namely, it is shown there that  there  exist
$r_0 >0$ i%ndependent of $r$
such that for any $r \in (0,{r_0})$ and any $z
\in X$ %, such that
%the set $E\big(F^+,B(z,r)\big) $
%\[\{ x \in X:{g_t}x \notin B({x_0},r)\,\,\,for\,any\,\,t \ge 0\} \]
%Has Hausdorff dimension at most:
one has
\eq{shiralibound}{{ \codim 
%X - \dim 
E\big(F^+,B(z,r)\big) {\,\gg}\  }\frac{{r^{\dim X}}}{{\log (1/r)}}\,.}}
{In the present paper we strengthen Kadyrov's result in two ways: by considering more general open sets $U$ in place of balls $B(z,r)$, and by relaxing the assumption of compactness of $X$ to that of  compactness of $X\ssm U$. Our main theorem generalizes results from both \cite{BK} and \cite{K} and can be used to produce new applications to \da.}

{We need to introduce the following notation: for a subset $U$ of $X$ and $r > 0$ denote by $\sigma_rU$ the \emph{inner $r$-core\/} of $U$, defined as
$$
\sigma_rU :=  \{x\in X: {\dist}(x,U^c) > r\},
$$}
%Moreover, for a subset $U$ of $X$ let us define
and by  $\partial_rU$ the \emph{$r$-neighborhood\/}   of $U$ by 
 $$
\partial_rU :=  \{x\in X: {\dist}(x,U) < r\}.
$$Also, for {$x\in X$ denote by $\pi_x$ the map  $G\to X$ given by $\pi_x(g) := gx$, and by  $r_0(x)$ the \textsl{injectivity radius} of $x$: $$
{r_0(x) :=}\,\sup\{r > 0: %\text{the map }G\to X,\ g\mapsto gy
\pi_x\text{ is injective on }B(r)\}.$$
If $K\subset X$ is bounded, let us denote by $r_0(K)$ the \textsl{injectivity radius} of $K$: $$
r_0(K) := \inf_{x\in K}r_0(x) = \sup\{r > 0: %\text{the map }%G\to X,\ 
%g\mapsto gy
\pi_x\text{ is injective on }B(r)\  \ \forall\,x\in K\}.$$}
%\end{defn}

Here is the main result of the paper:
\begin{thm}
\label{Main Theorem} Let $G$ be a Lie group, $\Gamma $   a %uniform L
lattice in $G$, $X = G/\Gamma$, and let ${F^+}$ be a one-parameter {$\Ad$-}diagonalizable
{sub}semigroup of $G$ 
whose action on $X$ is exponentially mixing.
Then there {exists $ r' > 0$}
% and $d>0$} 
such that for any % one-parameter $\Ad$-diagonalizable subgroup $F$
%={({g_t})_{t \ge 0}}
% of $G$ 
%.  (depending only on $G$ and $\Gamma$) {with the following property}: for
  $U\subset X$ such that $U^c$ is compact and
any
%\footnote{Even though our results are stated for any $r > 0$, in the proofs we assume $r$ to be small enough and then extend  the result if necessary by changing the implicit constant in the inequality.}
{$\,0 < r < \min\big(r_0(\partial_{1}U^c)%^d
,%1,
 r'\big )
 %/8
 $}     
%{$r > 0$}
one has 
% %and $r$ and just dependent on $X$
%and $g_t$ such that for any $0<r<r_0/2$, the set: \eq{main
%set}{E(r,{({%artial _{r/2}}{U^c})^c},x): = \{ h \in B(r/2):{g_t}hx
%\in {U^c}\,\,\,\,\forall t \ge 0\} }
 \eq{mainbound}{ {%\dim X - 
 \codim E(F^+,U)  {\,\gg}\ }\frac{{\mu ({\sigma_{r}U)}}}{{\log (1/r) + \log (1/\mu ({\sigma_{r}U))}}}\,.}
\end{thm}
We note that in the above inequality, as well as in similar statements below, the implicit constant in ${\,\gg}$ is independent of $U$ and $r$ and is only dependent on $X$ and $F$. 
Also note that the right hand side of  \equ{mainbound}  depends on  $r$ while the left hand side   does not. Since the inequality holds for all sufficiently small values of $r$, in  applications one needs to choose an optimal $r$ to strengthen the result. In particular,
it is not hard to see, by taking $U$ to be an open ball {of radius $r$} centered at $z$ and assuming that $X$ is compact,  that {Kadyrov's result}  \equ{shiralibound} is a special case of  \equ{mainbound}.}
%Note that there is an obvious relationship between  the  $r$-neighborhood of a set and the inner $r$-core of its complement: 
%\eq{nbhdcore}{(\partial_rU)^c = \sigma_rU^c\quad \text{for any subset }U \text{ of }X.}
%{Make one of the inequalities strict and describe the connection between the two quantities.}
Moreover
one has the following generalization:
% of Kadyrov's result:

%When $U$ is an r-neighborhood of the subset $S$ in $X$ we have the
%following corollary which is the generalization of  the main theorem
%in \cite{K}.

%\begin{cor}
%\label{Cor1} There exists $k'>0$ such that for sufficiently small
%$c>0$ the set ${\Bad_{\m ,\n }(c) = \left\{A\in \mr : \{g_tu_A\Z^{\m
%+\n }: t \ge 0\} \cap U_{\vre} = \varnothing\right\},}$ :  Has
%Hausdorff dimension at most:
%\[mn - k\frac{c}{{\log (\frac{1}{c})}}\]

%\end{cor}
%
\begin{cor}
\label{Cor2} Let %$G$, $\Gamma${, $X$} and
$F^+$ %and $r'$ 
be as in Theorem \ref{Main Theorem}. Assume that $X$ is compact.
%,
%{and let $r_0 = r_0(X)$ be the injectivity radius  of $X$}.
Then
{there exists $r' > 0$}
%a constant $C'>0$ (depending only on $G$ and $\Gamma$) 
such that
for any closed
subset $S$  of $X$ and any 
{$0<r<{r'}%\min\big(r_0(X)^d %,1
%, r'\big)
%/8
$} %{$r > 0$}
one has 
%\[\{ x \in X:{g_t}x \notin {\partial _r}S\,\,\,for\,\,any\,\,t \ge 0\} \]
%Has Hausdorff dimension at most:
\[{%\dim X -
\codim E(F^+,\partial_{r}S)  {\,\gg} \ }\frac{{\mu ({\partial _{r/2}}S)}}{{\log (1/r)}}\,.\]
%In particular,
{Consequently},
%\begin{itemize}
%\item[\rm (a)]
if $S\subset X$ is a $k$-dimensional {compact} embedded submanifold, then
{for some $C = C(S,F)$ and}
%some $C = C''(S)$ and
any {$0<r<{r'}
%\min\big(r_0(X)^d %,1
%, r'\big )
%/8
$}
%{$r > 0$}
  one has
\eq{forS} %\dim X - 
{{\codim E(F^+,\partial_{r}S)  {\,\ge}\ } {C}\frac{{r^{\dim X - k }}}{{\log (1/r)}}.}
%\end{itemize}
\end{cor}

{The case $k = 0$ and $S = \{z\}$ of \equ{forS}   coincides with   \equ{shiralibound}: it is easy to show, by looking at the proof, that $C({\{z\}},F)$ is independent on $z\in X$.}
\smallskip

Similarly to the previous  papers \cite{BK, K} on the subject, the
main theorem is deduced from
%the following
 {a result} that estimates
$$\dim E(F^+,\sigma_{r}U) \cap Hx,$$ where $x\in X$ and $H$ is the {\sl unstable horospherical subgroup} with respect to $F^+$, {defined as
\eq{uhs}{H := \{ g \in G:{\dist}({g_t}g{g_{ - t}},e) \to 0\,\,\,as\,\,\,t \to  - \infty \}.}}
%{to be defined in \S\ref{main Section}, see \equ{uhs}}.
More generally, in the following theorem we estimate $$\dim
E(F^+,\sigma_{r}U) \cap Px$$ for {$x\in X$ and} some proper
subgroups $P$ of $H$, namely those which have {\sl Effective
Equidistribution Property} (EEP,  see \S\ref{h} for {the
definition}) with respect to {the flow $(X,F^+)$}.  Note that for $P = H$ this
property follows from  the exponential mixing of the action, as
shown in \cite{KM1}.

\begin{thm}
\label{Main Theorem1} Let  $G$, $\Gamma $   and $X$ be as in  Theorem \ref{Main Theorem},  let ${F^+}$ be a one-parameter {$\Ad$-}diagonalizable
{sub}semigroup of $G$, and let $P$ be a subgroup of $H$ which  has property {\rm (EEP)} with respect to {the flow $(X,F^+)$}.
% whose action on $X$ is exponentially mixing. 
Then there exists {$r'' > 0$
%0<r''<1/8$ and $d>0$
}  such that %{for any
%$G$, $\Gamma$ and $F^+$ be as in Theorem \ref{Main Theorem}.,  and let
 %one-parameter $\Ad$-diagonalizable subgroup $F^+={({g_t})_{t
%\ge 0}}$ of $G$, any $P\subset H$ which  has property {\rm (EEP)} with respect to $F^+$,
%there exists
%a constant $C>0$ (depending only on $G$ and $\Gamma$) {with the following property}:
{for any
  $x \in X$, {any} $U\subset X$ such that $U^c$ is compact  and
any
{$\,0 < r < \min\big(r_0(\partial_{1/2}U^c)%^{d}
,r''\big)
$} 
%{$r > 0$}
one has   
%\eq{main
%set1}
\eq{lhs} {%\dim P - 
\codim \{ g \in P: gx \in E({F^ + },U)\}
%has Hausdorff dimension at most
{\,\gg}  \frac{{\mu ({\sigma _{r}}U)}}{{\log \frac{1}{r} + \log \frac{1}{{\mu ({\sigma _{r}}U)}}}}}\,.
}
\end{thm}
{The general statement of Theorem \ref{Main Theorem1}} makes it possible to derive a corollary involving simultaneous  \da\ with weights.
Take $$\vr = ({i_k}: k = 1,\dots, m)\text{ and }\vs = ({j_\ell}:  \ell = 1,\dots,  n)$$
%$\vr \in \mathbb{R}^m$ and $\vs \in \mathbb{R}^n$ 
{with \eq{simplex}{{i_k},{j_\ell} > 0\,\,\,\,and\,\,\,\,\sum\limits_{k = 1}^m {{i_k} = 1 = \sum\limits_{\ell = 1}^n {{j_\ell}} },}}
and define  the {\sl $\vr$-quasinorm} of $\x\in\R^m$ and the {\sl $\vs$-quasinorm} of $\vy\in\R^n$ by 
$$
\left\| \x \right\|_\vr:= {\max _{1 \le k \le
m}}{\left| {{x_k}} \right|^{1/i_k}}\text{ and }\left\| \vy \right\|_\vs:= {\max _{1 \le \ell \le
n}}{\left| {{y_\ell}} \right|^{1/j_\ell}}.$$
A system of linear forms given by $ A
\in {M_{m,n}}(\mathbb{R})$ is said to be \emph{$(\vr,\vs)$-badly approximable} if
\[
{\inf}_ {\p \in {\mathbb{Z}^m}, \ {\q
\in {\mathbb{Z}^n}\ssm \{ 0\} } } \left\| {A\q + \p} \right\|_\vr%^m 
 \left\| \q \right\|_\vs%^n
  > 0
\]
 %Namely, for a $k$-tuple $\w =
%({w_1},{w_2},...,{w_k})$, $k \in \mathbb{N}$,  with positive
%components, define the \emph{$\w$-quasinorm} ${\left\| . \right\|_\w}$ on
%$\mathbb{R}^k$ by
%\footnote{This notation is different from the one originally  introduced  in \cite{K2}.} 
%$${\left\| \x \right\|_\w} :=  {\max _{1 \le i \le
%k}}{\left| {{x_i}} \right|^{1/{%k
%w_i}}}.$$
This generalizes the notion of {(unweighted)} \ba\ systems of linear forms,  {which correspond to the choice of} equal weights
\eq{equalweights}{\vr = \vm := ( 1/m,\dots,1/m),\quad \vs = \vn := ( 1/n,\dots,1/n).}
Now for any $c>0 $ set \eq{Weighted badly}{\Bad_{\vr,\vs}(c) 
%\df
:=  \{
A\in M_{m,n} : \mathop {\inf }_{ {\p \in {\mathbb{Z}^m}, \ {\q
\in {\mathbb{Z}^n}\ssm \{ 0\} } }}\left\| {A\q + \p}
\right\|_\vr%^m
 \left\| \q \right\|_\vs%^n 
 \ge c \}. }
%It's obvious that the set of $(\vr,\vs)$-badly approximable linear forms
%is equal to the union $\mathop  \cup \limits_{c > 0} Ba{d_{\vr,\vs}}(c)$
It is known, {see \cite[Theorem 2]{PV} and \cite[Corollary 4.5]{KW}}, that for any $\vr,\vs$ {as in \equ{simplex}} the set of
$(\vr,\vs)$-badly approximable systems of  linear forms, which is
the union of the sets $\Bad_{\vr,\vs}(c) $ over $c > 0$,  has
Hausdorff dimension $mn$. One can ask for {an estimate for} the Hausdorff dimension of
$\Bad_{\vr,\vs}(c)$ for fixed $\vr$, $\vs$ and $c$. Our %main
goal in
\S \ref{weighted badly} is to {deduce} the following theorem {from Theorem \ref{Main Theorem1}:}
\begin{thm}\label{main theorem 3}
{There exists $c_0 > 0$ such that} for any
% is a constant $B>0$ f
%{For any $c > 0$ and any}
 $\vr,\vs$ as in \equ{simplex}
 and any $0<c<c_0$
%and %for sufficiently
%small}
%$0 < c < c_0$} %
%the set $\Bad_{\vr,\vs}(c) $  has Hausdorff dimension
%at most:
one has 
%\comm{We need to change the statement, now it looks as if the implicit constant does not depend on $\vr,\vs$. Also does $c_0$ depend on $\vr,\vs$?}
\[
{%mn - 
{\codim  \Bad_{\vr,\vs}(c) }{\,\gg}\\}  \frac{c}{{\log  \frac{1}{c}}}, \]
{where the implicit constant in ${\,\gg}$ is independent of $c$ {but depends on $\vr,\vs$}.}
\end{thm}
{This is a weighted generalization of \cite[Theorem 1.3]{BK}.
%, with
%the constant $B$ and the notion of 'sufficiently small' being
%independent of $c$ and $p = \min \{
%{i_1},{i_2},...,{i_m},{j_1},{j_2},...,{j_n}\} $ }.
{Note that in the paper \cite{Si},   mentioned in the footnote before \equ{slmnbound}, it is  shown that $\codim  \Bad_{\vm,\vn}(c) $ is  asymptotic to a constant times $c$ as $c\to 0$. 
%recently a precise asymptotic formula for the \hd\ of $
%\Bad_{\vm,\vn}(c)$ was obtained by  Simmons \cite{Si}. 
However
the methods of \cite{Si} do not seem to extend to the weighted case.} 

\smallskip

{The structure of the paper is as follows. In the next section we define exponential mixing and property (EEP), and, following  \cite{KM1, KM4}, show that the exponential mixing of the $g_t$-action on $X$ implies (EEP) for the expanding horospherical subgroup relative to $g_1$. In \S \ref{main Section}
we deduce Theorem \ref{Main Theorem} and Corollary \ref{Cor2} from Theorem \ref{Main Theorem1}. The next three sections   are devoted to proving Theorem \ref{Main Theorem1}. In \S \ref{weighted badly} we   prove Theorem \ref{main theorem 3} 
%Our proof of
%Theorem \ref{main theorem 3} consists in 
by reducing the problem to
dynamics on the space $\ggm$ with $G$ and $\Gamma$ as in \equ{slmn}
and \eq{generalgt}{g_t = {g_t^{\vr,\vs} :=} \diag({e^{{i_1}t}},\dots,{e^{{i_m}t}},{e^{
- {j_1}t}},\dots,{e^{ - {j_n}t}}).} %This is done using a generalized
%version of Dani's Correspondence  \cite{K2}  and then applying
Theorem \ref{Main Theorem1} is then applied to the subgroup
\eq{subgroup mix}
{P = \left\{ \left( {\begin{array}{*{20}{c}}
{{I_m}}&A\\
0&{{I_n}}
\end{array}} \right) : A\in  {M_{m,n}}(\mathbb{R})\right\}
}
of $G$, which, following \cite{KM4}, is shown in \S\ref{eepp} to  satisfy property (EEP) relative to the \linebreak
$g_t^{\vr,\vs}$-action. We conclude the paper with a few  remarks and open questions.

 \medskip
{\bf Acknowledgements.} The authors are grateful to the hospitality
of the %Mathematical Sciences Research Institute 
MSRI (Berkeley, CA)
where some parts of this work were accomplished. We also thank
Shirali Kadyrov for useful discussions and suggestions{, and a reviewer for helpful comments}.

% $\tilde{\pi}'$
{
\section{Exponential mixing implies (EEP) for $H$ }\label{h}

%\begin{proof}[]

We start with the definition of Sobolev spaces on {Lie groups and their homogeneous spaces.
%Riemannian
%manifolds  (see for example  \cite{EH} for more
%details).  
Let $G$ be a %connected semisimple 
Lie
group  %without compact factors 
and  $\Gamma$ %an irreducible
a discrete subgroup of $G$. Denote by $X$ the \hs\ $\ggm$ and by $N$ the dimension of $G$.  In what follows, $\|\cdot\|_{{p}}$ will stand for the $L^p$ norm, and $(\cdot,\cdot)$ for the inner product in $L^2(X,\mu)$, where $\mu$ is a (fixed) $G$-invariant measure on $X$. If $\Gamma$ is a lattice in $G$, we will always take $\mu$ to be the probability measure. {Note though that} %However 
much of the set-up below applies to the case $\Gamma = \{e\}$ and $X = G$.}

{Fix a basis $\{Y_1,\dots,Y_n\}$ 
for the Lie algebra
$\frak g$
of
$G$, and, given a smooth
 function $h \in C^\infty(X)$ and $\ell\in{\Z_+}$, define the ``{\sl $L^{{p}}$, order $\ell$" Sobolev norm} $\|h \|_{\ell{,p}}$ of $h $ by 
 $$
 \|h \|_{\ell{,p}} \df \sum_{|\alpha| \le \ell}\|D^\alpha h \|_{{p}},
 $$
 where 
 %$\|\cdot\|_2$ stands for the norm in $L^2(X,\mu)$, 
 $\alpha = (\alpha_1,\dots,\alpha_n)$ is a multiindex, $|\alpha| = \sum_{i=1}^n\alpha_i$, and $D^\alpha$ is a differential operator of order $|\alpha|$ which is a monomial in  $Y_1,\dots, Y_n$, namely $D^\alpha = Y_1^{\alpha_1}\cdots Y_n^{\alpha_n}$.
%\end{defn}
This definition depends on the basis,
however, a change of basis
would only  distort
$ \|h \|_{\ell{,p}}$
by a bounded factor. We also let 
$$C^\infty_2(X) = \{h \in C^\infty(X): \|h \|_{\ell{,2}} < \infty\text{ for any }\ell = \Z_+\}.$$
Clearly smooth compactly supported functions belong to $C^\infty_2(X)$. We will also use the operators $D^\alpha$ to define $C^\ell$ norms of smooth %compactly supported 
functions $f$ on $X$: 
$$
\|f\|_{C^\ell}  := \sup_{x\in X, \ |\alpha|\le \ell}|D^\alpha f(x)|.
$$}
\begin{defn}

\label{exp decay} {Let $F^+ = \{g_t : t\ge 0\}$ be a one-parameter subsemigroup of $G$, and let $X = \ggm$ where $\Gamma$ is a lattice in $G$.} We say that a flow $(X,F^+)$ is {\sl exponentially
mixing}  if there exist {$\gamma  > 0$ and $\ell\in\Z_+$}
such that for any   $\varphi, \psi \in  C^\infty_2(X) $
 and for any $t\ge 0$ one has
{\eq{em}{\left| {({g_t}\varphi ,\psi )- \int_X\varphi\, d\mu \int_X\psi\, d\mu}\right|   \ll  {e^{ - \gamma t}}{\left\| \varphi  \right\|_{\ell{,2}}} {\left\| \psi  \right\|_{\ell{,2}}}.}}
\end{defn}
%Here ${\left\| . \right\|_{\ell{,2}}}$ stands for the norms in the Sobolev
%space $W_\ell^2(X)$.

%\begin{defn}

%\label{ink} {For $y\in X$ let us denote by $r_0(y)$ the {\sl injectivity radius} of $y$, defined as
%$$
%r_0(y) \df  \sup\{r > 0: \text{the map }G\to X,\ g\mapsto gy\text{ is injective on }B(r)
%\  \forall\,x\in K
%\}.$$
%Also for a bounded subset $K$ of $X$ we will call $r_0(K) \df  \inf_{y\in X} r_0(y)$ the injectivity radius of $K$.}
%\end{defn}
%, if $K\subset X$ is bounded we will denote by  the \emph{injectivity radius} of $K$, defined as $$
%\sup\{r > 0: \text{the map }G\to X,\ g\mapsto gx\text{ is injective on }B(r)\  \forall\,x\in K\}.$$
%Equivalently,$ \sigma_r(U) = {\partial _r}{U^c})}^c}

{As is the case in many applications, we will use the exponential mixing to study expanding translates of pieces of certain subgroups of $G$. If $P\subset G$ is a subgroup with a fixed Haar measure $\nu$, $\psi$ a function on $X$, $f$ a function on $P$, $x\in X$ and $t\ge 0$, let us define
$$I_{f,\psi}({g_t},x):= \int_P
f({{h}})\psi(g_t{{h}}x)\,d\nu({{h}})\,.$$}

\begin{defn}\label{subgroup}
Say that a subgroup $P$ of $G$ has {\sl Effective Equidistribution
Property} {\rm (EEP)} with respect to the flow $(X,F^+)$
%, where $F^+  = \{g_t : t \ge 0\}$ is one-parameter subsemigroup of $G$, 
if 
 $P$ is normalized by $F^+$, and there exists %constants $%E,
 %C_1, T_0, 
% {a,b},
$ \lambda > 0$ and {$\ell \in {\N}$} such that 
 %for any compact $L\subset X$ there exists $T_L > 0$ with the following property:
for any %compact subset $L$, any 
{ $x \in X$ and $t > 0$ with \eq{conditionont}{t\  {\ggplus \log\frac{1}{r_0({x}) },} }
 any $f\in \cic(P)$ with $\supp f \subset B^P(1)$ and %for any 
any $\psi\in
C^\infty_2(X)$}
% and for any compact  $L\subset X$ there
%exist ${{K_4}}>0$ and   $\tilde C=\tilde C(f,\psi,\rho)>0$ such that for all
%and} any   %\frac{1}{\lambda} %\max\left(
%T_0, \log\frac{1}{C_1 \cdot r_0(\partial_2 L)}
%{a + b \log\frac{1}{r_0({x})
%(B(x,2))
%\right)
%\,,} 
it holds that
%with $\lfloor
% \vt\rfloor >\max \{ \frac{1}{\alpha }\log \frac{1}{{C_1  \cdot {{{K_4}}} \cdot \rho}},\frac{1}{\tilde\gamma}\log \frac{2}{{{r_0}}}\} $ we have:
\eq{eep}{\left| {{I_{f,\psi }}({g_t},x) - \int_P f\,d\nu {\mkern 1mu} \int_X \psi\,d\mu  {\mkern 1mu} } \right| {\ll}\ {\max (
%\mathop {\sup }\limits_X |\psi |,
{\left\| \psi  \right\|_{C^1}, \left\| \psi  \right\|_{\ell{,2}} }
%,{\left\| \psi  \right\|_{Lip}}
)} \cdot {\left\| f \right\|_{{C^\ell }}} \cdot {e^{ - \lambda t}}{\mkern 1mu}.}
\end{defn}
%\comm{Maybe simplify even further, replacing $\max (
%\mathop {\sup }\limits_X |\psi |,
%\left\| \psi  \right\|_{C^0}, {\left\| \psi  \right\|_{Lip}})$ with $\left\| \psi  \right\|_{C^1}$?}

Here $\nu$ stands for a Haar measure on $P$. {Note that the implicit constants in both \equ{conditionont} and \equ{eep} are independent on $f$, $\psi$, $t$ and $x$.} 
%\comm{Is this understandable enough? in fact what happens is that for any choice of implicit constants in \equ{conditionont} one gets a constant in \equ{eep}, maybe we should give more detail?} 
This definition is
quite involved but it is justified by the fact that  in many special cases {\equ{eep} can be derived from exponential mixing, for example when $P = H$, the unstable horospherical
subgroup relative to $F^+$.}
%, defined as in
%that is ,
%\equ{uhs},
%{H: = \{ g \in G:{\dist}({g_t}g{g_{ - t}},e) \to 0\,\,\,as\,\,\,t \to  - \infty \} }
%has property (EEP) with respect to $F^+$. 
{This was essentially proved in \cite{KM4}, together with another important example of a proper subgroup of $H$ with the same property, namely with $P$ as in \equ{subgroup mix}. We are going to revisit the argument from that paper and  make the constants appearing there explicit.
\begin{remark} \rm Note that it suffices to establish (EEP) for functions $\psi$ with $\int_X \psi\,d\mu = 0$: indeed, if $\psi_0 := \psi - \int_X \psi\,d\mu$, one clearly has $$I_{f,\psi_0 }({g_t},z)  = I_{f,\psi }({g_t},z) - \int_H f\,d\nu  \int_X \psi\,d\mu.$$
\end{remark}}

Let $\mathfrak{g}$ be a Lie algebra of $G$,
$\mathfrak{g}_\mathbb{C}$ its complexification, and for $\lambda
\in \mathbb{C}$, let $E_ \lambda$ be %a generalized 
the eigenspace of
$\Ad\, g_1$ corresponding to $\lambda$.
%\[{E_\lambda } = \{ {Y} \in {\mathfrak{g}_\mathbb{C}}:{(\Ad\, {g_1} - \lambda I)^j}{Y} = 0\,\,\,for\,\,some\,\,j\} \]
Let $\mathfrak{h}$, $\mathfrak{{h^0}}$, $\mathfrak{{h^ - }}$ be the
subalgebras of $\mathfrak{g}$ with complexifications:
\[{\mathfrak{h}_\mathbb{C}} = \spn({E_\lambda }: \left| \lambda  \right| > 1),\ \mathfrak{h}_\mathbb{C}^0 = \spn({E_\lambda }: \left| \lambda  \right| = 1),\ \mathfrak{h}_\mathbb{C}^ -  = \spn({E_\lambda }:\left| \lambda  \right| < 1).\]
Let $H$, ${H^0}$, ${H^ - }$ be the corresponding subgroups of $G$. 
{Note that} $H$ is {precisely} the
unstable horospherical subgroup {with respect to $F^+$} (defined in \equ{uhs}) and $H^-$ is {the}
stable horospherical subgroup defined by:
\[{H^ - } = \{ h \in G:{g_t}h{g_{ - t}} \to e\,\,\,as\,\,t \to  + \infty \}. \]
Since $\Ad\, g_1$ is assumed to be diagonalizable over $\mathbb{C}$, $ \mathfrak{g} $ is the direct sum of
$\mathfrak{h}$, $\mathfrak{{h^0}}$ and $\mathfrak{{h^ - }}$. Therefore $G$ is
{locally (at a neighborhood of identity) a} direct product of the subgroups $H$, ${H^0}$ {and} ${H^ - }$.
% while $H$ is horospherical with respect to $g_{-1}$.(Recall that a
%subgroup:
%\[U(g) = \{ h \in G|\,\,{g^l}h{g^{ - l}} \to e\,\,as\,\,l \to  + \infty \} \]
%of G is called horospherical with respect to $g \in G$ .) \\
{In what follows, if $P$ is a subgroup of $G$, we will denote by $B^{P}(r)$  the open ball of radius $r$ centered at the identity element with respect to the  %Riemannian 
metric  on $P$ corresponding to the Riemannian structure induced from $G$.}

{Denote the group ${H^ - }{H^0}$ by $\tilde H$, and fix ${0 < \rho < 1}$ with the following properties:}
{
%\begin{itemize}
\begin{equation}\label{rho1}%\tag{2.4a}
 \text{the multiplication map }\tilde H \times H\to G\text{ 
%, $(\tilde h,h)\mapsto \tilde hh$
is one to one  on }B^{{\tilde H}}(\rho) \times {B^H}(\rho), \end{equation}
%\item  
and
\begin{equation}\label{conjugate implied}
%\tag{2.4b}
{g_tB^{\tilde H}(r)g_{-t} \subset B^{\tilde H}(2r) \text{ for any $0<r<\rho$ and }t\ge 0}\end{equation}}
%\end{itemize}}
{(the latter can be done since  $F$ is $\Ad$-diagonalizable and  the restriction of the map $g \to g_tgg_{-t}$,
%:G \to G,g \to {g_ttg{g_{ -t}},
$t > 0$, to the subgroup $\tilde H$
% = {H^ - }{H^0}$
is non-expanding).} 
\smallskip

Let $\mu^G$ be the Haar measure on $G$ which locally projects to $\mu$, and 
let us choose Haar measures $\nu^-$, $\nu^0$ and $\nu$ on  $H^-$, $H^0$  and $H$
respectively,  normalized  so that 
$\mu$ is 
locally almost the product of $\nu^-$, $\nu^0$ and $\nu$.
More precisely, see \cite[Ch.~VII, \S 9, Proposition 13]{Bou},
%we mean that 
$\mu$ can be expressed via $\nu^-$, $\nu^0$ and $\nu$ in the
following way: for any $\varphi\in L^1(G,\mu^G)$ supported on a small neighborhood of idenity,
\eq{bou}{
{\int_{G} \varphi(g)\, d\mu(g)} = 
{\int_{H^-\times H^0\times H}\varphi(h^-h^0h)\Delta(h^0)\,d\nu^-(h^-)\,
d\nu^0(h^0) \, d\nu(h)}\,,%\tag 2.3  
}
where $\Delta$ is the modular function of (the non-unimodular group)
$\tilde H$.

Now we are going to show, following \cite{KM4}, that  $H$, the unstable horospherical subgroup of $G$ with respect to $F^+$, satisfies property  (EEP).  We will start with an auxiliary statement, {essentially\footnote{{The statement of \cite[Theorem 2.3]{KM4} featured a constant $E(\psi)$ in place of $\max\big(\|\psi\|_{C^1},\|\psi\|_{\ell{,2}}\big)$, but it is easy to see from the proof that $E$ depends linearly on  $\|\psi\|_{C^1}$ and $\|\psi\|_{\ell{,2}}$.}}} established in \cite[Theorem 2.3]{KM4}:
 \begin{thm}\label{th1} Suppose that the flow $(X,F^+)$ is {exponentially
mixing}, and let $\gamma$ and $\ell$ be as in %Definition \ref{exp decay}
\equ{em}.
%There exist $\gamma, \ell>0$ such that f
Then for any $f\in \cic(H)$,
$0 < r < \rho/2$ and $x \in X$, if

{\rm (i)} $\supp f \subset B^H(r)$, and

{\rm (ii)} $\pi_x$ is injective on $B^G(2r)$,

\noindent then for any 
$\psi\in
C^\infty_2(X)$  with $\int_X\psi\,d\mu = 0$ 
and 
 any $\,t\ge 0$ one has
$$
%\aligned
\left| I_{f,\psi}(g_t,x) %- \int_H f\,d\nu   \int_X \psi\,d\mu   
%\int_H f(h)\psi(g_thx)\,d\nu(h) 
%-  \int_H f\, d\nu \int_X\psi\,
%d\bar\mu
\right|\ll\ {\max\big(\|\psi\|_{C^1},\|\psi\|_{\ell{,2}}\big) }  \left(r %\cdot d(\psi) 
%\int_H |f|\,d\nu\,%d\nu 
{\|f\|_1} +
%\tilde E \cdot
%\|\psi\|_{\ell{,2}}
e^{-\gamma
t}r^{-(\ell + \tilde k/2)} \|f\|_{\ell{,2}}%\|\psi\|_{\ell{,2}} 
\right),
%\endaligned
%\tag 2.4  
$$
where $\tilde k = \dim \tilde H.$ 
%where $\gamma$ and $\ell$ are as in Theorem 2.1 and $N = m^2 + mn + n^2
%$-1 = \dim \tilde H$. 
 %and \newline $d(\psi) \df \sup_{x,y\in X}\frac{|\psi(x) -
%\psi(y)|}{\dist(x,y)}$ is the Lipschitz constant of $\psi$.
\end{thm}
%Let  $\nu^-$ and $\nu^0$ denote the Haar measures on  $H^-$, $H^0$ 
%respectively,  normalized  so that 
%$\mu$ is 
%locally almost the product of $\nu^-$, $\nu^0$ and $\nu$.
 %By the
%latter, 
%we mean for any $\varphi\in L^1(G)$ we have
%$$
%{\int_{H^-  H^0  H} \varphi(g)\, d\mu(g)} = 
%{\int_{H^-\times H^0\times H}\varphi(h^-h^0h)\Delta(h^0)\,d\nu^-(h^-)\,
%d\nu^0(h^0) \, d\nu(h)}\,, 
%$$
%where $\Delta$ is the modular function of (the non-unimodular group)
%$\tilde H$. 

Using this {and again following \cite{KM4}, we can establish}

\begin{thm}\label{thmheep}$H$ satisfies property {\rm (EEP) {with respect to the flow $(X,F^+)$}}.\end{thm}

{For the proof and for later applications we will need the following lemma, which is a modification of \cite[Lemma 2.4.7(b)]{KM1} and \cite[Lemma 2.2(a)]{KM4}:
\begin{lem}
\label{estimateKM} Let $G$ be a Lie group of dimension $N$. Then for each {$\ell\in \Z_+$} there exists   $M_\ell$ (depending only on $G$) with the following property: for any
$0 < {{\vre}} < 1$ there exists a nonnegative smooth function  $\varphi_{{\vre}}$ on $G$ such that
{
\begin{enumerate}
\item the support of $\varphi_{{\vre}}$ is inside the ball of radius ${{\vre}}$ centered at $e$;
\item $\|\varphi_{{\vre}}\|_1
= 1$; %\ %\int_X {{1_{O}}\,d\mu }  = 
%\mu
%(O)
\item
$\| \varphi_{{\vre}} \|_{C^\ell} \le M_\ell \cdot {{\vre}} ^{-(\ell+N)}$;
\item
$\| \varphi_{{\vre}} \|_{\ell{,p}} \le M_\ell \cdot {{\vre}}^{-(\ell+\frac{p-1}pN)}$.
%\item
%${\left\| {{\psi }} \right\| _{\Lip} } \ll  {M_1 }{\varepsilon ^{ - (N + 1)}}.$  \comm{maybe we can omit it, since it is just a special case of (4)? and why $\ll$?}
\end{enumerate}}
\end{lem}}
%\comm{I think it is ok to have this lemma without proof.}

\begin{proof} [Proof of Theorem {\ref{thmheep}}]Suppose we are given  %compact subset $L$, any 
$f\in \cic(H)$ with $\supp f \subset B^H(1)$,  %for any 
$\psi\in
C^\infty_2(X)$ with $\int_X\psi\,d\mu = 0$, % and for any compact  $L\subset X$ there
%exist ${{K_4}}>0$ and   $\tilde C=\tilde C(f,\psi,\rho)>0$ such that for all
 and $x \in X$. % and large enough $t$, its magnitude to be determined later.  
 Put $r = e^{-\beta t}$, where
$\beta$ is to be specified later, and {take $\ell$ as in  \equ{em}. Then, using %\cite[Lemma 2.2(a)]{KM4}
Lemma \ref{estimateKM} with $G$ replaced by $H$,} take a non-negative {smooth} function $\theta$ supported on $B^H(r)$ such that
 \eq{theta}{\int_H
\theta\,d\nu 
= 1\text{ and   }\|\theta\|_{\ell{,2}} 
\ll   r^{{-(\ell + k/2)}},} where $k = \dim H = N - \tilde k$.
%\tilde c  
Since $\nu$ is translation-invariant, one can 
 write
$$
\aligned
%\int_H
%f(h)\psi(g_t hx)\,d\nu(h) 
I_{f,\psi}(g_t,x) &= 
%\frac1{J_\vu}
\int_H
%{\tilde B}
%\tilde 
f(h)\psi(g_t hx)\,d\nu(h)
%f(h)\psi(g_t  h x)\,d\nu(h)
\int_H
\theta(y)\,d\nu(y)\\&= %\frac1{J_\vu}
\int_H
%{\tilde B}
\int_H
%\tilde 
f\big(y h\big)\theta(y)\psi\big(g_tyhx\big)\,d\nu(y)\,d\nu(h)\\&=\int_H
%{\tilde B}
\int_H
%\tilde 
f\big(y h\big)\theta(y)\psi\big(g_t y 
hx\big)\,d\nu(y)\,d\nu(h)\,.
\endaligned
 $$
{Note that, as long as $\theta(y)\ne 0$, the supports of all  functions of the form $h\mapsto
f\big(yh\big)$ are contained in  $\tilde B := B^H(2).$ 
We would like to apply Theorem \ref{th1} with $r = e^{-\beta t}$, $%\tilde z \df 
hx $ % \in K_\vre$ 
in place of $x$ and $$f_h(y)
:=  f\big(y h\big)\theta(y)$$ in place of $f$. It is clear that $\supp\, f_h \subset B^H(r)$ for any $h$, i.e.\ condition (i) of Theorem~\ref{th1} is satisfied. For other conditions we need to require $e^{-\beta t} \le \min \big(r_0(hx)/2, \rho/2\big)$. {Since $r_0(hx)\gg r_0(x)$ as long as $h\in \tilde B$, it}  amounts to assuming \eq{anotherconditionont}{2e^{-\beta t} \le {a_0}\min \big(r_0({x})
%\big(B(x,2)\big)
, \rho\big)}
{for some uniform constant $a_0 > 0$.}
Also, in view of \cite[Lemma 2.2(b)]{KM4} and \equ{theta}, we have $$\|f_h\|_{\ell{,2}} \ll \| f\|_{C^\ell} \|\theta\|_{\ell{,2}} 
\ll   
%\tilde c  
e^{(\ell + k/2)\beta t} \| f\|_{C^\ell}.$$
Then from Theorem \ref{th1} one gets
$$
\aligned
\left| I_{f,\psi}(g_t,x) \right|&=\left|\int_{\tilde{B}}\int_H
%\tilde 
f\big(yh\big)\theta(y)\psi\big(g_t y 
 hx\big)\,d\nu(y)\,d\nu(h)\right| \le 
\int_{\tilde B}\left|I_{f_h,\psi}(g_t,h x)
%\int_H\tilde 
%f(y)\psi\big(g_t y\tilde x\big)\,d\nu(y)
\right|\,d\nu(h)\\ &\ll  
%E\cdot
%\|\psi\|_{\ell{,2}}
%r\big(\pi_x(\supp f)\big)^{-s} \tilde c  r^{-(\ell+mn/2)}\|f\|_{\ell{,2}}
%e^{-\gamma t}
{\max\big(\|\psi\|_{C^1},\|\psi\|_{\ell{,2}} \big)} \left(e^{-\beta t} %\cdot d(\psi) 
\int_H |f_h|\,d\nu(h)\,%d\nu(y)\,%d\nu 
 +
%\tilde E \cdot
%\|\psi\|_{\ell{,2}}
e^{(\ell + \tilde k/2)\beta t} \|f_h\|_{\ell{,2}}%\|\psi\|_{\ell{,2}} 
\cdot  e^{-\gamma
t}\right) \nu(\tilde B)
\\ &\ll  {\max\big(\|\psi\|_{C^1},\|\psi\|_{\ell{,2}} \big)}
\left(\sup|f|\cdot e^{-\beta t} %\cdot d(\psi) 
 +
%\tilde E \cdot
%\|\psi\|_{\ell{,2}}
\| f\|_{C^\ell}  \cdot 
e^{-(\gamma - (2 \ell + \frac N2)\beta)
t}\right).\,
\endaligned
$$ }
An elementary computation shows that  choosing $\beta$  equalizing the
two exponents
%in the first and the third summands 
above will produce 
{$$  \beta = \lambda = \frac \gamma {1 +2 \ell+N/2}\,,$$ and therefore
 \equ{anotherconditionont} becomes equivalent to  \equ{conditionont} with some uniform constants $a,b$. This shows that  \equ{conditionont} implies \equ{eep}, and finishes the proof.}
 \end{proof}

\section{Proving Theorem \ref{Main Theorem} {and Corollary \ref{Cor2}}  }\label{main Section}

We now assume Theorem \ref{Main Theorem1} is true and give a proof of
Theorem \ref{Main Theorem}.

\begin{proof}[Proof of Theorem \ref{Main Theorem} assuming Theorem \ref{Main
Theorem1}]
Let {$r''$} be as in Theorem \ref{Main Theorem1}, and define
{\eq{radius}{r':= \min  \big(1/4, r'',\rho \big)}
where $\rho$ is as in \eqref{rho1}, \eqref{conjugate implied}.  For any $r \le \rho$ choose $s$ such that ${B%^G
}(s)$ is contained in the product  $B^{{\tilde H}}(r/4){B^H}(r/4)$.} %in $G$
Now take   $U\subset X$ such that $U^c$ is compact, and for  $x\in X$ denote %by ${E_{x,s}}$  the set 
\eq{reduced set}{{{E_{x,s}}:=}\, \{ g \in {B%^G
}(s):gx \in E({F^ + },U)\}.}
%Define: 
%
%{\eq{radius}{r':= \min  \big(1/8, r'',r_1 \big)}}
%Note that i
In view of the countable stability of Hausdorff dimension, in order to prove the theorem it suffices to prove that for any  $x \in X$,
\eq{need}{
\dim E_{x,s}  \le  \dim X - C\frac{{\mu ({\sigma _{r}}
{U})}}{{\log  \frac{1}{r} + \log \frac{1}{{\mu ({\sigma _{r}}{U})}}}}\ }
with the constant $C>0$  only {dependent on $X$ and $F$}. Indeed, $E({F^ + },U)$ can be covered by countably many sets $\{gx : g \in E_{x,s}\}$, with the maps 
{$\pi_x: E_{x,s} \to X$}
%, $g \mapsto gx$ 
being Lipschitz and at most finite-to-one.

Since every ${g\in B(s)}
%E_{x,s}}
$ can be written as $g = h'h$, where
$h' \in {B^{\tilde H}}(r/4)$ and $h \in {B^{H}}(r/4)$, 
for any $y \in X$ we can write
\eq{transition1}{
\begin{aligned}
{\dist}({g_t}gx,y) &\le {\dist}({g_t}h'hx,{g_t}hx) + {\dist}({g_t}hx,y)\\ &={\dist}\big(g_th'g_{-t}{g_t}hx,{g_t}hx\big)+ {\dist}({g_t}hx,y).\end{aligned}}
Hence  in view of \eqref{conjugate implied},
%and \equ{transition1} 
$g \in {E_{x,s}}$ implies that {$h{x}$ belongs to $E({F^ + },{\sigma
_{{r/2}}}U)$}, and by using
Wegmann's Product Theorem \cite{Weg} we conclude that: 
\eq{wegmann}
{\begin{split}
\dim {E_{x,s}} 
&\le {\dim} \left(\{ h \in {B^H}(r/4):hx \in E({F^ + },{\sigma
_{r/2}}U)\} \times {B^{\tilde H}(r/4)}% {H^ - }{H^0}
\right)\\
%& \le {\dim} \left(\{ h \in {B^H}(r/8):hx \in E({F^ + },{\sigma
%_{4r_2}}U)\} \times {B^{\tilde H}(r_2)}% {H^ - }{H^0}
%\right) \\
& \le {\dim} \big(\{ h \in {B^H}(r/4):hx \in E({F^ + },{\sigma
_{r/2}}U)\}\big) +  \dim  {\tilde H }.
\end{split}
}
\ignore{{which is not greater than ${\dim} \big(\{ h \in {B^H}(\frac{r}{8}):hx \in E({F^ + },{\sigma
_{4r_2}}U)\}\big) +  \dim  {\tilde H }$} }
%,  $r$ replaced with   
% $r/8$ and $P$ replaced with $H$. Note that $H$ satisfies  property (EEP), and, 
Since $ {\partial _{1/2}}({\sigma _{r/2}U)^c}$ is contained in $  {\partial _1}{U^c}$, we have: %\comm{need to rewrite it, I'll do it myself later -- DK}
\[%r/8<r< 
{r_0}({\partial _1}{U^c})  \leqslant {r_0}\left({\partial _{1/2}}({\sigma _{r/2}U)^c}\right).  \]
%and
%{\[%r/8<r< r'<r''.\]}
Therefore, {by Theorem \ref{thmheep} and Theorem \ref{Main Theorem1}  applied to $P = H$ and $U$ replaced by $\sigma_{r/2}U$}, there exists a  {constant $C>0$}, only dependent on $X$ and $F$, such
that  {the set $\{ h \in {B^H}(r/4):hx \in E({F^ + },{\sigma
_{r/2}}U)\}$ has Hausdorff dimension at most
\eq{estimateH}{\begin{split} 
\dim H - C\frac{{\mu ({\sigma _{r/4}}{\sigma _{r/2}}U)}}{{\log \frac{4}{r} + \log \frac{1}{{\mu ({\sigma _{r}}U)}}}} 
& \le \dim H - C\frac{{\mu ({\sigma _{r}}U)}}{{\log \frac{4}{r} + \log \frac{1}{{\mu ({\sigma _{r}}U)}}}} \\
& \le \dim H - C'\frac{{\mu ({\sigma _{r}}U)}}{{\log \frac{1}{r} + \log \frac{1}{{\mu ({\sigma _{r}}U)}}}},
%\end{array}}
\end{split}}
where {$C'= 2C$}. ($C'$ should be chosen so that we have $$C' \geqslant C \cdot \frac{{\log \frac{4}{r} + \log \frac{1}{{\mu ({\sigma _{r}}U)}}}}{{\log \frac{1}{r} + \log \frac{1}{{\mu ({\sigma _{r}}U)}}}} = C\cdot \left( {1 + \frac{{\log 4}}{{\log \frac{1}{r} + \log \frac{1}{{\mu ({\sigma _{r}}U)}}}}} \right).$$  Since {$r<1/4$}, we can choose {$C'= 2 C$.}) 
It follows from \equ{wegmann} and \equ{estimateH}  that 
%the set ${E_{x,s}}$ has Hausdorff dimension at most
$$
\dim {E_{x,s}} \le  \dim X - C'\frac{{\mu ({\sigma _{r}}U)}}{{\log \frac{1}{r} + \log \frac{1}{{\mu ({\sigma _{r}}U)}}}},
%\end{array}}
$$
which finishes the proof. }
\end{proof}

\begin{proof}[Proof of Corollary \ref{Cor2}] {Take $r'$ as in \equ{radius}. If $S = \varnothing$ there is nothing to prove. Otherwise, by   Theorem \ref{Main Theorem} 
{applied to $U = \partial_rS$ and  with $r/2$ in place
%\footnote{Note that $r< \min  \big(r_0(X)%,1
%,r' \big) $ implies $r/2< \min  \big(r_0(\partial_{1}U^c)%,1
%,r' \big)$, a condition needed to apply Theorem \ref{Main Theorem}.} 
of $r$}, there exists a constant $C>0$ 
independent of $S$ such that for any  {$0<r< \min  \big(r_0(X)%^d%,1
,r' \big)
$}, the set $E({F^ +
},{\partial _r}S)$  has Hausdorff codimension at most
\eq{partial 1}
{%\begin{array}{*{20}{l}}
%\begin{aligned}
%\dim X  -
 {C}\frac{{\mu \left({\sigma _{r/2}}({\partial _r}S)\right)}}{{\log \frac{2}{r} + \log \frac{1}{{\mu ({\sigma _{r/2}}({\partial _r}S))}}}}
 %\\
%&
 \ge  {%\dim  X  -  
 {C}\frac{{\mu ({\partial _{r/2}}S)}}{{\log \frac{2}{r}
+ \log \frac{1}{{\mu ({\partial _{r/2}}S)}}}}}. 
%\end{aligned}
%\end{array}
}
Since $S$ is non-empty, ${\partial _{r/2}}S$ contains a ball of radius $r/2$, so there exists a constant {$d_0$} independent of $r$ such that for any {$0<r< %\min  %\big(
r_0(X)%^d%,1
%,r' \big)\}
$} we have:
\eq{partial 2}
{\mu ({\partial _{r/2}}S) \geqslant {{d_0}}{r^N%{\dim X}
}.}
%By \equ{radius} 
Since {$r'<1/4$}, by combining \equ{partial 1} and \equ{partial 2} it is easy to see that %{since $\rho \le 1/2$} for any {$0<r<\rho (X)$},
the set $E({F^ +
},{\partial _r}S)$  has Hausdorff codimension at most
{\[%\dim X - 
{C}\frac{{\mu ({\partial _{r/2}}S)}}{{(N%\dim X
+1)\log \frac{1}{r}
+ \log 2 + \log \frac{1}{{{d_0}}}}} \ge  \frac{{C\log 4}}{{(N
%\dim X 
+1)\log 4 + \log2 + \log \frac{1}{{{d_0}}}}} \cdot\frac{{\mu ({\partial _{r/2}}S)}}{{\log \frac{1}{r}}}.\]}
%where $C' = \frac{{C\log 2}}{{(\dim X + 2)\log 2 + \log (1/c_0)}}$. 
 This proves the main part of the corollary.}

{For the ``consequently" part, if $S$ is a $k$-dimensional compact embedded submanifold in $X$, then it is
easy to see that %since {$S$} is compact, 
for some constant {$d_1$}
dependent on $S$ {and  for all $r< %\min  %\big(
r_0(X)%^d
$ one has}  \eq{partial 3}{\mu ({\partial _{r/2}}S)\, {\ge} \,{d_1}{r^{N%\dim  X 
- k}}.}
%\comm{I do not believe this statement. Why is it true? Are you saying that $c$ does not depend on $S$?}
%{You are right, maybe if we say $c$ is dependent on $S$ that would be enough? Or should be provide a proof? }
Therefore in this case, combining \equ{partial 1} and \equ{partial 3}, it is easy to see that for any  {$0<r< \min  \big(r_0(X)%^d  %,1
,r' \big)
$} one has
%the set $E({F^ + },{\partial _r}S)$ has
%Hausdorff codimension at least
%\[\dim X - 
{$$%c  {C'}
\codim  E(F^+,\partial_{r}S)  \ge\frac{{C\log 4}}{{(N%\dim X 
- k + 1)\log 4 + \log 2 + \log \frac{1}{{{d_1}}}}} \cdot\frac{{{r^{N%\dim  X 
- k}}}}{{\log \frac{1}{r}}} .$$}}
%= \dim X -  {C''}\ \frac{{{r^{\dim  X - k}}}}{{\log (\frac{1}{r})}}\]
%Where $C''=cC'$.\\
\end{proof}
%\comdima{I will write the proof later.}
\section{Reduction to a covering result} \label{main Section1}
In the next {three} sections our goal is to prove Theorem \ref{Main Theorem1}. Fix a
subgroup $P$ of $H$ that
%is normalized by $\{ {g_t}\} $ and
satisfies
(EEP) {relative to $F^+$}, and fix a Haar measure $\nu$  on $P$.  Put
%$N=\dim G$ and
 $L=\dim P$.  Also take $0<r''<1/8$ such that the exponential map from $\frak p := \Lie(P)$   to $P$ is $2$-bi-Lipischitz on the ball of radius $r''$ centered at $0\in \frak p$,
{The latter implies that there exist constants $c_1,c_2,c_3 > 0$  such that for any  $0<r<r''$ one nas
\eq{Ball measure}{{c_1}{r^L} \leqslant \nu \big(B^P(r)\big) \leqslant {c_2}{r^L}}
%\equ{Ball measure} implies that for any $r_1,r_2 <r''$ we have:
%\eq{measure ratio}{\frac{{\nu (B^P({r_1}))}}{{\nu (B^P({r_2}))}} \leqslant \frac{{{c_2}}}{{{c_1}}}   {\left(\frac{{{r_1}}}{{{r_2}}}\right)^L}.}
and
\eq{Ball measure derivative}{
\frac d{dr}\nu \big(B^P(r)\big) \le c_3r^{L-1}.}
}

For $x \in X$, $t>0, {k \in \N}$ and {a} subset {$S$} of $X$ we define
\eq{escape 1}{{{A}^P(t,r,{S} ,{k},x)\mathop  := \big \{ {h} \in
B^P(r) :{g_{{\ell} t}}{h}x \in {S}  \,\,\,  {\forall \ell \in \{1,2,\cdots,k  \}\mkern 1mu}   \big\}}.}
%and let $K$ be a compact subset of $X$ such that $\partial_{1/2}U^c\subset K$. 
%In this section we establish an estimate for  the Haar measure of ${A}^P(t,r,\sigma_r{U},x)$. 

{Also, let us define
$$\lambda_{\max}:=\max \{ |\lambda |:\,\lambda \text{ is an eigenvalue of }\ad_{{g_1}}|_{\frak p}\}.$$}
{One of our main goals in the next three sections will be to prove the following theorem:
\begin{thm}\label{main cov} Let ${F^+}$ be a one-parameter {$\Ad$-}diagonalizable
{sub}semigroup of $G$, and $P$ a subgroup of $G$ with property {\rm (EEP)}. Then there exist positive constants {$a,b,K_0,K_1,K_2$ and
%$0<r'' <1/8$
%,F'>0$ 
%only dependent on $X$ and 
%${{K_4}}>0$ 
$\lambda_1$} %dependent on $F$
%only dependent on $L$ and $X$ 
such that 
%For any $x\in X$,there exists $n \in \N$ and a constant ${}$ such that 
for any subset $U$ of $X$ whose complement is compact, %dependent on $F$
%only dependent on $L$ and $X$ 
%For any $x\in X$,there exists $n \in \N$ and a constant ${}$ such that 
any $0<r<r_0$ {where 
\eq{r0}{r_0 :=  \min \big(r_0(\partial_{1/2} U^c),{r''} \big),}}
any $x\in \partial_{
r}{U^c}$, %\partial_{
%r}{U^c}$, 
%\comm{(you wrote that here it is required to have $x\in \partial_{
%r}{U^c}$, but why?)} 
 {$k\in\N$ and any
{\eq{t estimate}{t> a +b \log \frac{1}{r},}}}
%there exist $x_1,x_2,x_3,...,x_k \in \partial_r{U^c}$ such that for $t>\log 2D$ , 
the set ${A}^P\Big(t,{\frac{r}{16 \sqrt{L}}},{U^c},{k},x\Big)$ can be covered with at most
$$
K_0{e^{Lk \lambda_{\max}t}} \left(1 - K_1 \mu (\sigma_rU) +\frac{K_2 e^{- \lambda_1 t}}{r^L}  \right)^k$$
balls in $P$ of radius $re^{-k\lambda_{\max}t}$. 
\end{thm}}

{It is not hard to see a connection between the above theorem and Theorem \ref{Main Theorem1}: indeed,
for any $x\in X$ the intersection of the set in the left hand side of \equ{lhs} with
$B^P\big(\frac{r}{16 \sqrt{L}}\big )$ is contained in ${A}^P\Big(t,{\frac{r}{16 \sqrt{L}}},{U^c},{k},x\Big)$  for any $t > 0$ and any $k\in\N$. Thus the covering constructed in Theorem \ref{main cov} can be used to estimate the \hd\ of the intersection of the set $ \pi_x^{-1}\big(E(F^+,U)\big)$ with $P$ from above.}

\begin{proof}[Proof of Theorem \ref{Main Theorem1} assuming Theorem \ref{main cov}]
{First note that the statement of Theorem \ref{Main Theorem1} involves just the semigroup $F^+$ as a whole and does not depend on its parametrization. Thus,}
{applying a linear time change to the flow $g_t$, without {loss of}  generality {for the proof of the theorem} we can assume that}
%Define
%{\eq{eigen value}
{${\lambda_{\max} = 1}$}.

Let {$0<r<r_0$}. {We are again going to use the notation $E_{x,s}$ introduced in  \equ{reduced set}.} {{In view of the countable stability of Hausdorff dimension it suffices to  {find $s > 0$ such}   that for any  $x \in X$,
\eq{reduced}{
\dim {\left(E_{x,s} \cap P \right)}  \le  \dim X - C' \frac{{\mu ({\sigma _{r}}
{U})}}{{\log  \frac{1}{r} + \log \frac{1}{{\mu ({\sigma _{r}}{U})}}}}\ }
with the constant $C'>0$  only {dependent on $X$ and $F$}}.}

Note that  {$E_{x,r/2} \cap P=\varnothing$}  for any $x \,{\notin \partial_{r}U^c}$,  so in this case \equ{reduced} is clearly satisfied for $s=r/2$. So, let $x \in \partial_{r}U^c$ {and take} %suffices to prove that \equ{reduced} is satisfied for 
$s=\frac{r}{16 \sqrt{L}}$. 

 Let ${\underline {\dim } _B}$ denote the lower box dimension. { Since  for any $t > 0$ we have
 $${E_{x,\frac{r}{16 \sqrt{L}}} \cap P} \subset {\bigcap_{k \in \N}}{A}^P\Big(t,\frac{r}{16 \sqrt{L}},U^c,{k},x\Big),$$ 
from Theorem \ref{main cov},  in view of the assumption {${\lambda_{\max} = 1}$},
%\equ{eigen value}}, 
it follows that  
 %$$
 %{\begin{aligned}
 %\underline {{{\dim }}}_ B  {\left(E_{x,\frac{r}{16 \sqrt{L}}} \cap P \right)} &\le 
 %\mathop {\liminf }\limits_{k \to \infty }   \frac{\log \left(K_0{e^{Lkt}} \left(1 - K_1 \mu (\sigma_rU) +\frac{K_2 e^{- \lambda_1 t}}{r^L}  \right)^k\right)}{-\log (re^{- kt})}\\
%   &=  \mathop {\liminf }\limits_{k \to \infty } 
%   \frac{\log K_0 + Lkt + k \log \left(1 - K_1 \mu (\sigma_rU) +\frac{K_2 e^{- \lambda_1 t}}{r^L}  \right)}{-\log  r + kt}\\
%& = L + \frac{\log \left(1 - K_1 \mu (\sigma_rU) +\frac{K_2 e^{- \lambda_1 t}}{r^L}\right)}{t}.
%    \end{aligned}}
%    $$
% we obtain
%by Theorem \ref{main cov}  {and in view of \equ{eigen value}} we have: 
\eq{es in}
{\begin{aligned}
 \underline {{{\dim }}}_ B  {\left(E_{x,\frac{r}{16 \sqrt{L}}} \cap P \right)} &\le 
 \mathop {\liminf }\limits_{k \to \infty }   \frac{\log \left(K_0{e^{Lkt}} \left(1 - K_1 \mu (\sigma_rU) +\frac{K_2 e^{- \lambda_1 t}}{r^L}  \right)^k\right)}{-\log (re^{- kt})}\\
   =  \mathop {\liminf }\limits_{k \to \infty } \ 
  & \frac{\log K_0 + Lkt + k \log \left(1 - K_1 \mu (\sigma_rU) +\frac{K_2 e^{- \lambda_1 t}}{r^L}  \right)}{-\log  r + kt}\\
& = L + \frac{\log \left(1 - K_1 \mu (\sigma_rU) +\frac{K_2 e^{- \lambda_1 t}}{r^L}\right)}{t}
    \end{aligned}}
    whenever $t$ satisfies 
{\equ{t estimate}.}}
 {It remains to choose an optimal $t$.
 % satisfying 
% {\eq{t estimate}{t> a +b \log \frac{1}{r}.}} 
Take $q $ to be} a natural number which satisfies the following conditions: 
\eq{optimization}{
\begin{gathered}
{(\tfrac{1}{8})}^q< \frac{K_1}{2K_2},  \hfill \\
q> \lambda_1   b -L,
\end{gathered} }
and set $$t = a + \frac{{L + q}}{{\lambda_1}}\log
% \left(
\frac{1}{{r\mu ({\sigma _{r}}U)}}
%\right)
.$$ %Note that $r<r''<1/8$. So, i
It is easy to see that in view of \equ{optimization},   $t$ {as above} satisfies \equ{t estimate}, and we have
\eq{optimal1}
{
\begin{split}
\frac{{K_2} e^{-\lambda_1 t}}{r^L} 
 & =  {K_2} {r^{ - L}}{e^{ - \lambda_1 (a+ \frac{{L + q}}{{\lambda_1}}\log \frac{1}{{r\mu ({\sigma _{r}}U)}})}}\\
 & =  e^{-\lambda_1 a} {K_2} {{r^{ - L}}}{{r}^{L + q}} {\mu {({{\sigma _{r}}{U})}}}^{L + q} %\\ 
 %&
 =  e^{-\lambda_1 a} {K_2} \cdot r^q \cdot {\mu {({{\sigma _{r}}{U})}}}^{L + q} \\
 & <  e^{-\lambda_1 a} {K_2} \cdot {(\tfrac{1}{8})^q} \cdot \mu {({{\sigma _{r}}{U})}} %\\
 %&
  < e^{-\lambda_1 a} {K_2}  \frac{K_1}{{2K_2}} \cdot \mu {({{\sigma _{r}}{U})}} \le \frac{K_1}{2}\mu ({{\sigma _{r}}{U}}). 
\end{split} 
}
Combining \equ{es in} and \equ{optimal1},  we have:  %for $r \in (0,{r_0}/2)$:
\[\begin{split} 
\dim  {\left( E_{x,\frac{r}{16 \sqrt{L}}} \cap P \right)} %&\le \underline {{\dim }}_B E_{x,\frac{r}{16 \sqrt{L}}} \\
& \le L + \frac{{\log \left(1 - \frac{K_1}{2}\mu ({\sigma _{r}}U)\right)}}{{{}t}} %\\
%&
 \le L - \frac{{\frac{K_1}{2}\mu ({\sigma _{r}}U)}}{{{}t}} \\
& = L - \frac{{\frac{K_1}{2} \cdot \mu ({\sigma _{r}}U)}}{{{}  \frac{{(L + q)}}{{\lambda_1}}\cdot \log \frac{1}{{r\mu ({\sigma _{r}}U)}}}} %\\
%&
 = L - C' \cdot \frac{{\mu ({\sigma _{r}}U)}}{{\log \frac{1}{{r}} + \log \frac{1}{{\mu ({\sigma _{r}}U)}}}} ,
\end{split}
\]
where $C' = \frac{{K_1 \lambda_1}}{{2{} (L + q)}} $. % is some constant dependent on $X$ and $F$.
 This finishes the proof.
\end{proof}

%\comm{Note: it is clear that we can assume $\lambda_{\max} = 1$ for the proof of Theorem \ref{Main Theorem1}. But it is not obvious to me that we can make this assumption while proving Theorem \ref{main cov}. Its statement, for example \equ{t estimate}, does not seem to be invariant with respect to linear changes of variable. We should either carefully reduce to the case $\lambda_{\max} = 1$ or prove this theorem for the general case, which means that $\lambda_{\max} = 1$ should appear in the subsequent sections.}

\section{A measure estimate} \label{main Section0}
{Our goal in this section is to prove the following proposition which gives a lower bound for the measure of sets %of the form 
\eq{ap}{{A}^P\Big(t,{\frac{r}{16 \sqrt{L}}},\sigma _{r/2}{U},1,x \Big)  =  \left \{  {h} \in
B^P\Big(\frac{r}{16 \sqrt{L}}\Big ) :{g_{ t}}{h}x \in \sigma _{r/2}{U}    \right\}}
 whenever $t$ satisfies 
{\equ{t estimate},} and $x$ belongs to ${\partial _{r}}{U^c}$.}
%We get the following upper bound for the %Lebesgue
%measure of  {sets}  $E_1^P(t,r,{\partial _{r/2}}{U^c},x')$:
\begin{prop} \label{exponential mixing} {Let ${F^+}$ be a one-parameter {$\Ad$-}diagonalizable
{sub}semigroup of $G$, and $P$ a subgroup of $G$ with property {\rm (EEP)}. Then there exist positive constants $a,b,E', \lambda '$ such that 
for any  $U\subset X$ such that $U^c$ is compact, %dependent on $F$
%only dependent on $L$ and $X$ 
%For any $x\in X$,there exists $n \in \N$ and a constant ${}$ such that 
any $x\in \partial_{
r}{U^c}$,  %\partial_{
%r}{U^c}$, 
%\comm{(you wrote that here it is required to have $x\in \partial_{
%r}{U^c}$, but why?)} 
any $0<r<r_0$ where $r_0$ is as in 
\equ{r0}, and any $t$ satisfying 
\equ{t estimate} one has}
%} \comm{do you think we can replace it by  $t\ggplus \log (\frac{1}{r})$? is it important that the multiplicative constant is precisely $1/\lambda'$?} 
%we have:
{\eq{conclusion}{ {\mathop {\inf }\limits_{x \in {\partial _{r}}{U^c}} \nu \left({A}^P\Big(t,{\frac{r}{16 \sqrt{L}}},{\sigma _{r/2}}{U},{1},x \Big)\right)\ge \nu\left(B^P\Big(\frac{r}{16 \sqrt{L}} \Big)\right)\mu ({\sigma _{r}}U)  - E'{e^{ - \lambda 't}}.}}}
%where
%{\eq{et}{\eta(r,t) :=  \mathop {\inf }\limits_{x \in {\partial _{r}}{U^c}} \nu \left({A}^P\Big(t,{\frac{r}{16 \sqrt{L}}},{\sigma _{r/2}}{U},{1},x \Big)\right).}}
\end{prop}
\ignore{It is easy to see that 
%\eq{connection}
${{{\partial _r}U^c} \subset (\sigma _rU)^c}$ for any $r>0$ and any $U$;
%For getting an upper bound for the Lebesgue measure of the set
%$E_1^P(t,r,{\partial _{r/2}}U,x)$ 
%In view of \equ{connection}, 
therefore, in order to prove Proposition \ref{exponential mixing} it suffices to get a lower bound for
the measure of the   set 
\eq{complement}{
%N(t,r,U,x)\df
%\mathop  = \limits^{def}
{\{ p \in {B^{{P}}}(r):{g_t}px \in {\sigma _{r/2}}U\}.}}
which is included in the complement of ${A}^P(t,r,{\partial _{r/2}}U^c,x)$   in ${B^P}(r)$:}
% which are the
%following sets:
%\comm{Here we have the same ambivalence as in the definition of (EEP): does $E'$ depend on the implicit constants in $\ggplus$?}
To prove %the inequality 
\equ{conclusion}  we will apply (EEP) to smooth approximations of
{$1_{B^P({\frac{r}{16 \sqrt{L}}})}$ and ${1_{{\sigma _{r/2}}U}}$}.
% for sufficiently small
%$r$. 
In order to extract useful information from (EEP) we will need to bound the {norms of the derivatives} of those approximations.
The next two  lemmas will be used to approximate {${1_{{{\sigma _{r/2}U}}}}$}  and
%the
%second lemma
%{Lemma \ref{ball estimate}} to approximate 
{$1_{B^P({\frac{r}{16 \sqrt{L}}})}$} respectively.

\begin{lem}
\label{estimate} %For each {$\ell\in \Z_+$} there exists   $M'_\ell$ (depending only on $X$) with the following property. 
Let $O$ be a nonempty open subset of $X$,
%  with null boundary, 
and  let $0<\varepsilon_0 <1$, ${\delta}<1$
%. Choose
%${M_{1,l}}$, ${M_{m,l}}$, 
%${\varepsilon _{O,{\delta}}}
%> 0$  
be such that %the following holds:
\eq{inner}{{\delta}\mu (O) \le \mu ({\sigma _{\varepsilon_0} 
%_{O,{\delta}}
}O) < \mu (O).}
%(since $\mu (O)>0$ and $\mu (\partial O)=0$ this selection of $\varepsilon _{O,{\delta}}>0$ is
%possible).
Then for any $0 < \varepsilon \le {\varepsilon
_0
%{O,{\delta}}
}$ one can find {a nonnegative function $\psi_\varepsilon \in C_{comp}^\infty (X)$} such that: %the following holds:
\begin{enumerate}
\item ${\psi _\varepsilon } \le {1_O}$;
\item ${\delta}\mu (O) \le \int_X {{\psi _\varepsilon }\,d\mu}$; %\ %\int_X {{1_{O}}\,d\mu }  = 
%\mu
%(O)
{ \item
${\left\| {{\psi _\varepsilon }} \right\|_{\ell{,2}} } \le {4^\ell}{M_\ell }{\varepsilon ^{ - \ell}}$;
\item
${\left\| {{\psi _\varepsilon }} \right\|_{C^\ell} } \le {4^\ell}{M_\ell }{\varepsilon ^{ - \ell}}$,
%\item
%${\left\| {{\psi }} \right\| _{\Lip} } \ll  {M_1 }{\varepsilon ^{ - (N + 1)}}.$  \comm{maybe we can omit it, since it is just a special case of (4)? and why $\ll$?}
}
\end{enumerate}
{where $M_\ell$ is as in Lemma \ref{estimateKM}}.
\end{lem}

\begin{proof}
Let $O$ be a nonempty open subset of $X$, and let {$0< \varepsilon_0 <1$} and ${\delta}<1$
%. Choose
%${M_{1,l}}$, ${M_{m,l}}$, 
%${\varepsilon _{O,{\delta}}}
%> 0$  
be such that %the following holds:
\equ{inner} holds. %Take $\varepsilon _{O,{\delta}}$ as in \equ{inner}. 
Since $O$ is open and the function $x\mapsto \dist(x,O^c)$ is continuous,   for
any $0<\varepsilon  < {\varepsilon _0}$ we have:
\[{\delta}\mu (O) < \mu ({\sigma _\varepsilon }O) < \mu (O).\] By the inner regularity of $\mu$
  we can find a compact subset ${A_\varepsilon } \subset {\sigma _\varepsilon }O$ such that:
\[{\delta}\mu (O) \le \mu ({A_{{\varepsilon}}}) \le \mu ({\sigma _{{\varepsilon}}}O) < \mu (O).\]
{ }
Denote by $A_\varepsilon ^ + ,A_\varepsilon ^{ +  + }$  the closed
$\frac{{{\varepsilon}}}{4}$ and {$\frac{{{\varepsilon}}}{2}$}
neighborhoods of ${A_\varepsilon }$.  Since $A_{{\varepsilon}}$ is
compact, these sets are compact as well.
%Since $\mu (O)>0$ and $\mu (\partial O)=0$ this selection of $\varepsilon _{O,{\delta}}>0$ is
%possible. \\%Now if $0 < \varepsilon  < {\varepsilon _{O,{\delta}}}$ then by
%the definition above ${{O'}_{{\varepsilon _{O,{\delta}}}}} \subset {{O'}_\varepsilon }$  \comm{What is ${O'}_\varepsilon$?}
{Now %choose $c = \frac{\vre}{4}$ and 
take ${\psi _\varepsilon } =
%({h_\varepsilon } \circ \phi )
\varphi_{{\vre/4}}*{1_{A_\varepsilon ^ + }}$, where $\varphi_{{\vre/4}}$ is as in Lemma \ref{estimateKM}. {Since}t $\varphi_{{\vre/4}}$ is supported on 
$B^G(\varepsilon/4)$, the support of the function
$\psi_\varepsilon$ is contained in $A_\varepsilon ^{ +  + } \subset
O$, }
% \comm{(I think we need to use a constant $K$ here)} 
so property $(1)$ holds. Furthermore, $\psi_\varepsilon=1$ on
$A_\varepsilon$, therefore:
\[\mu (O) \ge \int_X {{\psi _\varepsilon }\,d\mu  \ge \mu ({A_\varepsilon }) \ge {\delta}\mu (O)} ,\]
which gives us  property (2). 
{Let $\alpha= (\alpha_1,\dots,\alpha_N) $ {be} such that $\left| \alpha  \right| \le \ell$. For any $x \in X$ we have %the following:
}
{ 
%\eq{norm estimate}
\begin{equation*}{\begin{aligned}
 \left| {{D^\alpha }{\psi _\varepsilon }(x)} \right| &
 = \left| {{D^\alpha }( \varphi_{\vre/4}  * {1_{A_\varepsilon ^ + }} )(x)} \right| =  \left| {{D^\alpha }\varphi_{\vre/4}  * {1_{A_\varepsilon ^ + }} (x)} \right| %= \left| {\int_G {{D^\alpha }\varphi_c (g){1_{A_\varepsilon ^ + }}({g^{ - 1}}x)\,d\mu (g)} } \right| \hfill 
 \\
  &
   \le {\left\| D^ \alpha{{\varphi_{{\vre/4}}}} \right\|}_1  \le {\left\| {{\varphi_{{\vre/4}}}} \right\|}_{\ell,1} \le M_{\ell}  {(\tfrac\varepsilon{4}) ^{ - \ell}}   ,
\end{aligned} }
\end{equation*}
and {likewise},  
%which implies (3). {Similarly, 
by Young's inequality, 
%\eq{L2 estimate}
$${{\left\| D^ \alpha{{\psi _\varepsilon }} \right\|}_2 \le \|{D^\alpha }\varphi_{\vre/4}  * {1_{A_\varepsilon ^ + }} \|_2 \le {\left\| D^ \alpha{{\varphi_{{\vre/4}}}} \right\|}_1 \cdot {\left\| {{1_{A_\varepsilon ^ + }}} \right\|}_2 \le  {\left\| D^ \alpha{{\varphi_{{\vre/4}}}} \right\|}_1 \le M_{\ell}  {(\tfrac\varepsilon{4}) ^{ - \ell}},}
$$%}
which implies (3) and (4).}
%Therefore %, by \equ{L2 estimate}:
%\eq{norm estimate 1}{{\left\| {{\psi _\varepsilon }} \right\|_{\ell{,2}} } \le  {\left\| {{\varphi_{{\vre/4}}}} \right\|}_{\ell,1}  \le   M_{\ell} \cdot {(\frac{\varepsilon}{4}) ^{ - \ell}}.}
%where $C_\ell:= \# \{\alpha : \left| \alpha  \right| \le \ell \}.$
%So, if we define $$M'_\ell:=4^{\ell}  M_\ell ,$$ we can conclude property (3) from \equ{norm estimate 1} and property (4) from \equ{norm estimate} and we are done.}
\end{proof}
%As a corollary of 
Similarly to the proof of the above lemma, one can get the smooth estimations for  characteristic
functions of small balls in $P$ (we omit the proof for brevity):
\begin{lem}
\label{ball estimate}
% For a  Lie group $P$ of dimension
%$L$ and {$\ell \in \Z_+$}, t
{For any $\ell \in \Z_+$ there exist constants
$M'_\ell> 0$  (depending only on $P$)}
% (these constants are equal exactly to the ones in the previous
%lemma) 
such that the following holds:
for any $\varepsilon ,r
> 0$   there exist functions ${f_\varepsilon }:P \to [0,1]$ such
that
\begin{enumerate}
\item ${f_\varepsilon } = 1\,\,on\,\,\,{B^P}(r)$;
\item ${f_\varepsilon } = 0\,\,on\,\,\,{\big({B^P}(r + \varepsilon
)\big)^c}$;
{\item ${\left\| {{f_\varepsilon }} \right\|_{\ell{,2}}} \le {M'_{\ell}}{\varepsilon ^{  -\ell}}$;
\item ${\left\| {{f_\varepsilon }} \right\|_{C^\ell}} \le {M'_{\ell}}{\varepsilon ^{ - 
\ell}}$. }
\end{enumerate}
\end{lem}
\begin{proof}[Proof of Proposition \ref{exponential mixing}] 
{Let $\ell$ and $\lambda$ be as in  Definition \ref{subgroup}, and let $a,b{,E_1}$ be the implicit constants in \equ{conditionont} and \equ{eep} such that %\equ{eep} holds for 
$t>a +b \log \frac{1}{r_0(x)}$
implies
\eq{eepnew}{\left| {{I_{f,\psi }}({g_t},x) - \int_P f\,d\nu {\mkern 1mu} \int_X \psi\,d\mu  {\mkern 1mu} } \right| \le E_1\ {\max (
%\mathop {\sup }\limits_X |\psi |,
{\left\| \psi  \right\|_{C^1}, \left\| \psi  \right\|_{\ell{,2}} }
%,{\left\| \psi  \right\|_{Lip}}
)} \cdot {\left\| f \right\|_{{C^\ell }}} \cdot {e^{ - \lambda t}}{\mkern 1mu}} for any $f$ and $\psi$ as in Definition \ref{subgroup}. Then choose $\lambda'>0$ such that 
{\eq{exponent}{\lambda  - { {2\ell}}\lambda ' > \lambda '\quad \text{ and }\quad1/\lambda' > b. }}}
{Now let  $U\subset X$ be such that $U^c$ is compact, and take $0<r<r_0$ and $x \in \partial_{r}U^c$}. 
 %Recall that we are given $x \in \partial_{r}{U^c} $.
% where %$U\subset X$ is connected and
%$U^c$ is compact. %Note that it suffices to prove the Proposition for $r>0$ in a dense subset of $(0,\frac{r_0}{2})$. \comm{Why?} Let $0<r<r_0/2$ be a continuity point of $r \to \mu (\overline {{\sigma _r}U} )$ (Since the function $r \to \mu (\overline {{\sigma _r}U} )$ is increasing, the set of its continuity pointa is dense). We have:
%%\[\mu (\partial ({\sigma _r}U)) = \mu (\overline {{\sigma _r}U} ) - \mathop {\lim }\limits_{r' \to {r^ - }} \mu (\overline {{\sigma _{r'}}U} ) = 0\]
%\comm{How and where is the above equality used?}
If $\mu(\sigma_{r}U) = 0$, \equ{conclusion} is trivially satisfied; thus let us assume that $\mu(\sigma_{r}U) > 0$. 
%Note that since the function $r \to \mu ( {{\sigma _r}U} )$ is non-increasing, the set of its continuity points is dense. 
Then %take $0<r'< r$ such that $r'/2$ is a continuity point of this function, and 
put $$O := \sigma_{r/2} U$$ and take 
%\[\mu (\partial O) = \mathop {\lim }\limits_{r \to {{r'}^ + }} \mu ({\sigma _r}U) - \mu ({\sigma _{r'}}U) = 0\]
$$\delta :=  \frac{{\mu ({\sigma _{r}}U)}}{{\mu({\sigma _{r/2}}U)}}.$$ 
Note that \equ{inner} holds with $\varepsilon_0 = r/2$. %in Lemma \ref{estimate}; s
Also, since $U$ is open, the function $x\mapsto \dist(x,U^c)$ is continuous,  which implies that $\delta < 1$. 

%Observe thatwe can take ${\varepsilon _{O,\delta}} = r/2$ in
%proposition
%\equ{inner}. 
%Take $E,T_0,\ell, \lambda$ as in Definition \ref{subgroup} \comm{(There are no $E,T_0$ in that definition, only $\ell, \lambda$, need to change something here)}, and 
{Now} 
%choose $\lambda'>0$ such that 
%{\eq{exponent}{\lambda  - { {2\ell}}\lambda ' > \lambda '\quad \text{ and }\quad1/\lambda' > b, }}
%and
%{ $$1/\lambda' > b . $$}
%Now 
set $f = {1_{B^P({\frac{r}{16 \sqrt{L}}})}}$ 
 %and $\psi  = 1_{\sigma _{r/2}U }$ 
 {and take  \eq{t estimate 1}{t \ge a+ \frac{1}{{\lambda '}}\log \frac{2}{r}> a+b \log \frac{2}{r}> a+b \log \frac{1}{r_0(x)}}
{(the last inequality holds since $x \in \partial_{r}U^c$). Also define}
% \eq{epsilon}
 $${\varepsilon : = {e^{- \lambda ' t }}}.$$
{Note that $\vre < r/2$ in view of \equ{t estimate 1}}.}
%  \eq{rad es}{\vre < r/2 .} }.
So  let us apply Lemma \ref{estimate} with $\vre _0=r/2$, 
%and $\vre$ 
and Lemma \ref{ball
estimate} with %$\vre$ and 
$\frac{r}{16 \sqrt{L}}$ in place of $r$.  Let ${\psi
_\varepsilon}$ and ${f _\varepsilon}$ be the corresponding functions. Then we have 
%, by Lemma \ref{estimate} and Lemma \ref{ball
%estimate}:
\eq{prime}{
\begin{aligned}
\max (\|\psi_{{\varepsilon}}\|_{C^1},{\left\| \psi_{{\varepsilon}}  \right\|_{\ell{,2}} }) \cdot {\left\| f_{{\varepsilon}} \right\|_{{C^\ell }}} \cdot {e^{ - \lambda t}} 
& \le \max (\|\psi_{{\varepsilon}}\|_{C^\ell},{\left\| \psi_{{\varepsilon}}  \right\|_{\ell{,2}} }) \cdot {\left\| f _{{\varepsilon}}\right\|_{{C^\ell }}} \cdot {e^{ - \lambda t}} \\ 
%\comm{\text{need to work from here}} \ 
& \le  4^\ell M_{\ell} {\varepsilon ^{ - \ell }}{M'_{\ell}}{\varepsilon ^{ - \ell }} e^{-\lambda t} \\
&  
=  4^\ell M_{\ell} {M'_{\ell}} {e^{2 \ell \lambda '-\lambda t}} %\\
%& 
 \le   4^\ell M_{\ell} {M'_{\ell}} {e^{-\lambda 't}} .%\\
%& 
%= E_1 {e^{-\lambda 't}}
\end{aligned}}
%where
%\[{E_1} = 4^\ell M_{\ell} {M'_{\ell}}.\]
%$E_1$ is a constant independent of $U$ and $r$ and dependent only on $X$,  $F$ and $P$. %Define  $\psi  = 1_{\sigma _{r/2}U }$. 
% {Also $\supp\,({f_\varepsilon })\subset B^P(r_{0}) \subset B^P(r_0(\partial_rU^c)) \subset B^P(R)$}, 
{Note {also} that %$\|\psi_{\vre}\|_2 < \infty$, 
$\supp f_{\vre} \subset B^P(\frac{r}{16 \sqrt{L}}+r/2) \subset B^P(1)$.  In view of \equ{t estimate 1} {and \equ{prime}, %$t>a+b \log \frac{1}{r_0(x)}$.} So, by applying 
the estimate \equ{eepnew} can be applied to ${\psi
_\varepsilon}$, ${f _\varepsilon}$, $x$ and $t$, and yields}
% and using 
%, {there exists $E_2>0$ independent of $x$ and $t$ such that:
}
%\eq{estimate final}
$${\begin{split}
\int_P {{f_\varepsilon }({h})} \psi_\varepsilon ({g_t}{h}x)\,d\nu ({h}) 
&\ge \int_P {{f_\varepsilon }\,d} \nu \int_X {{\psi _\varepsilon }{\mkern 1mu} \,d\mu }  - {4^\ell M_{\ell} {M'_{\ell}} {{E_1}}}{e^{ - \lambda 't}}.
 \end{split}}$$
%Recall that ${A}^P\left(t,\frac{r}{16 \sqrt{L}},\sigma_{r/2} U,{1},x\right)  = \left\{ {h} \in
%{B^{P}}\big(\frac{r}{16 \sqrt{L}}\big):{g_{t}}{h}x \in \sigma_{r/2} U {\mkern 1mu}\right\}$. 
{In view of \equ{ap}} we have:
%\eq{estimate final 1}
$${\begin{aligned}
\nu \left({A}^P\Big(t,\frac{r}{16 \sqrt{L}},\sigma_{r/2} U,{1},x\Big) \right) 
&= \int_P {f({h}){1_{\sigma_{r/2} U}}({g_t}} {h}x)\,d\nu ({h}) \\ 
\ge \int_P {f({h})\psi_\varepsilon ({g_t}} {h}x)d\nu ({h})  &\ge \int_P {{f_\varepsilon }({h})\psi_\varepsilon ({g_t}} {h}x)\,d\nu ({h}) - \int_{P} {|{f_\varepsilon } - f|\,d\nu}  \\
 \ge \int_P {{f_\varepsilon }({h})\psi_\varepsilon ({g_t}} {h}x)\,d\nu ({h}) &- \nu\left(B^P\Big(\frac{r}{16 \sqrt{L}} + {e^{ - \lambda ' t}}\Big)\ssm B^P\Big(\frac{r}{16 \sqrt{L}}\Big)\right) 
.\end{aligned}
}$$
%Now consider the function $\phi(s)=\nu\big(B^P(\frac{s}{16 \sqrt{L}})\big)$; it is not hard to see that for some positive constant $E_2$
% $c_1,c_2,c_3$ 
%one has
%\eq{2sided}{c_1r^L \le \phi(r) \le c_2r^L}
%and
%\eq{estder}
%$${ \phi'(s) \le  {E_2} \cdot s ^{L-1}}$$
% for any $s \in (0,r_0)$. 
% is differentiable and for some constant $s>0$ independent of $r$ we have $f'(r)= sLr^{L-1}$. 
By the mean-value theorem {and \equ{Ball measure derivative}}, for some $\frac{r}{16 \sqrt{L}}<s<\frac{r}{16 \sqrt{L}}+{e^{ - \lambda 't}}$ it holds that
%\eq{mean}{
\begin{equation*}\begin{aligned}\nu \left( {{B^P}\Big(\frac{r}{16 \sqrt{L}} + {e^{ - \lambda 't}}\Big)\ssm {B^P}\Big(\frac{r}{16 \sqrt{L}}\Big)} \right)  = \nu\left(B\Big(\frac{r}{16 \sqrt{L}} + e^{ - \lambda 't}\Big)\right)& -  \nu\left(B\Big(\frac{r}{16 \sqrt{L}}\Big)\right) \\ %\le {e^{ - \lambda ' t}} \phi'(s) %\\
 \le {c_3  {e^{ - \lambda 't}} s^{L-1}}  \le {c_3} {e^{ - \lambda 't}} %{E_2}
   \left(\frac{r}{16 \sqrt{L}}+\frac{r}{2}\right)^{L-1} %\\ 
%& =  {e^{ - \lambda ' t}}{E_2} \left(\frac{r}{16 \sqrt{L}}+\frac{r}{2}\right)^{L-1} 
&\le {c_3}{e^{ - \lambda 't}}. \end{aligned}\end{equation*}
%Thus we have by \equ{estimate final}, \equ{estimate final 1} and \equ{mean}:
{Combining the above computations, we obtain}
\begin{align*}
\nu &\left({A}^P\Big(t,\frac{r}{16 \sqrt{L}},\sigma_{r/2} U,{1},x\Big)\right)
 %{\ge \int_P {{f_\varepsilon }({h})\psi_\varepsilon ({g_t}} {h}x)\,d\nu ({h}) - \nu \left( {{B^P}\Big(\frac{r}{16 \sqrt{L}} + {e^{ - \lambda 't}}\Big)\ssm {B^P}\Big(\frac{r}{16 \sqrt{L}}\Big)} \right) }\\
%&
\ge \int_P {{f_\varepsilon }({h})\psi_\varepsilon ({g_t}} {h}x)\,d\nu ({h}) - {c_3}{e^{ - \lambda 't}} \\
&\ge \int_P {f_\varepsilon }\,d\nu\int_X {\psi_\varepsilon \,d\mu  - {4^\ell M_{\ell} {M'_{\ell}} {{E_1}}} {e^{ - \lambda 't}}}  -   c_3{e^{ - \lambda 't}} \\
&\ge  \nu \left( {{B^P}\Big(\frac{r}{16 \sqrt{L}} \Big)} \right) \frac{{\mu ({\sigma _{r}}U)}}{{\mu
({\sigma _{r/2}}U)}} \cdot \mu ({\sigma _{r/2}}U) - ( {4^\ell M_{\ell} {M'_{\ell}} {{E_1}}} + {c_3}){e^{ - \lambda 't}} \\
&= \nu \left( {{B^P}\Big(\frac{r}{16 \sqrt{L}} \Big)} \right)\mu ({\sigma _{r}}U) - {E'}{e^{ - \lambda 't}} 
\end{align*}
%% \[F = \left( \begin{gathered}
%  n \hfill \\
%  1 \hfill \\ 
%\end{gathered}  \right) + \left( \begin{gathered}
%  n \hfill \\
%  2 \hfill \\ 
%\end{gathered}  \right) + ... + \left( \begin{gathered}
%  n \hfill \\
%  n \hfill \\ 
%\end{gathered}  \right)\]
%and 
\ignore{
So we have for $r \in
(0,\frac{{{r_0}}}{2})$ and t as \equ{t estimate} :
\[\begin{array}{*{20}{l}}
{\nu ({A}^P(t,r,{{{\sigma _{r/2}}{U}}},x) \le \nu (B^P(r)) - c_1 \mu ({{{\sigma _r}{U}}}){r^L} + (E_1+c_3){e^{ - \lambda 't}}}\\
{{\mkern 1mu} {\mkern 1mu} {\mkern 1mu} {\mkern 1mu} {\mkern 1mu}
{\mkern 1mu} {\mkern 1mu} {\mkern 1mu} {\mkern 1mu} {\mkern 1mu}
{\mkern 1mu} {\mkern 1mu} {\mkern 1mu} {\mkern 1mu} {\mkern 1mu}
{\mkern 1mu} {\mkern 1mu} {\mkern 1mu} {\mkern 1mu} {\mkern 1mu}
{\mkern 1mu} {\mkern 1mu} {\mkern 1mu} {\mkern 1mu} {\mkern 1mu}
{\mkern 1mu} {\mkern 1mu} {\mkern 1mu} {\mkern 1mu} {\mkern 1mu}
{\mkern 1mu} {\mkern 1mu} {\mkern 1mu} {\mkern 1mu} {\mkern 1mu}
{\mkern 1mu} {\mkern 1mu} {\mkern 1mu} {\mkern 1mu} {\mkern 1mu}
{\mkern 1mu} {\mkern 1mu} {\mkern 1mu} {\mkern 1mu} {\mkern 1mu}
{\mkern 1mu} {\mkern 1mu} {\mkern 1mu}  = \nu (B^P(r)) - c_1\mu
({{{\sigma _r}{U}}}){r^L} + {\mkern 1mu} E'{e^{ - \lambda 't}}},
\end{array}\] }
where $E' {:=4^\ell M_{\ell} {M'_{\ell}} {{E_1}}} +{c_3}$.
% is a constant independent of $U$.
% and $r$ and just dependent on $X$ and $F$.
\end{proof}

}

\section{Tessellations of $P$ and Bowen boxes: proof of Theorem \ref{main cov}}
{ In order to prove  {Theorem \ref{main cov}} it will be instrumental to use a technique of {\sl tessellations} of nilpotent Lie groups, as developed in \cite{KM1}. It allows one to cover subsets of $P$  with objects that behave like non-overlapping cubes in a Euclidean space.  In this aspect our method differs from the one by Kadyrov \cite{K}: using {\sl Bowen boxes} defined below, as opposed to Bowen balls considered in \cite{K}, turns out to be a   more efficient way to cover $P$ {(see \equ{compwithkadyrov} below and the subsequent footnote for explanation).}
%\comm{Indeed, we should refer to the precise place in \cite{K} which is not correct and the corresponding place in our paper which replaces that step.} 
We are going to revisit {{the} construction} {in \cite{KM1}} and then use it to {find efficient coverings of sets of the form} $A^P\left(t,{\frac{r}{16 \sqrt{L}}}, {{U}^c,{k}},x\right)$.}

Let us say that an open subset $V$ of $P$ is a {\sl tessellation domain} for $P$ relative to a countable subset $\Lambda$ of $P$ if 
\begin{itemize}
\item
$\nu (\partial V) = 0.$
\item
$V \gamma_1 \cap V \gamma_2 = \varnothing$ for different $\gamma_1,\gamma_2 \in \Lambda$.
\item
$P = \bigcup\limits_{\gamma  \in \Lambda } {\overline V \gamma }.$
\end{itemize} 

{Note that $P$ is a connected   simply connected nilpotent Lie group. Let $I_P\subset \frak p = \Lie(P)$ be the cube centered at $0$ with side length $1$ with respect to a suitably chosen basis of $\frak p$.  
%By the construction in \S $3$ of \cite{KM1}, since $P$ is a nilpotent Lie group of dimension $L$, there exists a suitable basis $\{X_1,X_2,\dots, X_L\}$ of the Lie algebra of $P$ such that if ${I_P} = \{ \sum\limits_{j = 1}^P {{x_j}{X_j}:\left| {{x_j}} \right|}  < 1/2\} $ represents the unit cube in Lie algebra of $P$, then for any $R \ge 1$, $V=\exp(\frac{1}{R}I_P)$ is a tessellation for $P$ relative to $\Lambda=\exp(\frac{1}{R} \Z X_1)$.
For any $r  > 0$ let us define {$V_r:=\exp (\frac{r}{4 \sqrt{L}} I_P)$.} Then, as shown in  \cite[Proposition 3.3]{KM1}, $V_r$ is a tessellation domain for $P$ relative to some discrete subset $\Lambda_r$ of $P$.}
%{Take $0<r''<1/8$ sufficiently small so that the exponential map is 2-bi-lipischitz on $\frac{r}{16 \sqrt{L}}I_P$ for any $r<r''$.} Let $0<r<r''$. Since $V_r$ is the image of {$\frac{r}{16 \sqrt{L}}I_P$ that is a} cube of side length $\frac{r}{16 \sqrt{L}}$ under the exponential map and 
{Since the
exponential map is $2$-bi-Lipschitz on {$\frac{r}{4 \sqrt{L}} I_P$} for $r < r''$}, we have
{\eq{Bowen inc}{{B^P}\Big(\frac{r}{{16 \sqrt L }}\Big) \subset {V_r} \subset {B^P}(r/4)}}

{{Also} it is easy to see that there exists {${K_3} > 0$} such that for any ${\delta} \le 1$ \eq{cube d}{\nu \big(\{ h \in P:{\dist}(h,\partial {V_r}) <  {\delta} \} \big)<{{K_3}  {\delta}.}}}

\ignore{Applying a linear time change to the flow $g_t$, without {loss of}  generality we can assume that 
%Define
\gr
{\eq{eigen value}
{\lambda_{\max} = 1.}}}
Define
\eq{b1}{\lambda_0 := \min \{ |\lambda |:\,\lambda \text{ is an eigenvalue of }\ad_{{g_1}}|_{\frak p}\} .}
%Note that by reparametrizing the flow $\lambda_0=1$. 
%Let $0 < b \le b_2 \le \cdots \le b_L=1$ be eigenvalues of  $\ad_{{g_1}}|_{\Lie ( P)}$. 
%Since $g_t$ is Ad-diagonalizable for sufficiently small $r' , r''>0$ there exists a $2$-bi-Lipchitz homeomorphism $\phi :{B^{R^L}(r'')} \to {B^P}(r') $ such that for standard basis ${\{ {e_i}\} _{1 \leqslant i \leqslant m}}$ of 
%$\R^L$ and for any $x<r''$:
%\eq{basis}{{g_{-t}}\phi ({xe_i}){g_{  t}} = \phi ({e^{-t{b_i}}}{xe_i})}
%\equ{basis} 
%{Again, since} 
{Again using the  bi-Lipschitz property of $\exp$, we can conclude that}
%on {$\frac{r}{16 \sqrt{L}}I_P$}, 
%for any $0<r<r''$, any $t>0$ and any {$p_1,p_2 \in V_r$} we have:
%\eq{basis1}{d(g_{-t}p_1g_{t},g_{-t}p_2g_t)<4e^{-\lambda_0 t}d(p_1,p_2).}
%{Therefore, 
for any $0<r<r''$ and any $t>0$ one has
\eq{diam}{\diam(g_{-t}V_rg_t)< 2re^{-\lambda_0 t} %<1
.}%}
%Similar to \cite{KM1}, for any $r,t>0$ let $c_{r,t}$ to be the contraction bound of the map $g \to g_{-t}gg_t$ in $V_r$ defined by:
%\[{c_{r,t}} = \mathop {\sup }\limits_{p_1,p_2 \in V_r,p_1 \ne p_2} \frac{{d({g_{ - t}}p_1{g_t},{g_{ - t}}p_2{g_t})}}{{d(p_1,p_2)}}\]
%{\equ{basis1} implies that for any $0<r<r''$:
%\eq{contraction bound}{c_{r,t} \le 2re^{-\lambda_0 t}.}}
%Now define a function:
%\[{f_{{V_r}}}(t) = \nu (\{ h \in P:d(h,\partial {V_r}) < {c_{r,t}} \cdot \diam(V_r) \} ).\]
%Note that $\nu (\partial(V_r))=0$ and $c_{r,t} \to 0$ as $t \to \infty$, thus ${f_{{V_r}}}(t) \to 0$ as $t \to \infty$. Moreover, since the exponential map is $2$-bi-Lipschitz in a small neighborhood of $0$ and $V_r$ is the image of the cube of side length $\frac{r}{16  \sqrt{L}}$ under the exponential map, it's not hard to see that there exist constants $s''>0$ and $K_2 \ge 1$ independent of $F$ such that for any $0<r<s''$ and any $t>0$, if $c_{r,t} \cdot \diam(V_r)<1$ then:
%\eq{local bound}{{f_{{V_r}}}(t) < K_2 \cdot c_{r,t} \cdot \diam(V_r)}
%So, if we define
%\eq{opt r}{r'':= \min\{s,s',s''\}<1/8} 
%then, in view of \equ{Bowen inc}, \equ{contraction bound}, for any $0<r<r''$ and any
%$t>\frac{1}{\lambda_0} \log \frac{1}{8r}$: 
%\eq{main local bound}{{f_{{V_r}}}(t) <4 K_2 \diam(V_r) e^{-\lambda_0t}<2K_2 r e^{-\lambda_0 t}<K_2 e^{-\lambda_0 t}.}

{Let us now define} a {\sl Bowen $(t,r)$-box} in $P$ %we mean 
{to be} a set of the form $g_{-t}V_r \gamma g_t$  for some $\gamma \in P$ {and $t > 0$}. 
Also define
$$S_{r,t}:=\{ \gamma  \in \Lambda_r :{g_{ - t}}{V_r\gamma}{g_t}  \cap {V_r} \ne \varnothing \}.$$
Note that $V_r$ can be covered with at most $\#S_{r,t}$ Bowen $(t,r)$-boxes in $P$. 
The following lemma gives   an upper bound for $\#S_{r,t}$: 
\begin{lem}
\label{covering} For any $0<r < r''$ and any {$t>0 $} 
\[\# S_{r,t}  \le \frac{{\nu ({V_r})}}{{\nu ({g_{ - t}}{V_r}{g_t})}}  \left(1 + \frac{{K_3}e^{-\lambda_0 t}}{{\nu ({V_r})}}\right).\]
\end{lem}

\begin{proof}
Let $0<r < r''$ and {$t>0 $}.
One has:
\[\# S_{r,t}  = \# \{ \gamma  \in \Lambda_r :{g_{ - t}}{V_r\gamma}{g_t}  \subset {V_r}\}  + \# \{ \gamma  \in \Lambda_r :{g_{ - t}}{V_r\gamma}{g_t}  \cap \partial {V_r} \ne \varnothing \}. \]
 Since $V_r$ is a tessellation domain of $P$ relative to $\Lambda_r$, the first term in the above sum is not greater than $\frac{\nu(V_r)}{\nu(g_{-t}V_rg_t)}$, {while in view of \equ{cube d} and \equ{diam}, the second term is not greater than:
%\eq{second}
$${\frac{{\nu (\{ p \in P:{\dist}(p,\partial {V_r}) < \diam({g_{ - t}}{V_r}{g_t})\})}}{{\nu ({g_{ - t}}{V_r}{g_t})}}<\frac{2r{K_3}e^{-\lambda_0t}}{\nu({g_{ - t}}{V_r}{g_t})}<\frac{{K_3}e^{-\lambda_0 t}}{\nu({g_{ - t}}{V_r}{g_t})}.}$$
This finishes the proof.}
% Indeed, suppose that the sets {${\{ {g_{ - t}}{B^P}(r){g_t}q\} _{q \in S}}$} do not cover $B^P(r)$. Then there exists $p \in B^P(r)$ such that $p$ doesn't belong to {$g_{-t}B^P(r)g_tq$} for any $q \in S$. This implies that for any $q \in S$, {$d(g_tpg_{-t},g_tqg_{-t}) \ge r$} which is a contradiction since $S$ was maximal. Thus the conclusion follows.
 \end{proof}

{Now let $U$ be an arbitrary subset of $X$. The next lemma can be used to turn the measure estimate from \S\ref{main Section0} into a covering result.}

\begin{lem}
\label{covering of A^P} For any
{$x \in X$, any $U \subset X$, any $0<r<r''$  and  any $t>0$ we have} 
$${A}^P\Big(t,{\frac{r}{16 \sqrt{L}}},{\sigma _{r/2}}{U}, {1},x\Big) \subset \bigcup_{\substack{\gamma \in S_{r,t}\\V_r \gamma g_t x \subset U}}{  {g_{ - t}}V_r \gamma {g_t}}.$$
\end{lem}

\begin{proof}
%\gr{Let $x \in X$, and take $0<r<r''$ and $t>0.$ Also, let $U \subset X$.} 
For any $\gamma \in {P}$ and any $p_1,p_2 \in V_r$ 
we have:
{
\eq{bd}
{\begin{split}
{\dist}\big(p_1\gamma g_t  x,p_2\gamma g_t x\big) 
 \le {\dist}(p_1,p_2) \le \diam(V_r)<r/2.
\end{split}}
}
{Hence, if $$ {A}^P\Big(t,{\frac{r}{16 \sqrt{L}}},{\sigma _{r/2}}{U},{1},x\Big) \cap g_{-t}V_r \gamma g_t  \ne \varnothing$$  for $\gamma \in \Lambda_r$, 
then  {for some $p\in B^P\big(\frac{r}{16 \sqrt{L}}\big)\subset V_r$  one has $g_tpx \in \sigma _{r/2}{U} \cap V_r\gamma g_tx$}, 
%by {\equ{Bowen inc}}, \equ{Conjugate Bowen distance} 
and in view of {\equ{bd}} and ${\partial _{r/2}}{\sigma _{r/2}}{U} \subset U$, we can conclude that 
{$V_r \gamma g_t x \subset U$}.}
\end{proof}
The next corollary follows immediately from Lemma \ref{covering of A^P}:
\begin{cor}\label{sub count}
For any
{$x \in X$,   $U \subset X$,   $0<r<r''$  and    $t>0$ we have}
$$  \#\{\gamma \in S_{r,t} : V_r \gamma g_t x \subset U \}
 \ge \frac{\nu\left({A}^P(t,{\frac{r}{16 \sqrt{L}}},{\sigma _{r/2}}{U},{1},x)\right)}{\nu({g_{ - t}}V_r {g_t})}.           $$
\end{cor}

{For the proof of  {Theorem \ref{main cov}} we will also need to cover Bowen boxes by small balls. 
% Informally we can think about Bowen $(t,r)$-boxes as translates of exponential images of boxes in $\frak p$ whose maximal (resp.\ minimal) sidelength is $2re^{-\lambda_0 t}$ (resp.\ $2re^{-\lambda_{\max} t}$). 
The next lemma provides a bound for the number of balls of radius $re^{-\lambda_{\max} t}$ needed to cover a Bowen $(t,r)$-box.
\begin{lem}
\label{coveringballs} There exists $K_4 > 0$ such that for any $0<r < r''$ and any {$t>0 $},  any Bowen $(t,r)$-box in $P$ can be covered with at most  $K_4 \frac{{\nu ({g_{ - t}}V_r {g_{t}})}}{{\nu \left({B^P}(r e^{ -  \lambda_{\max} t})\right)}}$
balls in $P$ of radius $re^{- \lambda_{\max}t}$. 
\end{lem}
\begin{proof} Let $B = g_{-t}V_r \gamma g_t$ be a Bowen $(t,r)$-box. 
In view of the Besicovitch covering property of $P$, any covering of $B$ by balls in $P$ of radius $re^{- \lambda_{\max}t}$ has a subcovering of  index uniformly bounded from above by a fixed constant   (the Besicovitch constant of $P$). The union of those balls is contained in the $re^{-\lambda_{\max} t}$-neighborhood of $B$. But since $B$ is a translate of the exponential image of a box in $\frak p$ whose  smallest sidelength is $re^{-\lambda_{\max} t}$, it follows that the measure of  the $re^{-\lambda_{\max} t}$-neighborhood of $B$ is bounded by a uniform constant times $\nu(B)$, and the lemma follows.\end{proof}}

 %Since $(V,\Lambda_r)$is a tessellation, Lemma \ref{covering} implies that $V_r$ can be covered with at most   
 %$\frac{{\nu ({V_r})}}{{\nu ({g_{ - t}}{V_r}{g_t})}}\cdot \left(1 + \frac{{{f_{{V_r}}}(t)}}{{\nu ({V_r})}}\right)$ Bowen $(t,r)$-balls in $P$.
%For any $p,q \in S$, since {$d(g_tpg_{-t},g_tqg_{-t}) \ge r$}, the Bowen \new {$(t,r)$-balls $g_{ - t}{B^P}(r/2){g_t}p$ and $g_{ - t}{B^P}(r/2){g_t}q$} are disjoint. Also, using \equ{basis1} {applied with $r$ replaced with $r/2$} one sees that for any $p \in S \subset B^P(r)$ we have {$g_{ - t}{B^P}(r/2){g_t}p \subset B^P(r+2re^{-tb})$.} Therefore %, the cardinality of $S$ is at most:
%{\eq{rectangle}{\# S \le \frac{{\nu ({B^P}(r+2re^{-tb}))}}{{\nu ({g_{ - t}}{B^P}(r/2){g_t})}}.}}
\ignore{ Vice versa, if $g_tpq^{-1}g_{-t} \in B^P(r)$ then there exists $p_2 \in B^P(r)$ such that  $g_tpq^{-1}g_{-t}=p_2$ which implies that $p=g_{-t}p_2g_tq$. So we conclude that $p \in g_{-t}B^P(r)g_tq$ and  $g_{-t}B^P(r)g_tp$ and $g_{-t}B^P(r)g_tq$ have non-empty intersection and if and only if ${\dist}(g_tp,g_tq)<r$. Thus if $S$ is the maximal subset of $P$ such that ${\dist}(g_tp,g_tq) \ge r$ for any $ p,q \in P$, then ${\{ {g_{ - t}}{B^P}(r){g_t}p\} _{p \in S}}$ will be a covering of $P$ with disjoint Bowen-(t,r) balls in $P$. In particular, \comm{(?)} we can cover $B^P(r)$ with disjoint Bowen-(t,r) balls and it's easy to see that the number of Bowen balls in the cover will be at most:  
\eq{rectangle}{{\frac{{\nu ({B^P}(r))}}{{\nu ({g_{ - t}}{B^P}(r){g_t})}}}}}

{We are now ready to begin  the}
\begin{proof}[Proof of Theorem \ref{main cov}]
{Take $a,b,E', \lambda'$ be as in Proposition \ref{exponential mixing}, $K_3$ as in \equ{cube d}, {$K_4$ as in Lemma \ref{coveringballs}} and $\lambda_0$ as in \equ{b1}.}
%Here is our first  observation:
%Let us define the following set:
%$$E_{V_r}(t,x):=\{p \in V_r:g_tpx {\,\notin U}\}   $$
 %$$
%S'_{r,t,x} := %\{ \gamma \in S_{r,t}: g_{-t}V_r \gamma g_t  {\,\cap\, E_{V_r}(t,x ) = \varnothing} \}
% {\{ \gamma \in S_{r,t}: V_r \gamma g_t x \subset U  \}}.
%$$
%\gr{Note that by using the Besicovitch covering property of $P$, there is ${{K_4}}>0$ only dependent on $L$ such that for any $r>0$ and any $t>0$, the Bowen $(t,r)$-box in $P$ can be covered with at most  ${{K_4}} \frac{{\nu ({g_{ - t}}V_r {g_{t}})}}{{\nu \left({B^P}(r e^{ -  \lambda_{\max} t})\right)}}$
%balls in $P$ of radius $re^{- \lambda_{\max}t}$. 
Fix $U \subset X$   such that $U^c$ is compact, and take $0<r<r_0$, $x\in \partial_{r}{U^c},$ and  $t>a+b \log \frac{1}{r}$. Define for any $k \in \N$
$$E_{V_r}(t,{k},x ):=\big\{p \in V_r:g_{{\ell t}}px {\,\notin U}\,\, {\forall \ell \in \{1,2,\cdots,k     \}}\big\}.   $$
Recall that our goal is to construct a covering of the set ${A}^P\big(t,{\frac{r}{16 \sqrt{L}}},{U^c},{k},x\big)$ {for any $k \in \N$}, which is a subset of $E_{V_r}(t,{k},x)$ in view of \equ{Bowen inc}. Note that 
 for $\gamma \in P$, the Bowen $(t,r)$-box $g_{-t}V_r \gamma g_t $ does not intersect $E_{V_r}(t,{1},x )$ if and only if $V_r \gamma g_t x \subset U$. Combining  Lemma \ref{covering} with Corollary \ref{sub count} {and then with Proposition \ref{exponential mixing}}, we conclude that 
 %
% $$ 
%S'_{r,t,x} := \{ \gamma \in S_{r,t}: g_{-t}V_r \gamma g_t  {\,\cap\, E_{V_r}(t,x ) = \varnothing} \}$$
 ${E_{V_r}(t,{1},x)}$ can be covered with at most
%\eq{Number of balls}
\eq{nrt}{
\begin{aligned}
\# S_{r,t}  &- \# \{ \gamma \in S_{r,t} : V_r \gamma g_t x \subset U\} \\
& \le  \frac{{\nu ({V_r})}}{{\nu ({g_{ - t}}{V_r}{g_t})}}  \left(1 + \frac{{K_3}e^{-\lambda_0 t}}{{\nu ({V_r})}}\right) - \frac{\nu\left({A}^P(t,{\frac{r}{16 \sqrt{L}}},{\sigma _{r/2}}{U}, {1},x)\right)}{\nu({g_{ - t}}V_r {g_t})}   \\
& \le \frac{{\nu ({V_r})}}{{\nu ({g_{ - t}}{V_r}{g_t})}} \cdot \left( {1 + \frac{{{{K_3}}{e^{ - \lambda_0 t}} - %\eta(r,t)
{\nu\left(B^P\big(\frac{r}{16 \sqrt{L}} \big)\right)\mu ({\sigma _{r}}U)  + E'{e^{ - \lambda 't}}}
}}{{\nu ({V_r})}}} \right)\\&=:N(r,t) 
\end{aligned}}
Bowen $(t,r)$-boxes in $P$.

Now let $g_{-t}V_r \gamma g_t$ be one of the Bowen $(t,r)$-boxes in the above cover which has non-empty intersection with ${E_{V_r}(t,{1},x)}$. %Also let $p_0$ be the center of $B'$. 
Take any $q={g_{ - t}}{h} \gamma{g_t} \in g_{-t}V_r \gamma g_t$; then
%\eq{induction}
${{g_t}qx = {h} \gamma g_t x}$, hence
%So by \equ{induction}
%\eq{induction1}
${\left\{ {{g_t}qx:q \in g_{-t}V_r \gamma g_t\,} \right\} = \left\{ {h} \gamma g_t x:{h} \in V_r \right\}.}$
%Thus, if we set $E_{\gamma}:=\{q \in g_{-t}V_r \gamma g_t: g_{2t}qx \in U^c\}$, then we have by \equ{induction1}:
Consequently, 
\eq{induction2}{%E_{\gamma}
\{q \in g_{-t}V_r \gamma g_t: g_{2t}qx \notin U\}=g_{-t}{E_{V_r}(t,{1},x)} \gamma g_t.}
Note that since $\diam(V_r)<r$ and $g_{-t}V_r \gamma g_t \cap {E_{V_r}(t,{1},x)}  \ne \varnothing $, we have $\gamma g_t x \in \partial_{r} U^c$. 
%and by using \equ{Conjugate Bowen distance} with $p,q$ replaced by $p$ and $p_0$ we have {$g_tp_0x
% \in \partial_{2r}U^c$.} 
Hence, %in view of \equ{induction2}, 
by {going through} the same procedure, this time using $\gamma g_t x$ in place of $x$, we can cover the set %$E_{\gamma}$ 
in the left hand side of  \equ{induction2} with at most $N(r,t)$ Bowen $(2t,r)$-boxes in $P$. Therefore, we conclude that the set ${E_{V_r}(t,{2},x)}$ can be covered with at most $N(r,t)^2$ Bowen $(2t,r)$-boxes in $P$. By doing this procedure inductively, we can see that for any $k \in \N$, the set ${E_{V_r}(t,{k},x)}$ can be covered with at most
\ignore{\eq{Bowen}{{\left( \frac{{\nu ({B^P}(r+4re^{-tb}))}}{{\nu ({g_{ - t}}{B^P}(r){g_t})}} \right)^k} \cdot {\left( 1 - \frac{c_1}{c_2} \cdot 2^{-L} \nu {({B^P}(r+4re^{-t \lambda_0}))^{ - 1}}\cdot \mathop {\inf }\limits_{y \in {\partial _r}{U^c}} \nu ({{A}^P}(t,r,{\sigma _{2r}}U,y)))\right)^k}}}$N(r,t)^k$ Bowen $(tk,r)$-boxes in $P$. \ignore{Moreover, {using \equ{eigen value} and the Besicovitch covering property of $P$,
%\comm{(Is it correct that here \equ{eigen value}  has to be used? are there other places where it is used? Shahriar: Yes that's the only place that it is used.)} 
it is easy to see that for some ${{K_4}}>0$ only dependent on $L$, any Bowen $(tk,r)$-box in $P$ can be covered with at most  ${{K_4}} \frac{{\nu ({g_{ - tk}}V_r {g_{tk}})}}{{\nu \left({B^P}(re^{ - k {\lambda_{\max}}t})\right)}}$
 balls in $P$ of radius $re^{-k {\lambda_{\max}}t}$.}}
%\comm{The introduction of the constant ${{K_4}}$ should take place earlier, since it is an ingredient for the needed constants.} 
 Thus, {in view of Lemma \ref{coveringballs},} the set 
${E_{V_r}(t,{k},x)}$ can be covered with at most 
%\eq{final cover}
$${{K_4}} \frac{{\nu ({g_{ - tk}}V_r{g_{tk}})}}{{\nu \big({B^P}(r e^{ - k {{\lambda_{\max}}}t})\big)}}  N(r,t)^k
$$
%\ignore{multiplied by\eq{final cover}{ {\left(1 - \frac{c_1}{c_2}2^{-L} \nu {({B^P}(r+4re^{-tb}))^{ - 1}}\mathop {\inf }\limits_{y \in {\partial _r}{U^c}} \nu ({{A}^P}(t,r,{\sigma _{2r}}U,y)))\right)^k}}}
balls of radius $re^{-k{\lambda_{\max}}k}$ in $P$. 

{Now observe} that for any $r>0$ and any $k \in \N$ one has
{
\eq{compwithkadyrov}{
\left(\frac{{\nu (V_r)}}{{{{\nu ({g_{ - t}}V_r{g_t})}}}}\right)^k = {\frac{{\nu (V_r)}}{{\nu ({g_{ - kt}}V_r{g_{tk}})}}}
\,.}
{Here it is crucially important\footnote{{We note that a similar step in the proof of \cite[Theorem 3.1]{K} uses balls instead of boxes, and the boundary effects make it difficult to justify the corresponding equality.}} that the translates of $V_r$ form a tessellation of $P$.}
Using \equ{Ball measure} and  \equ{compwithkadyrov} we get
$$
\begin{aligned}\frac{{\nu ({g_{ - tk}}V_r{g_{tk}})}}{{\nu \big({B^P}(r{e^{ - k {\lambda_{\max}}t}})\big)}}   {\left( {\frac{{\nu (V_r)}}{{\nu ({g_{ - t}}V_r{g_t})}}} \right)^k} 
%&
 &= \frac{{\nu (V_r)}}{{\nu \big({B^P}(r{e^{ - k {\lambda_{\max}}t})\big)}}} \\ &\le  \frac{c_2 (r/4)^L}{c_1 r^L e^{ - Lk {\lambda_{\max}}t}} = 
 \frac{c_2}{4^Lc_1} {e^{Lk {\lambda_{\max}}t}},
 \end{aligned}
$$
which,  in view of  \equ{Bowen inc}, \equ{Ball measure} and %{Proposition \ref{exponential mixing}}
{the definition \equ{nrt} of $N(r,t)$,}
 implies that   $${A}^P\big(t,{\frac{r}{16 \sqrt{L}}},{U^c}, {k},x\big)\subset{E_{V_r}(t, {k},x)}$$ can be covered with at most
$${\begin{split}
& \  \frac{{{K_4}} c_2}{4^Lc_1} {e^{Lk {\lambda_{\max}}t}} 
\cdot \left(1 + \frac{{K_3} e^{-\lambda_0 t}-
{\nu\left(B^P\big(\frac{r}{16 \sqrt{L}} \big)\right)\mu ({\sigma _{r}}U)  + E'{e^{ - \lambda 't}}}
}{{\nu ({V_r})}}\right)^k \\
 \le & \  \frac{{{K_4}} c_2}{4^Lc_1} {e^{Lk {\lambda_{\max}}t}}  \left(1 + \frac{{K_3}(16 \sqrt{L})^L e^{-\lambda_0 t}}{c_1 r^L} -  \frac{c_1}{c_2 (4  \sqrt{L})^L } \mu (\sigma_r U) + \frac{4^L E' e^{-\lambda ' t}}{c_2 r^L}  \right)^k  \\
 \le &\ {K_0}{e^{Lk {\lambda_{\max}}t}} \left(1 - K_1 \mu (\sigma_r U)+ \frac{K_2 e^{-\lambda_1 t}}{r^L}\right)^k 
\end{split}}$$
balls in $P$ of radius $re^{-k {\lambda_{\max}}t}$, where $${K_0}= \frac{{{K_4}} c_2}{4^Lc_1}, \ {{K_1}=\frac{c_1}{c_2 (4  \sqrt{L})^L },\  K_2=\frac{K_3 (16 \sqrt{L})^L}{c_1} +\frac{4^L E'}{c_2},}$$ and $\lambda_1 = \min(\lambda_0, \lambda ')$. 
 } \end{proof}

\section{(EEP) for the group $P$ {as in \equ{subgroup mix}}} \label{eepp}

In the last two sections of the paper  we %are going to 
prove Theorem \ref{main theorem
3}. Namely we %will 
fix two positive integers $m, n$, take   $X  = \ggm$  as in \equ{slmn} %, i.e.\ the space of unimodular lattices in
%${\mathbb{R}}^{m+n}$, 
and consider $F = \{g_t\} = g_t^{\vr,\vs}$ as in  
\equ{generalgt}, where $\vr$ and %= (
%{i_k}:1 \le i \le m)$ 
$\vs$ %= ({j_\ell}:1 \le j \le n)$  
are as in \equ{simplex}. We also define
\eq{defalpha}{{\alpha}=\min\{%{m}
i_1,\dots,%{m}
i_m,%{n}
j_1,\dots,%{n}
j_n\}.}
%\comm{($\omega \le 1$ anyway, so maybe it is not needed at all?) \bro{I agree, we can use 1 instead of $\omega$ so I omitted it}.}\\
{Let us denote $m+n$ by ${d}$. In what follows, constants $C_1,C_2,\dots$ %and $E_1,E_2,\dots$ 
will only depend on $m$ and $n$.} 

Our goal in this section is to prove that $P$ as in 
\equ{subgroup mix} satisfies (EEP) with respect to the $F^+$-action on $X$.  Note that, unless $\vr = \vm$ and $\vs = \vn$, $P$ is a proper subgroup of the expanding horospherical subgroup relative to $g_1$, hence Theorem \ref{thmheep} is not applicable. 
%On the other hand, $P$ is   the expanding horospherical subgroup relative to another element of $G$. Namely, let us  write $g_t=a_tb_t$, where
%$$a_t=\diag(e^{(i_1-{\alpha}/2m)t},\dots,e^{(i_m-{\alpha}/2m)t},e^{-(j_1-{\alpha}/2n)t},\dots,e^{-(j_n-{\alpha}/2n)t})$$
%and
%$$b_t=\diag(e^{{\alpha} t/2m},\dots,e^{{\alpha} t/2m},e^{-{\alpha} t/2n},\dots,e^{-{\alpha} t/2n}).$$
In \cite{KM4}, the  proof of effective equidistribution of $g_t$-translates of orbits of $P$ %presented 
used the observation that $P$ is an expanding horospherical subgroup relative to another element of $G$.
%$b_1$.
We are going to work out an explicit estimate for the constant in \cite[Theorem 1.3]{KM4}; namely, establish
\begin{thm}\label{1.5} Let $P$ be as in 
\equ{subgroup mix}, 
$F = \{g_t\}$ as in  
\equ{generalgt}, and $X$ as in \equ{slmn}. Then $P$ satisfies {\rm (EEP)} relative to the $F^+$-action on $X$.
\ignore{There exist $\lambda, \alpha, C_1> 0$ such that
for any $f\in \cic(H)$ with $\supp f \subset B:=B(1)$, %for any 
$\psi\in
\cic(X)$
 and for any compact  $L\subset X$ there
exist $\lambda_L>0$, $\tilde C:=\tilde C(f,\psi,\rho)$ and $T:=T(L)=\frac{1}{\lambda} \max \{ \frac{\lambda}{\alpha } \cdot \log \frac{1}{{C_1  \cdot {r_0(L)} }},\log \frac{2}{{{r_3}}}\}$ such that  for
all $x \in L$ and all $\vt\in\fa^+$ with $\lfloor
 \vt\rfloor >T$ we have:
\eq{estimate 2}{
\left|%\int_H
%f(h)\psi(g_\vt hz)\,d\nu(h) 
I_{f,\psi}(g_\vt,z) - \int_H
f\,%d\nu
 \int_X\psi\,%d\bar\mu
\right|\le \tilde C e^{-\lambda
\lfloor
 \vt\rfloor}\,.
}
where
$$\tilde C (f,\psi,\rho) \ll \max (\sup|\psi| ,\|\psi\|_{\ell{,2}} ,\|\psi\|_{\Lip}) \cdot \| f\|_{C^\ell} $$
and
\eq{lambda}{{\lambda_L} := \inf \{ \left\| {g\vw} \right\|:\pi (g) \in L,\,\vw \in { \bigwedge ^j}({\Z^k}) \setminus \{ 0\}   \,,\,j = 1, \cdots ,k - 1\} .}}
\end{thm}
 {Recall that $X$ can be identified with the space of unimodular lattices in
${\mathbb{R}}^{{d}}$  via $g\Gamma \mapsto g\Z^{{d}}$.} 
 It will be useful to relate the injectivity radius $r_0(x)$ of an element $x = g\Gamma\in X$ %(a lattice in $\R^{{d}}$) 
 with the function 
{ \eq{defdelta}{\delta(g\Gamma) := \inf_{\vv\in\Z^{{d}}\nz}\|g\vv\|.}}Here $\|\cdot\|$ stands for some norm on $\R^{{d}}$; the implicit constants in the statements below will depend on the choice of the norm.
%
%$$
%\delta(x) = \inf_{\vv\in x\nz}\|\vv\|$$
%introduced in  \equ{defdelta}. 
{
\begin{lem}\label{injectivity1} There exist ${C_1},{C_2} > 0$ such that for any $x\in X$ one has
$$
{C_1} \delta(x)^{{d}} \le r_0(x)\le {C_2} \delta(x)^{ \frac {{d}}{{d}-1}}.
$$
\end{lem}
\begin{proof} The lower estimate can be found in \cite[Proposition 3.5]{KM4} or \cite[Lemma 3.6]{BK}. To prove the upper estimate, {take $\|\cdot\|$ to be the Euclidean norm,} suppose $\delta(x) = \vre$, {and let $\lambda_1,\dots,\lambda_{{d}}$ be the successive minima of the lattice $x$. Let $\vv_1,\dots,\vv_{{d}}$ be vectors realizing the first and the last minimum of $x$ respectively, and  take $g$ to be an element of $G$ which fixes $\vv_1,\dots,\vv_{{d}-1}$ and sends $\vv_{{d}}$ to $\vv_{{d}} + \vv_1$. Then $gx = x$, and, since   $\|\vv_1\| = \vre$ and  $\|\vv_{{d}}\| \ge  \vre^{-\frac{1}{{d}-1}}$,  it follows that} 
%By lattice basis reduction procedure, one can find a basis for $x$ containing the shortest vector that is nearly orthogonal. More precisely, if we represent this basis by the matrix $g$, in view of Iwasava decomposition, one can find $k,a,n$ such that $g=kan$ and $k \in K,a \in A,n \in N_{1/2}$, where
%$$K=SO_{{k}}(\R),$$
%$$A=\{\diag(a_1,\cdots,a_{{k}}):a_1 \cdots a_{{k}}=1, a_i>0, \forall i=1,2,\cdots,m+n\},$$
%and
%$$N_{1/2}=\{(n_{ij}) \in \SL_{{k}}(\R):n_{ii}=1, n_{ij}=0 \,\, for \, \, i>j, and \,\, |n_{ij}| \le 1/2 \,\, for \,\, j>i \}.$$
%So, it's not hard to see that it suffices to prove the upper estimate for the case $g=a$; that is $g \in A.$ Suppose $g=\diag(\vre,a_2, \cdots ,a_{{k}})$. Since $det(g)=1$, there exists $2 \le j_1 \le m+n$ such that $a_{j_1} \ge (\frac{1}{\vre})^\frac{1}{m+n-1}.$ Hence, if we define $h=(h_{ij}) \in \SL_{{k}}(\R)$ by:
%\[{h_{ij}} := \left\{ \begin{gathered}
%  1\,\,\,\,\,\,\,\,if\,i = j \hfill \\
%  \frac{\varepsilon }{{{a_{{j_1}}}}}\,\,\,if\,i = 1\,and\,j = {j_1} \hfill \\
%  0\,\,\,\,\,\,otherwise \hfill \\ 
%\end{gathered} \right. \]
%then, we can see that $g^{-1}hg \in \SL_{{k}}(\Z)$ and:
$${\dist}(g,e) \ll \|g - I\|_{op} \ll %\frac{\vre}{a_{j_1}} \ll 
\vre ^{1+\frac{1}{{d}-1}} = \delta(x)^{\frac{{d}}{{d}-1}},$$
{(here and hereafter $\|\cdot\|_{op}$   refers to the %sup norm on $\R^{n+m}$
operator norm as a linear transformation of $\R^{{d}}$)}, finishing the proof.
\end{proof}}
The next ingredient of the proof is quantitative nondivergence of translates of $P$-orbits. Let us denote by $\fa^+$
the set of ${d}$-tuples $\vt = (t_1,\dots,t_{{{d}}})\in \br^{{d}}$
such that
$$
t_1,\dots,t_{{d}} > 0%,\ t_{m+1},\dots,t_{d} < 0,
\quad \text{and}\quad 
\sum_{i = 1}^m t_i =\sum_{j = 1}^{n} t_{m+j} \,,
$$ 
and for $\vt\in\fa_+$ define
$$g_\vt := \diag(e^{t_1}, \ldots,
e^{ t_m}, e^{-t_{m+1}}, \ldots,
e^{-t_{{d}}})\in G$$
and 
$$
\lfloor \vt\rfloor  := \min_{i = 1,\dots,{d}} t_i. 
$$ 
 {The following statement about quantitative non-divergence of $g_\vt$-translates of $P$ orbits in $X$ was proved in \cite[Corollary 3.4]{KM4}: for any 
compact $L\subset X$ and any ball $B\subset P$  centered at 
$e$ there exist constants 
$T = T(B,L)$ and $C = C(B,L)$ such that for every
$0 < \vre < 1$, any $x\in L$ and any
$\vt\in\fa^+$ with $\lfloor \vt\rfloor \ge T$ one has
$$
\nu\left(\big\{{h} \in B:   \delta(g_{\vt}{h} x) < \vre
\big\}\right)
\le C%\tilde C_k\cdot %e^{-\frac{\beta}{mn(k-1)}\lfloor
%\vt\rfloor} 
\vre^{\frac{1}{mn({d}-1)}} \nu(B)\,.%\tag 3.3
$$
For our purposes we need an effective version:
% is an explicit form of 
\begin{prop}\label{qn}
%Denote \eq{tildeb}{\tilde B := B^P(2).} 
%be a ball centered at 
%$e\subset H$.  
%Then 
%there exists 
%$T = T(B,L)$ such that 
There exist constants $C_3, C_4, C_5$ such that for every
$0 < \vre < 1$, any $x\in X$ and any
$\,\vt\in\fa^+$ with ${\lfloor \vt\rfloor \ge  {C_3 + C_4 \log\frac{1}{r_0({x})  }}}$ it holds that
\eq{nondiv in}
{\nu\big( \{{h} \in B^P(2) : \delta(g_{\vt}{h} x) < \vre
%\notin K_{
%e^{-\beta\lfloor
%\vt\rfloor}
%\vre}
\}\big)
\le C_5 %\tilde C_k\cdot %e^{-\frac{\beta}{mn(k-1)}\lfloor
%\vt\rfloor} 
\vre^{\frac{1}{mn({{d}}-1)}} %\nu(\tilde B)
\,.%\tag 3.3
}
\end{prop}}
%{
{\begin{proof} 
According to \cite [Theorem 3.1] {KM4}, which is  a special case of general quantitative non-divergence result  \cite[Theorem 6.2]{BKM},  there exists an explicit constant 
$C_6 > 0$, depending only on $m$ and $n$, such that for every ball $B\subset P$, any $x = g\Z^{{d}}\in X$, any
$\vt\in\fa^+$ and any  $0 < \vre < 1$ not greater than
\eq{defrho}{
c :=  \inf_{\substack{\vw \in {\bigwedge}^{k}( \Z^{{d}}) \ssm \{0\} \\ {k}=1,\dots,{{d}}-1}}\ { \sup_{{h}\in B}\  \|g_{\vt}{h} g \vw\|,}}  it holds that 
  $${\nu\big( \{{h} \in B : \delta(g_{\vt}{h} x) < \vre
%\notin K_{
%e^{-\beta\lfloor
%\vt\rfloor}
%\vre}
\}\big)
\le C_6 %\tilde C_k\cdot %e^{-\frac{\beta}{mn(k-1)}\lfloor
%\vt\rfloor} 
(\vre/{c})^{\frac{1}{mn({d}-1)}}}\nu(B).$$ 
  On the other hand,  \cite[Lemma 3.2]{KM4} asserts the existence of $C_7 > 0$ and, for each ball $B\subset P$, a constant $C_B$ such that  for  any
$\vt\in\fa^+$ and any $\vw \in {\bigwedge}^{k}( \R^{{d}})$, $ {k}=1,\dots,{{d}}-1,$ one has 
$$\sup_{{h} \in B} \big\| g_{\vt}{h}\vw \big\| \ge 
C_Be^{C_7\lfloor \vt\rfloor}\|\vw\|.$$
Also, by Minkowski's Lemma there exists $C_8 > 0$ such that
$$\inf_{{\vw \in {\bigwedge}^{k}( \Z^{{d}}) \ssm \{0\}} }{   \|g \vw\|} \ge C_8\delta(x)^{{k}}.$$
Therefore $c$ as in \equ{defrho} is not less than \eq{boundonrho}{C_Be^{C_7\lfloor \vt\rfloor} C_8\delta(x)^{{{d}}-1} \ge C_8C_Be^{C_7\lfloor \vt\rfloor} \left(\frac{r_0(x)}{C_2}\right)^{\frac{({{d}}-1)^2}{{d}}}}
(the last inequality holds in view of Lemma \ref{injectivity1}). Now take $B = B^P(2)$ and choose $\vt$ so that the right hand side of \equ{boundonrho} is not less than $1$; equivalently, such that 
$$
\lfloor \vt\rfloor \ge \frac1{C_7} \log \frac{C_2^{\frac{({{d}}-1)^2}{{d}}}}{C_8C_{B^P(2)}}  + \frac{({{d}}-1)^2}{{{d}}C_7} \log\frac1{r_0(x)}.
$$
Then \equ{nondiv in} will hold for any $0 < \vre < 1$ with {$C_5 = {C_6} \cdot \nu\big(B^P(2)\big).$} 
%\comm{Why did you put $\rho^{\frac{1}{mn({k}-1)}}$ here? first of all $\rho$ is now $c$, and also $\vt$ is chosen so that $c \ge 1$ so it can simply be omitted, or am I missing anything? \bro{I changed it back to the previous version. $\rho$ is defined in previous sections as well and it was less than or equal to 1. That's why I got confused. Now that you changed it to $c$, it makes sense to me. But, $c$ is also used in the proof of Lemma 4.2 on page 11. Should we change it to something else, for example $c'$?}} 
\end{proof}}

\begin{proof}[Proof of Theorem \ref{1.5}]
  %For $t>0$, let $g_t$ be defined as in \equ{generalgt}. 
%We can 
Write $g_t=a_tb_t$, where
$$a_t=\diag(e^{(i_1-\frac{\alpha}{2m})t},\dots,e^{(i_m-\frac{\alpha}{2m}t},e^{(-j_1+\frac{\alpha}{2n})t},\dots,e^{(-j_n+\frac{\alpha}{2n})t})$$
and
$$b_t=\diag(e^{{\alpha} t/2m},\dots,e^{{\alpha} t/2m},e^{-{\alpha} t/2n},\dots,e^{-{\alpha} t/2n}),$$
{where $\alpha$ is as in \equ{defalpha}.}
Suppose we are given  %compact subset $L$, any 
$f\in \cic({P})$ with $\supp f \subset B^{P}(1)$,  %for any 
$\psi\in
C^\infty_2(X)$ with $\int_X\psi\,d\mu = 0$, % and for any compact  $L\subset X$ there
%exist ${{K_4}}>0$ and   $\tilde C=\tilde C(f,\psi,\rho)>0$ such that for all
 and $x \in X$. % and large enough $t$, its magnitude to be determined later.  
 Put $r = e^{-\frac{\beta {\alpha}}{2}t}$, where 
$\beta$ is to be specified later, and, again using \cite[Lemma 2.2(a)]{KM4}, take a non-negative function $\theta$ supported on $B^{P}(r)$ such that
{\equ{theta} holds.}  
% {\int_{P}
%\theta\,d\nu 
%= 1\text{ and   }\|\theta\|_{\ell{,2}} 
%\ll   r^{\ell + k/2},} where $k = \dim {P}$. %= N - \tilde k$.
%\tilde c  
Since $\nu$ is translation-invariant, one can 
 write
$$
\aligned
%\int_{P}
%f(h)\psi(g_t hx)\,d\nu(h) 
I_{f,\psi}(g_t,x) &= 
%\frac1{J_\vu}
\int_{P}
%{\tilde B}
%\tilde 
f(h)\psi(g_t {h}x)\,d\nu({h})
%f(h)\psi(g_t  hg_\vu x)\,d\nu(h)
\int_{P}
\theta(y)\,d\nu(y)\\&= %\frac1{J_\vu}
\int_{P}
%{\tilde B}
\int_{P}
%\tilde 
f\big(a_{-t}y a_t {h}\big)\theta(y)\psi\big(a_t b_t
a_{-t}ya_t{h}x \big)\,d\nu(y)\,d\nu({h})\\&=\int_{P}
%{\tilde B}
\int_{P}
%\tilde 
f\big(a_{-t}y a_t {h} \big)\theta(y)\psi\big(b_t y 
a_t {h}x \big)\,d\nu(y)\,d\nu({h})\,.
\endaligned
 $$
{Note that \eq{contr}{\min \left (i_1 - \frac{\alpha}{2m},\cdots,i_m-\frac{\alpha}{2m},j_1-\frac{\alpha}{2n},\cdots, j_n- \frac{\alpha}{2n} \right) \ge {\alpha}/2,}
%Note that the map $h \to a_{-t}ha_t$ is a
%contracting
%automorphism of ${P}$, in fact, we have 
therefore} 
$$\dist\big(e,a_{-t}{h}a_t \big)
\le e^{ -{\alpha} t}\dist(e,{h})
$$
for any ${h} \in {P}$. Also, as long as $\theta(y)\ne 0$, the supports of all  functions of the form  ${h} \mapsto
f\big(a_{-t}y a_t {h}\big)$ are contained in  $$B^{P}(1 +
e^{-( {\alpha}+\frac{\beta {\alpha}}{2})t}) \subset {\tilde B }:= B^{P}(2) \,.$$  
\ignore{For $\varepsilon >0$ define:
\eq{cusp}{U(\vre) := \pi\big(\big\{ g \in G \bigm| \| g
\vv \| \ge  {\varepsilon} \quad \forall\, \vv \in
\bz^{d}\nz\big\}\big)\,.}
By \cite[Proposition 3.5]{KM4}, there exists $c_3>0$ such that:
\eq{inj eq}{r_0(U(\vre)) \ge c_3 \cdot \vre ^{{{d}}}.}}
Define
\eq{inj1 eq}{
\vre : = \left(\tfrac {2}{{C_1}}e^{-\frac{\beta {\alpha}}{2} t} \right)^{1/{{d}}}\,,
}
where ${C_1}$ is as in Lemma \ref{injectivity1}, and let
$$A(x,t) := \big\{{h} \in \tilde B\mid \delta(a_t {h} x) < \vre
% \notin U(
%\vre )
\big\}\,.$$ 
So, in view of %Lemma 3.2, Corollary 3.3 and Corollary 3.4 in \cite{KM4}, it's easy to see that one can find constants $\alpha_1, c_3>0$ independent of $t$, $F^+$ and $x$ such that 
{\equ{contr} and} Proposition \ref{qn}, for any $x\in X$ and any  
{\eq{bigt}{t \ge  \frac2\alpha\left({C_3 + C_4 \log\frac{1}{r_0({x})  }}\right)}}
%$$t \  {\ggplus \ \log\frac{1}{r_0({x}) }}$$ 
one has
%\eq{measure in}
$${\nu\big(A(x,t)\big) %\ll 
%\tilde C_k\cdot %e^{-\frac{\beta}{mn({{d}}-1)}\lfloor
%\vt\rfloor} 
{\,\le C_5} {\vre} ^{\frac{1}{mn({{d}}-1)}}%\tilde C_k\cdot %e^{-\frac{\beta}{mn(k-1)}\lfloor
%\vt\rfloor} 
%\vre^{\frac{1}{mn({{d}}-1)}} %\nu(\tilde B)
\,.%\tag 3.3
}$$

\ignore{for any $t > T_1(x):= \frac{2}{\alpha_1 \cdot {\alpha}}\log \frac{1}{c_3 \lambda_x}$ we have
\eq{measure in}{\nu(A) \ll 
%\tilde C_k\cdot %e^{-\frac{\beta}{mn({{d}}-1)}\lfloor
%\vt\rfloor} 
{\vre} ^{\frac{1}{mn({{d}}-1)}} \nu(\tilde B)\,%\tag 4.6
\ll {\vre} ^{\frac{1}{mn({{d}}-1)}}}
where
$$\lambda_x:=\inf\big\{\|g\vw\|
\bigm|
\pi(g)=x,\ \vw \in
\bigwedge^j(\bz^{{{d}}})\nz, j = 1,\dots,{{d}}-1\big\}.$$
If we represent the norm of shortest vector in $x$ as $l_0$, we can conclude from Minkowski Lemma:
\eq{Min}{\lambda_x \ge (\frac{l_0}{\sqrt{m+n}})^{m+n}.}
Also, it's not hard to show for some constant $c_4,\alpha_2>0$ independent of $x$:
\eq{inj lattice}{r_0(x) \le c_4 \cdot {l_0}^{\alpha_2}.}
So, by combining \equ{Min} and \equ{inj lattice} we conclude that:
$$\lambda_x \ge \frac{r_0(x)^{(m+n)/ \alpha_2}}{c_5} \ge \frac{r_0(B(x,2))^{(m+n)/ \alpha_2}}{c_5}$$
where $c_5=c_4^{(m+n)/\alpha} \cdot \sqrt{m+n}^{m+n}.$ Therefore, the inequality \equ{measure in} holds for all $t>T_2(x):= a_1 +b \log \frac{1}{r_0(B(x,2))} > T_1(x)$, where $a_1=\frac{2}{\alpha_1 \cdot {\alpha}} \cdot \log \frac{c_5}{c_3}$ and $b_1=\frac{2(m+n)}{\alpha_1 \cdot \alpha_2\cdot {\alpha}}$.} Hence, %whenever
%$t> T_2(x)$ 
%$t \  {\ggplus \ \log\frac{1}{r_0({x}) }}$, 
{assuming \equ{bigt}}, the absolute value of 
$$
\int_{A(x,t)}\int_{P}
%\tilde 
f\big(a_{-t}y a_t {h} \big)\theta(y)\psi\big(b_t y 
a_t {h}x \big)\,d\nu(y)\,d\nu({h})%\right|
$$
is %not greater than 
$$ 
\ll %\tilde C_k\cdot %e^{-\frac{\beta}{mn(k-1)}\lfloor
%\vt\rfloor} 
{\vre}^{\frac{1}{mn({{d}}-1)}} \nu(\tilde B)\sup|f|\sup|\psi|\int_{P}
\theta\, d\nu
\ll %2\tilde C_k\cdot %e^{-\frac{\beta}{mn(k-1)}\lfloor
%\vt\rfloor} 
\sup|f|\sup|\psi| \cdot e^{-\frac{{\beta {\alpha} }}{2 mn{{d}}({{d}}-1)}t}\,.
%\endaligned\tag 4.5
 $$
 Next, let us assume that ${h} \in \tilde B\ssm A(x,t)$. We are going to apply Theorem \ref{th1} with $b_t$ in place of $g_t$, $r = e^{-\frac{\beta {\alpha}}{2} t}$, $%\tilde z \df 
a_t {h}x $ % \in K_\vre$ 
in place of $x$ and $$f_{{h}}(y)
:=  f\big(a_{-t}y a_t {h}\big)\theta(y)$$ in place of $f$. It is clear that $\supp\, f_{{h}} \subset B^{P}(r)$ for any ${h}$, i.e.\ condition (i) of Theorem~\ref{th1} is satisfied. Since $\delta(a_t {h}x) < \vre$ 
%  \in U(\vre)$ 
whenever 
${h} \notin A(x,t)$, condition (ii) is satisfied in view of 
%\equ{inj eq} 
Lemma \ref{injectivity1} and \equ{inj1 eq}. So we only need to require that $e^{-\frac{\beta {\alpha}}{2} t}$ is less than $\rho /2$.
% which leads to assuming that 
%\eq{non div es}{t > T_3:=\frac{2}{\beta {\alpha}} \log \frac{2}{\rho}.}
Also, in view of \cite[Lemma 2.2(b)]{KM4} and \equ{theta}, {for any $\ell\in\Z_+$} we have $$\|f_{{h}}\|_{\ell{,2}} \ll \| f\|_{C^\ell} \|\theta\|_{\ell{,2}} 
\ll   
%\tilde c  
%e^{\frac{(\ell + {mn}/2) {\beta {\alpha}}}{2} t}
e^{\left(\ell + {\frac{mn}2}\right)\frac{\beta{\alpha}}{2}t } \| f\|_{C^\ell}.$$
This way, by using Theorem \ref{th1} we get, {for some $\gamma > 0$ and $\ell\in\Z_+$,}
$$
\aligned
\left|\int_{\tilde B \ssm A(x,t)}\int_{P}
%\tilde 
f\big(a_{-t}y a_t {h} \big ) \theta(y)\psi\big(b_t y 
a_t {h}x\big)\,d\nu(y)\,d\nu({h})\right| \le 
\int_{\tilde B \ssm A(x,t)}\left|I_{f_{{h}},\psi}(b_t,a_t {h} x)
%\int_{P}\tilde 
%f(y)\psi\big(g_t y\tilde x \big)\,d\nu(y)
\right|\,d\nu({h})\\ \ll 
%E\cdot
%\|\psi\|_{\ell{,2}}
%r\big(\pi_x(\supp f)\big)^{-s} \tilde c  r^{-(\ell+mn/2)}\|f\|_{\ell{,2}}
%e^{-\frac{\gamma {\alpha}}{2} t}
 \max \big( \|\psi\|_{{C^1}},  \|\psi\|_{\ell{,2}} \big ) \left(e^{-\frac{\beta {\alpha}}{2} t} %\cdot d(\psi) 
%\int_{P} |f_h|\,%d\nu(y)\,%d\nu 
{\|f_{{h}}\|_1} +
%\tilde E \cdot
%\|\psi\|_{\ell{,2}}
e^{\left(\ell + {\frac{{{d}^2}-1-mn}2}\right)\frac{\beta{\alpha}}{2}t }\|f_{{h}}\|_{\ell{,2}}%\|\psi\|_{\ell{,2}} 
\,%e^{-\frac{\gamma {\alpha}}{2}t} \right)
{e^{-\gamma  t}} \nu(\tilde B)\right)
\\ \ll  { \max \big (\|\psi\|_{{C^1}},  \|\psi\|_{\ell{,2}} \big )}
\left(\sup|f|\cdot e^{-\frac{\beta{\alpha}}{2}  t} %\cdot d(\psi) 
 +
\| f\|_{C^\ell}  \cdot 
e^{{ (2\ell + \frac{{{d}^2}-1}2)\frac{\beta {\alpha}}{2}t - \gamma t}  } %{\frac{{\alpha}}{2}}
 \right).\endaligned$$ 
%Define $a:=\max \big (a_1,T_3)$. Then b
%\comm{Shahriar, I didn't understand why you had $e^{-\frac{\gamma {\alpha}}{2}t} $ instead of $e^{-\gamma  t}$, can you check it?} \bro{OK you are right, I think something went wrong. Everything is correct now.}

By combining the two estimates above, we get that, as long as $t \  {\ggplus \ \log\frac{1}{r_0({x}) }}$,
%for any $t \ge T_4(x):=a+b \log \frac{1}{r_0(B(x,2))}\ge \max \big (T_2(x),T_3 \big )$ 
%the absolute value
%of 
%\eq{}
$${\begin{split}
\left|
I_{f,\psi}(g_t,x)\right|& 
\ll \sup|f|\sup|\psi| e^{-\frac{{\beta {\alpha}}}{2mn{{d}}({{d}}-1)}t} \\
& +  
{ { \max \big ( \|\psi\|_{{C^1}},  \|\psi\|_{\ell{,2}} \big )}
\left(\sup|f|\cdot e^{-\frac{\beta{\alpha}}{2}  t} %\cdot d(\psi) 
 +
\| f\|_{C^\ell}  \cdot 
e^{{ (2\ell + \frac{{{d}^2}-1}2)\frac{\beta {\alpha}}{2}t - \gamma t}  } %{\frac{{\alpha}}{2}}
 \right)}
%\max \big ( \|\psi\|_{\ell{,2}} ,\|\psi\|_{\Lip} \big )
%\left(\sup|f| e^{-\frac{\beta {\alpha}}{2} t} %\cdot d(\psi) 
% + \| f\|_{C^\ell}   e^{-\left ( \gamma - (2\ell + N/2)\beta \right) {\frac{{\alpha}}{2}}
%t} \right)
  \\
& \ll {
\max\left (
{\left\| \psi  \right\|_{C^1}, \left\| \psi  \right\|_{\ell{,2}} }
%,{\left\| \psi  \right\|_{Lip}}
\right)  {\left\| f \right\|_{{C^\ell }}} 
\cdot \max\left (e^{-\frac{{\beta {\alpha}}}{2mn{{d}}({{d}}-1)}t}, e^{ -\left(\gamma  - (2\ell + \frac{{{d}^2}-1}2)\frac{\beta {\alpha}}{2} \right)t} \right)}.
%\mathop {\sup }\limits_X |\psi |,
%{\left\| \psi  \right\|_{C^1}, \left\| \psi  \right\|_{\ell{,2}} }
%,{\left\| \psi  \right\|_{Lip}}
%)
% \ll  \sup|f| \cdot {\|\psi\|_{C^1}} e^{-\frac{{\beta {\alpha}}}{2mn({{d}})({{d}}-1)}t} \\
%& +\| f\|_{C^\ell}  \cdot {\max \big ( \|\psi\|_{C^1}, \|\psi\|_{\ell{,2}}\big)} e^{-\left ( \gamma - (2\ell + N/%2)\beta \right) {\frac{{\alpha}}{2}}
%t}.
\end{split}}$$
%An elementary computation shows that  c
{Choosing $\beta$  equalizing the
two exponents
%in the first and the third summands 
above, that is %will produce 
$$\beta  =
 \frac {2 \gamma/\alpha} {\frac{1}{mn{{d}}({{d}}-1)}  + 2\ell+\frac{{{d}^2}-1}2},$$  
 will satisfy \equ{eep} with $$\lambda = \frac{{\beta {\alpha}}}{2mn{{d}}({{d}}-1)} =  \frac { \gamma} {1 + mn{{d}}({{d}}-1)(2\ell+\frac{{{d}^2}-1}2)},$$ which finishes the proof.} %\comm{I don't understand -- why does it work with $ \lambda = \beta$? \bro{You are right. What you wrote is correct. Something went wrong here too.}}
%$$\beta = \lambda =
% \frac { \gamma} {1 + mn{{d}}({{d}}-1)(2\ell+N/2)}$$  
%Therefore, \equ{eep} is satisfied by choosing $\lambda$ as above,
% and $t \ge T_4(x):=a+b \log \frac{1}{r_0(B(x,2))}$ and this 
\end{proof}

\section{Weighted badly approximable matrices} \label{weighted badly}

%In this section we %are going to 
%prove Theorem \ref{main theorem
%3}. Throughout this section we will fix two positive integers $m, n$. We take 
%$\vr$ %= (
%{i_k}:1 \le i \le m)$ 
%and $\vs$ %= ({j_\ell}:1 \le j \le n)$  
%as in \equ{simplex} and consider $F = \{g_t\}$ as in  
%\equ{generalgt} acting on the space $X  = \ggm$  of unimodular lattices in
%${\mathbb{R}}^{{{d}}}$, where  $G = \SL_{{d}}(\mathbb{R})$ and $\Gamma = \SL_{m +
%n}(\mathbb{Z})$. 

{Now let us recall a connection between \da\ with weights and the action of $F = \{g_t\}$ as in  
\equ{generalgt} on the space $X$.}
It is shown in  \cite[Theorem 6.2]{K2} that $ A \in
{M_{m,n}}(\mathbb{R})$ is $(\vr,\vs)$-badly approximable iff
the orbit $\{ {g_t}{u_A}{\mathbb{Z}^{{k}}}: t  > 0\} $ is
bounded in $X $, where ${u_A} = \left( {\begin{array}{*{20}{c}}
{{I_m}}&A\\
0&{{I_n}}
\end{array}} \right)$. We want to make this equivalence quantitative. Recall that for $\p =
({p_1},\dots,{p_m})$ and  $\q =
({q_1},\dots,{q_n})$ we defined %the $\vr$-quasinorm of $\p$ and the $\vs$-quasinorm of $\vq$ by 
$${\left\| \p \right\|_\vr} =  {\max %_{i }
\left(\left| {{p_1}} \right|^{1/{i_1}},\dots,\left| {{p_m}} \right|^{1/{i_m}}\right)}\text{ and }{\left\| \q \right\|_\vs} =  {\max %_{i }
\left(\left| {{q_1}} \right|^{1/{j_1}},\dots,\left| {{q_n}} \right|^{1/{j_n}}\right)}.$$ 
Now,  for $\p \in \br^{m}$ and $\q \in \br^{n}$, if $\vv=(\p,\q)$ let us define the {\sl $(\vr,\vs)$-quasinorm} ${\left\| \vv \right\|_{\vr,\vs}}$ of $\vv$ by $${\left\| \vv \right\|_{\vr,\vs}}:=\max({\left\| \p \right\|^{1/m}_\vr},{\left\| \q \right\|^{1/n}_\vs}).$$
Then %define the following function on $X$:  
for ${x}  \in X $ let
\[{\delta _{\vr,\vs}}({x} )\mathop  := % \limits^{def} 
\mathop {\inf }\limits_{\vv \in {x} \ssm \{ 0\} }{\left\| \vv \right\|_{\vr,\vs}} ,\] 
and for $\varepsilon> 0 $ let us consider
%we define :
\eq{uepsilonweights}{{U_{\vr,\vs}(\varepsilon) } := \left\{ {{x}  \in X:{\delta _{\vr,\vs}}({x} ) <\varepsilon } \right\}.}
Mahler's Compactness Criterion implies that a subset $K$ of $X$ is relatively compact if and only if the restriction of $\delta _{\vr,\vs}$ to $K$ is bounded away from zero (that is, $K$ is contained in the complement of $U_{\vr,\vs}(\varepsilon)$ for some $\varepsilon> 0 $). 

%that ${\delta _{\vr,\vs}}({x} )$ can be thought of as a (normalized)
%distance from ${x}$ to infinity in $X$.
Note that in the case $\vr = \vm$ and $\vs = \vn$, the $(\vm,\vn)$-quasinorm is simply the sup norm on $\R^{{d}}$,  $\delta_{\vm,\vn}({x}){ = \delta(x)}$,
% is   the norm of a shortest nonzero vector of ${x}$, 
{and $U_{\vm,\vn}(\varepsilon)$ is the same as $ U(\vre)$ defined in \equ{uepsilon}}.  {Also it is easy to check that for arbitrary $\vr, \vs$ and any $x\in X$ one has} 
\eq{boundfordelta}{{\delta ({x} ) \ge \delta _{\vr,\vs}({x} )^{\max(m,n)}.}} Now we can state a quantitative form of \cite[Theorem 6.2]{K2}, {which is also a weighted version of \cite[Lemma 3.1]{BK}}:

\begin{lem}\label{Cusp}
For any $0<c<1$, $A \in \Bad_{\vr,\vs}(c)$ if and only if
\eq{emptyint}{\{ {g_t}{u_A}{\mathbb{Z}^{{d}}}:{\mkern 1mu}
{\mkern 1mu} t \ge 0\}  \cap {U_{\vr,\vs}(\varepsilon) } = \varnothing, } where
$\varepsilon  = {c^{{1}/{{d}}}}$.

\end{lem}
\begin{proof}
First note that ${g_t}{u_A}{\mathbb{Z}^{{d}}}$ consists of vectors
of the form
\[
\begin{pmatrix}
%\left( {\begin{array}{*{20}{c}}
%{\begin{array}{*{20}{l}}
{{e^{{i_1}t}}({{{p}}_1} + {A_1}{\vq})}\\
%{{e^{{i_2}t}}{{\rm{p}}_1} - {e^{{i_2}t}}{A_2}{\rm{q}}}\\
%{...}\\
\vdots\\
{{e^{{i_m}t}}({{{p}}_m} + {A_m}{\vq})}\\
%{{\mkern 1mu} {\mkern 1mu} {\mkern 1mu} {\mkern 1mu} {\mkern 1mu} {\mkern 1mu} {\mkern 1mu} {\mkern 1mu} {\mkern 1mu} {\mkern 1mu}  - 
e^{ - {j_1}t}{q_1}\\
%{{\mkern 1mu} {\mkern 1mu} {\mkern 1mu} {\mkern 1mu} {\mkern 1mu} {\mkern 1mu} {\mkern 1mu} {\mkern 1mu} {\mkern 1mu} {\mkern 1mu}  - 
%{e^{ - {j_2}t}}{q_2}}\\
\vdots\\
%{{\mkern 1mu} {\mkern 1mu} {\mkern 1mu} {\mkern 1mu} {\mkern 1mu}
%{\mkern 1mu} {\mkern 1mu} {\mkern 1mu} {\mkern 1mu} {\mkern 1mu}
%{\mkern 1mu} ...}
%\end{array}}\\
{ {e^{ - {j_n}t}{{{q}}_n}}}
\end{pmatrix},\]
where    $A_1,\dots,A_m$ are the rows of $A$. 
Suppose that
\eq{ba}{{\left\| {A\vq + \p} \right\|_\vr}%^m
 {\left\| \vq \right\|_\vs}
 %^n 
 \ge c} for all $\p \in
{\mathbb{Z}^m}$ and $\vq \in {\mathbb{Z}^n}\ssm \{ 0\}$.
Take an arbitrary $t \ge 0$. 
If  ${\left| {{e^{ - {j_\ell}t}}{q_j}}
\right|^{{1}/{{{	%n
j_k}}}}} \ge \varepsilon^n $ for some $1 \le k \le
n$, it follows that 
$${\left\| g_t\begin{pmatrix}A\vq + \p\\ \vq\end{pmatrix} \right\|_{\vr,\vs} \ge \varepsilon},$$ and we are done. So suppose that ${\left| {{e^{ - {j_k}t}}{q_k}}
\right|^{{1}/{{{ %n
j_k}}}}}  = e^{-t}{\left| {{q_j}}
\right|^{{1}/{{{ %n
j_k}}}}}  < \varepsilon^n $ for all $k$. Then we have  $ {{\left\| {{\vq}}
\right\|}_\vs} < {\varepsilon^n 
} e^{t}$. 
%For any $1 \le i \le m$, we denote the $i$-th
%row of $A$ by $A_i$. 
%Since $A \in \Bad_{\vr,\vs}(c)$ 
In view of \equ{ba}, there exists $1
\le k \le m$ such that
\[\left| {{A_k}\vq + {p_k}} \right|^{{1}/{{{i_k}}}} \ge \frac c%^{1/m}
{\left\| \vq \right\|_s}%^{n/m}
>  \frac {c%^{1/m} e^{-t/m}}
e^{-t}}{\varepsilon^n} %^{n/m}
,\]
hence
\[
{\left| {{e^{{i_k}t}}({A_k}\vq + {p_k})} \right|^{{1}/{{{%m
i_k}}}}}
= {e^{t}}{\left| {{A_k}q + {p_k}} \right|^{{1}/{{{i_k}}}}} \ge
\frac{c}{\varepsilon ^{n}} = \varepsilon ^{m}.
%\frac{\varepsilon^{(m+n)/m}}{\varepsilon ^{n/m}} = \varepsilon
\]
This proves that if $\vq \ne 0$, then ${g_t}{u_A}{\mathbb{Z}^{{d}}} \notin {U_{\vr,\vs}(\varepsilon) }$. And if $\vq = 0$ and $\p\ne 0$, then
$$ {\left\| g_t\begin{pmatrix}A\vq + \p\\ \vq\end{pmatrix} \right\|_{\vr,\vs}} = {\left\| g_t\begin{pmatrix} \p\\ \0 \end{pmatrix} \right\|_{\vr,\vs}} \ge e^{t/m} {\left\| {\p} \right\|_{\vr,\vs}} \ge 1 \ge \varepsilon.$$
So ${g_t}{u_A}{\mathbb{Z}^{{d}}} \notin {U_{\vr,\vs}(\varepsilon) }$ holds in this case as well, and we are done.  
%So
% \[\{ {g_t}{u_A}{\mathbb{Z}^{{d}}}:{\mkern 1mu}
%{\mkern 1mu} t > 0\}  \cap {U_{\vr,\vs}(\varepsilon) } = \varnothing \]

 Vice versa, %suppose that:
assume \equ{emptyint}; that is, suppose that
%\[\{ {g_t}{u_A}{\mathbb{Z}^{{d}}}:{\mkern 1mu}
%{\mkern 1mu} t > 0\}  \cap {U_{\vr,\vs}(\varepsilon) } = \varnothing \] Or
%equivalently,
 for any nonzero $(\p,\vq) \in \mathbb{Z}^{m+n}$ and
$t\ge 0$ we have
\eq{cusp equation}{{\left\| \begin{pmatrix}
%\left( {\begin{array}{*{20}{c}}
%{\begin{array}{*{20}{l}}
{{e^{{i_1}t}}({{{p}}_1} + {A_1}{\vq})}\\
%{{e^{{i_2}t}}{{\rm{p}}_1} - {e^{{i_2}t}}{A_2}{\rm{q}}}\\
%{...}\\
\vdots\\
{{e^{{i_m}t}}({{{p}}_m} + {A_m}{\vq})}\\
%{{\mkern 1mu} {\mkern 1mu} {\mkern 1mu} {\mkern 1mu} {\mkern 1mu} {\mkern 1mu} {\mkern 1mu} {\mkern 1mu} {\mkern 1mu} {\mkern 1mu}  - 
e^{ - {j_1}t}{q_1}\\
%{{\mkern 1mu} {\mkern 1mu} {\mkern 1mu} {\mkern 1mu} {\mkern 1mu} {\mkern 1mu} {\mkern 1mu} {\mkern 1mu} {\mkern 1mu} {\mkern 1mu}  - 
%{e^{ - {j_2}t}}{q_2}}\\
\vdots\\
%{{\mkern 1mu} {\mkern 1mu} {\mkern 1mu} {\mkern 1mu} {\mkern 1mu}
%{\mkern 1mu} {\mkern 1mu} {\mkern 1mu} {\mkern 1mu} {\mkern 1mu}
%{\mkern 1mu} ...}
%\end{array}}\\
{ {e^{ - {j_n}t}{{{q}}_n}}}
\end{pmatrix}\right\|_{\vr,\vs}} \ge \varepsilon }
Fix such $\p$ and $\q$, take an arbitrary $0 < \varepsilon_1 < \varepsilon$, 
and choose $t\ge 0$ so that %. It is easy to see that
\[\left\| \begin{pmatrix}
%\left( {\begin{array}{*{20}{c}}
%{\begin{array}{*{20}{l}}
e^{ - {j_1}t}{q_1}\\
%{{\mkern 1mu} {\mkern 1mu} {\mkern 1mu} {\mkern 1mu} {\mkern 1mu} {\mkern 1mu} {\mkern 1mu} {\mkern 1mu} {\mkern 1mu} {\mkern 1mu}  - 
%{e^{ - {j_2}t}}{q_2}}\\
\vdots\\
%{{\mkern 1mu} {\mkern 1mu} {\mkern 1mu} {\mkern 1mu} {\mkern 1mu}
%{\mkern 1mu} {\mkern 1mu} {\mkern 1mu} {\mkern 1mu} {\mkern 1mu}
%{\mkern 1mu} ...}
%\end{array}}\\
{ {e^{ - {j_n}t}{{{q}}_n}}}
\end{pmatrix}\right\|_{\vs} = {e^{ - t}}{\left\| \vq \right\|_\vs = \varepsilon_1^n}.\]
%Choose $t>0$ so that ${e^{ - t/n}}{\left\| q \right\|_s} = \varepsilon
%$, t
Then by \equ{cusp equation}  for some $1 \le k\le m$ we must have
\[{\left| {{e^{{i_k}t}}({A_k}q + {p_k})} \right|^{1/{i_k}}} = {e^{t}}{\left| {{A_k}q + {p_k}} \right|^{1/{i_k}}} \ge \varepsilon^m .\]
%This implies that:
Consequently
${\left\| {A\vq + \p} \right\|_\vr}%^m
 {\left\| \vq \right\|_\vs}
 %^n 
 \ge  \varepsilon^m
  \varepsilon_1^n
  % {e^{ - t}}.\varepsilon^n {e^t} = {\varepsilon ^{m+n}} = c
 $, which, since $\varepsilon_1 < \varepsilon$ was arbitrary, implies that ${\left\| {A\vq + \p} \right\|_\vr}%^m
 {\left\| \vq \right\|_\vs}
 %^n 
 \ge c$. 
Since $\p$ and $\vq$ were arbitrary, $A \in \Bad_{\vr,\vs}(c)$, which finishes the proof of the lemma.
\end{proof}

%Now we want to find 
{We will also need} a lower %and upper 
bound for the Haar measure of
the {inner $r$-core of} the set ${U_{\vr,\vs}(\varepsilon)}$, {where $0 < \vre < 1$ and   $r$ is small enough.  
%for $varepsilon >0$ 
%as in \equ{uepsilon}. 
%Here we follow the lines of the argument from
The first step is a weighted version of \cite[Proposition 7.1]{KM2}:}
%with the new norm that has
%been defined. %We adopt the ideas used for ordinary badly
%approximable forms in \cite{KM2}.
\ignore{
\begin{defn}
For a function $\Delta$ on $X$, define the tail distribution
function ${\phi _\Delta }$ of $\Delta $ by:
\[{\Phi _\Delta }(z) = \mu (\{ x|\Delta (x) > z\} )\]

\end{defn}
We fix an $m$-tuple $r$ and an $n$-tuple $s$ with positive
components such that
\[\sum\limits_{i = 1}^m {{i_k} = 1 = \sum\limits_{j = 1}^n {{j_\ell}} } \]
We consider the function $\Delta $ on the space of unimodular
lattices of ${R^{{k}}}$ defined by:
\[\Delta ({x} ) = -\log \delta_{\vr,\vs}({x}) = \mathop {\max }\limits_{\vv \in {x} \ssm \{ 0\} } \log (\frac{1}{{{{\left\| \vv \right\|}_{\vr,\vs}}}})\]
Also note that in this case:
\eq{si2}{{\Phi _\Delta }(z) = \mu (U_{e^{-z}})}
}
%Our goal is to prove the following
\begin{prop}\label{weighted}
There exist ${{C_9}},{{C_{10}}} > 0$ depending only on ${d}$ such that for all $0 < \vre < 1$ one has
\eq{si}{
{{C_9}}
%{e^{ - ({d})z}}
\vre^{{d}} \ge \mu\big({U_{\vr,\vs}(\varepsilon) }\big) %{\Phi _\Delta }(z)
 \ge {{C_9}} \vre^{{d}} - {{C_{10}}}\vre^{2{d}}
%C{e^{ - ({d})z}} - C'{e^{ - 2({d})z}}
.}
\end{prop}
%\comm{We have already used the notation $E_1$ on equation (4.6) on page 11. Do you think we should use another letter here or on Section 4?} 
\begin{proof}%[Proof of  Proposition \ref{weighted}] 
%The main tool here is %reduction theory for $\SL_{{d}}(\mathbb{R})/\SL_{{d}}(\mathbb{Z})$ and in particular 
%a generalization of Siegel's summation formula %in 
%\cite{Si1}.
%Recall that for $1 \le k < {d}$, 
%an ordered $k$-tuple $(\vv_1 , \dots,\vv_{k})$ of vectors in a lattice ${x}  \subset {\R^{{d}}}$ is primitive if it is extendable
%to a basis of ${x}$, and d
{For $x\in X$ and $1\le k \le d$,  denote by $P^k({x})$ the set of all
{\sl primitive} (i.e.\ extendable
to a basis of ${x}$) ordered $k$-tuples $(\vv_1 , \dots,\vv_{k})$ of vectors in $x$.
 %such $k$-tuples. 
Then}, given a function $\varphi $ on $\mathbb{R}^{{d}}$, for any $k = 1,\dots,{d}-1$ define a function
$\phid$
%\hat \phi^d$
% \overset{ \ \scriptscriptstyle\wedge_d}{ \phi\ } $ 
on %the space of unimodular
%lattices in $\mathbb{R}^{{k}}$ %defined 
{$X$} 
by $$%\hat {\phi^d'}
%\overset{\mbox{\tiny {\wedge_2}}}{\phi} 
\phid ({x} ) := \sum\nolimits_{(\vv_1,\dots,\vv_k) \in P^k({x})} {\varphi (\vv_1,\dots,\vv_k)}. $$ 
%\begin{thm}\label{Siegel} {\rm (\cite[Theorem 7.3]{KM2}).}
{According to {a generalized Siegel's summation formula}  \cite[Theorem 7.3]{KM2},}  for any $1\le k \le {d}$ there exists a constant
$c_k$ dependent on $k$ and ${d}$ such that for any $\varphi  \in {L^1}(\mathbb{R}^{k{d}})$, 
\eq{siegel}{\int_{X } {\phid  }(\vv_1,\dots \vv_k) \,d x  = c_k \int_{\left( \mathbb{R}^{d} \right)^{k}} \,{\varphi \, d\vv_1\cdots d\vv_k}. }
%end{thm}
%For the proof you can see for example Theorem 7.3 in \cite{KM2}. 
The case ${k} = 1$ corresponds to the classical Siegel transform,
$$%\hat \phi^d
%\overset{\mbox{\tiny {\wedge_2}}}{\phi} 
\widehat\varphi ({x} ) := \phione({x}) = \sum\nolimits_{\vv \in P^1({x})} {\varphi (\vv)}, $$ 
{and Siegel's summation formula \cite{Si1}, $\int_{X } {\widehat\varphi  } \,d\mu  = c_1 \int_{\mathbb{R}^{{d}}} \,\varphi(\vv) \, d\vv$.} 

Take $0 < \varepsilon < 1$, denote by $D$ the region in $\br^{{d}}$  defined
by the following system of inequalities:
\[\begin{array}{l}
|{x_{{{\ell}}}}| < \varepsilon^{mi_{{\ell}}}\,\,\,\,\,\,\,\,\,\,\,\,\,\,1 \le {{\ell}} \le m, \\
{|{x_{{{m+\ell}} }}|<\varepsilon ^{nj_{{\ell}}} \,\,\,\,\,\,\,1 \le {{\ell}} \le {n}},
\end{array}\]
%centered at the origin, 
and by
$\varphi$ the characteristic function of $D$. Note that {the volume of $D$ is equal to $\vre^d$, and that}  \[
{x} \in U_{\vr,\vs}(\varepsilon) \Leftrightarrow {x}\cap D \ne
\{0\}. \]
The latter condition clearly implies that $D$ contains at least two primitive vectors %($\vv$ and $-\vv$) 
in ${x}$. Therefore in view of  %Theorem \ref{Siegel} 
%\equ{siegel}
{Siegel's formula} we have $$\mu\big(U_{\vr,\vs}(\varepsilon)\big) \le \frac{1}{2}\int_{X } {\widehat \varphi  } \,d\mu  = \frac{1}{2} c_1 \int_{\mathbb{R}^{{d}}} \,{\varphi \, d\vv} = {{\frac{1}{2} c_1}} \varepsilon^{{d}}.$$ 
%where ${{C_9}}=\frac{1}{2} c_1.$   
For getting the lower bound, %we will demonstrate that lattices
%${x}$ with $\widehat \phi  ({x} ) > 2$
%contribute very insignificantly to the integral in the left hand
%side of Theorem \ref{Siegel}. 
%N
note that if two linearly independent primitive vectors $\vv_1$ and $\vv_2$ in ${x} \cap D$ do not form a primitive pair, then the line segment between $\vv_1$ and $\vv_2$ must contain another lattice point; and since $D$ is convex, this lattice point must be in $D$.  So one can easily see that whenever there exist at
least two linearly independent vectors in ${x}\,\cap\, D$, for any
$\vv_1\in P^1({x}) \cap D$ one can find $\vv_2\in {x}\cap D$ such
that $(\vv_1,\vv_2)$, as well as $(\vv_1,-\vv_2)$, belong to
$P^2({x})$. 
Therefore, if $\widehat \varphi  ({x} ) > 2$, one has
\[\widehat \varphi   ({x} ) = \# (P^1({x} ) \cap D) \le \frac{1}{2}\# \big({P^2}({x} ) \cap (D \times D)\big) {\,=\frac{1}{2}\psitwo ({x} )},\]
%Note that
%the right hand side is equal to $$ 
where $\psi$ is the characteristic
function of $D\times D$ in $\br^{2{k}}$. Hence, 
\[\begin{array}{l}
\int_{{X}} {\widehat \varphi } \, d\mu  = \int_{\{ {x} : \widehat \varphi  ({x} ) \le 2\} } {\widehat \varphi }\, d\mu + \int_{\{ {x} : \widehat \varphi  ({x} ) > 2\} } {\widehat \varphi   \, d\mu } \\
\,\,\,\,\,\,\,\,\,\,\,\,\,\,\,\,\,\, \le 2\mu (\{ {x} : 
\widehat \varphi ({x} ) = 2\} ) +
\frac{1}{2}\int_{\{ {x} : \widehat \varphi 
({x} ) > 2\} } {{\psitwo} 
d\mu }  \le 2\mu \big(U_{\vr,\vs}(\varepsilon)\big) +
\frac{1}{2}\int_{X} {{\psitwo} 
d\mu },
\end{array}\]
%Therefore
{which implies that}
%\eq{si}
$${2 \mu\big(U_{\vr,\vs}(\varepsilon)\big) \ge \int_{{X}} {\widehat \varphi} \, d\mu - \frac{1}{2}\int_{X} {{\psitwo} \,
d\mu }={{c_1}}\varepsilon^{{d}}-\frac{1}{2}\int_{X} {{\psitwo}} \,
d\mu .}$$
Using {the $k=2$ case of \equ{siegel}}
%Theorem \ref{Siegel} again 
yields $\int_{X} {{\psitwo}} \,
d\mu =c_2\varepsilon^{2{d}}.$ Hence {\equ{si} holds with ${{C_{9}}}= \frac{1}{2}c_1$ and} ${{C_{10}}}= \frac{1}{4}c_2$.
%\[\mu\big(U_{\vr,\vs}(\varepsilon)\big) \ge {{C_9}} \varepsilon^{{k}} - {{C_{10}}}\varepsilon^{2{k}},\]
%where ${{C_{10}}}= \frac{1}{4}c_2$.
\end{proof}

{Finally let us choose $C_{11} > 0$ such that for any $0 < r < 1$, 
$$
\max\left(\|g - I_d\|_{op}, \|g^{-1} - I_d\|_{op}\|\right) < C_{11} r\text{ whenever } g\in B^G(r).
$$
\begin{lem}\label{rcore} Let $0 < \vre < 1$ and \eq{smallr}{0 < r < \frac{2^\alpha - 1}{dC_{11}}\vre^{\max(m,n)}.} Then $$U_{\vr,\vs}(\varepsilon/2) \subset \sigma_r \big(U_{\vr,\vs}(\varepsilon)\big).$$
\end{lem}
\begin{proof} 
%Recall that it was shown in \cite[Lemma 3.6]{BK} that \eq{uepsilonbound}{r_0(U(\vre)) \ge {E}_4{\varepsilon }^{{k}}\text{ for any }0 < \vre < 1,} where ${E}_4$ is a constant depending only on ${k}$, and $U(\vre) = U_{\vm,\vn}(\varepsilon)$ is as in \equ{uepsilon}.
Take $x
\in U_{\vr,\vs}(\varepsilon/2)$ and $g\in B^G(r)$. We know that there exists $\vv\in x\nz$ such that one of the following two conditions holds:
\begin{itemize}
\item[(1)] $|v_k| < (\vre/2)^{m i_k}$ for some $1\le k \le m$;
\item[(2)]  $|v_{m+k}| < (\vre/2)^{n j_k}$ for some $1\le k \le n$.
\end{itemize}
Assuming (1) and writing $g = (a_{k\ell})_{k,\ell = 1,\dots,d}$, one has
\begin{equation*}
\begin{split}
|(g\vv)_k| &= \Big|a_{kk} v_k + \sum_{\ell\ne k}a_{k\ell} v_\ell \Big| \le (1 + C_{11}r)(\vre/2)^{m i_k} + (d-1) C_{11}r(\vre/2)^{m\alpha}\\ 
& \le (\vre/2)^{m i_k} + dC_{11}r(\vre/2)^{m\alpha} \le \frac{\vre^{m i_k}}{2^{m\alpha}}\left(1 + dC_{11}r \vre^{-m}\right),
\end{split}
\end{equation*}
which is smaller than $\vre^{m i_k}$ in view of \equ{smallr}; hence $gx\in U_{\vr,\vs}(\varepsilon)$. The argument in case of (2) is similar.
\end{proof}}

{Now we can finish the}

\begin{proof}[Proof of Theorem \ref{main theorem 3}] 
%Recall that as demonstrated in \cite{KM4}, 
%$P = \{ {u_A}:A \in {M_{m,n}}\} $ is
%a subgroup of $H$ that is normalized by $\{ {g_t}\} $ and has
% property (EEP). 
{In view of Theorem \ref{1.5} and Lemma \ref{Cusp}, one can apply   \linebreak Theorem~\ref{Main Theorem1} to $P$  as in 
\equ{subgroup mix}
%$P = \{ {u_A}:A \in {M_{m,n}}\} $ 
and conclude that 
%there exists constants $r'',{E_5}>0$ such that 
for  any $c > 0$ and any \linebreak  $0<r<\min\left(r_0\left(\partial_{{1}/{2}}(X \ssm U_{\vr,\vs}(\varepsilon)\right),r''\right)$, it holds that
%the set $$\Bad_{\vr,\vs}(c) = \left\{ {A : {u_A}{\mathbb{Z}^{{d}}} \in E(F^+,U_{\vr,\vs}(\varepsilon)) }
%\right\}$$ has Hausdorff dimension at most:
\eq{finalestimate}{\codim \Bad_{\vr,\vs}(c)   \gg \frac{{\mu \big({\sigma _{r}}{U_{\vr,\vs}(\varepsilon) }\big)}}{{\log  \frac{1}{r}  + \log  \frac{1}{{\mu \left({\sigma _{r}}{U_{\vr,\vs}(\varepsilon) }\right)}} }}}
where $\varepsilon = c^{{1}/{{d}}}$ and the implicit constant in \equ{finalestimate} is independent of $c$ but depends on $\vr,\vs$.  
Note that in view of  \equ{boundfordelta} we have $X \ssm U_{\vr,\vs}(\vre) \subset X \ssm U(\vre^{\max(m,n)})$, 
 thus
\begin{equation*}
\begin{split}
r_0\Big(\partial_{{1}/{2}}\big(X \ssm U_{\vr,\vs}(\varepsilon)\big)\Big) &\ge r_0\left(\partial_{{1}/{2}}(X \ssm U(\vre^{\max(m,n)})\right) \\ &\gg r_0\left(X \ssm U\Big(\frac1{1 + C_{11}/2}\vre^{\max(m,n)}\Big)\right) \ge \frac{C_1}{1 + C_{11}/2} \vre^{d\cdot \max(m,n)},\end{split}
\end{equation*} 
the last inequality being a consequence of Lemma \ref{injectivity1}. It follows that %there exists $C_{11} > 0$ such that 
\equ{finalestimate} holds whenever \eq{conditiononr}{r <  \frac{C_1}{1 + C_{11}/2} \vre^{d\cdot \max(m,n)} \le r'' .}
%\bro{Shouldn't the first inequality in 8.10 be strict: $${r <  \frac{C_1}{1 + C_{11}/2} \vre^{d\cdot \max(m,n)} \le r'' ?}$$}
Now define 
$$
c_0 := \min\left(\Big( \frac{1 + C_{11}/2}{C_1}r''\Big)^{1/\max(m,n)}, C_9/2C_{10}\right),$$ 
%\bro{I don't see why the second inequality in 8.10 will hold with this choice of $c_0$ unless we have $\frac{C_1}{1 + C_{11}/2}<1.$ Should we modify the definition of $c_0$ or am I missing something?}
take $\vre < c_0^{1/d}$ and consider 
$$
r = \frac12\min\left(\frac{2^\alpha - 1}{dC_{11}},\frac{C_1}{1 + C_{11}/2} \right)\vre^{d\cdot\max(m,n)}.
$$
Then both \equ{smallr} and \equ{conditiononr} will hold, and thus the right hand side of \equ{finalestimate} is not less than 
$$
\frac{{\mu \big({U_{\vr,\vs}(\varepsilon/2) }\big)}}{{\log  \frac{1}{r}  + \log  \frac{1}{{\mu \left({U_{\vr,\vs}(\varepsilon/2) }\right)}} }} \ge  \frac{{\frac12 C_9(\vre/2)^d}}{{\log  \frac{1}{r}  + \log  \frac{1}{{\frac12 C_9(\vre/2)^d}} }} \gg   \frac{{\vre^d}}{{\log  \frac{1}{\vre}   }},$$ }
which finishes the proof.
\end{proof}

{\section{Concluding remarks}\label{c}}
\subsection{Precise estimates for the Hausdorff dimension}
Note that in view of the aforementioned result of Simmons \cite{Si}
and similar results for other dynamical systems (see  e.g, \cite{FP}), it is natural to expect that when $U$ is either a small ball or the complement of a large compact subset of $X$, the codimension of $E(F^+,U)$ is, as $U$ shrinks, asymptotic to a constant times the measure of $U$. That is, conjecturally there should not be a  logarithmic term in the right hand side of \equ{mainbound}. However it is not clear how to improve our upper bound, as well as how to 
%Since we have an upper bound that may not be the optimal one, a natural question to ask is: Can we prove 
establish a complimentary lower estimate for $\dim E(F^+,U)$, using the exponential mixing of the action. Such questions can be asked in other contexts, such as for expanding maps on manifolds, see e.g.\  \cite{AN} for a lower estimate improving on \cite{U1}.
% In fact, homogeneous dynamics is not the only place where these types of problems come up. For example, in the context Anosov flows on manifolds one can see Urbanski's work in \cite{U1} where he studied the Hausdorff dimension of set of points with non-dense orbit. Also, regarding  expanding maps of Riemannian manifolds one can see Abercrombie and Nair's paper \cite{AN}, where some lower bounds are obtained for the set points whose orbit misses a  ball. 
\subsection{A dimension drop problem} Another interesting question is whether the conclusion of Theorem \ref{Main Theorem} holds without the assumption of compactness of $U^c$. It fact, it is not even known in general that the dimension of $E(F^+,U)$ is strictly smaller than the dimension of $X$ as long as $U$ is non-empty. 
%  from. In answer to a question of Barak Weiss, 
In \cite{EKP} it was  established in the case when $G$ is a connected semisimple Lie group of real rank $1$.
% with finite center and  $\Gamma$ is a  lattice in $G$, then for any non-empty open subset $U$ of $X=G/\Gamma$ it holds that
%\eq{dimension drop}{\dim E(F^+,U)< \dim X.} It is an open question whether or not the same is true for higher rank non-compact \hs s. 
One possible approach to this problem for non-compact \hs s of higher rank is to combine the methods of the present paper with estimates on the escape of mass for translates of measures on horospherical subgroups, as developed in \cite{KKLM}. 
%
%Also, in case $X$ is compact, \equ{dimension drop} is proved in \cite{KW2}. In turns out that in both \cite{EKP} and \cite{KW2} computing $\dim E(F^+,U)$ is closely related to metric entropy and escape of mass. It would be intresting to prove \equ{dimension drop} for the general case where $X$ is not necessarily compact or of rank 1. This is called dimension drop problem. One possible way to develop similar tools for higher rank groups is to use Margulis' arithmeticity theorems and embed our space in space of lattices. In space of lattices one can use for example the the height function that is defined in \cite{EMM} to study the behavior of divergent orbits. Even for the case where $X$ is as in \equ{slmn} and $F^+$ is as in \equ{gt} and $U$ is an open subset of $X$ whose complement is compact, the dimension drop problem is open. One can also consider the weighted version of \equ{gt}, that is \equ{generalgt} and ask the same question. 
This is a work in progress. % \comm{I think here we should mention recent work by Guan and Shi.} \gr
{Recenly in \cite{GS}, by generalizing the methods used in \cite{KKLM} to arbitrary homogeneous spaces, it was shown that for any one parameter subgroup action on a homogeneous space, the Hausdorff dimension of the set of points with divergent trajectories is not full.}

\bibliographystyle{alpha}

\begin{thebibliography}{22}

%\bibitem{AN} A.G.\ Abercrombie and R.\ Nair, \textsl{An exceptional set in the ergodic theory of expanding maps on manifolds}, Monatshefte f\"{u}r Mathematik, {\bf 148} (2006), no. 1, 1--17.

%\bibitem{BFKRW} R.\ Broderick, L.\ Fishman, D.\ Kleinbock, A.\ Reich and B.\ Weiss, \textsl{The set of badly approximable vectors is strongly $C^1$ incompressible}, Math.\ Proc.\ Cambridge Philos.\ Soc.\ {\bf 153} (2012), 319--339.

%\bibitem{BFS} R.\ Broderick, L.\ Fishman and D.\ Simmons, \textsl{Badly approximable systems of affine forms and incompressibility on fractals}, J.\  Number Theory {\bf 133}, no. 7 (2013), 2186--2205.
\bibitem{AN} A.\,G.\ Abercrombie  and R.\ Nair, \textsl{An exceptional set in the ergodic theory of expanding maps on manifolds}, Monatsh.\ Math.\ {\bf 148} (2006), 1--17.

\bibitem{BK} R.\ Broderick and D.\ Kleinbock, \textsl{Dimension estimates for sets of uniformly
badly approximable systems of linear forms}, Int.\ J.\ Number Theory {\bf 11} (2015), no.\ 7, 2037--2054.

\bibitem{BKM}  V.\ Bernik, D.\ Kleinbock and G.\,A.\ Margulis, \textsl{Khintchine-type theorems  on
manifolds:  the convergence case for standard  and multiplicative
versions},  Internat.\ Math.\ Res.\ Notices {\bf 2001}. no.\ 9,    
453--486.

\bibitem{Bou} N.\ Bourbaki, \textsl{El\'ements de mathematique},
Livre VI: Integration, Chapitre 7: Mesure de 
Haar, Chapitre 8: Convolution et representations, Hermann, Paris, 1963.


%\bibitem{Bu} Y.\ Bugeaud, \textsl{Approximation by Algebraic Numbers}, Cambridge Tracts in Mathematics {\bf 160}, Cambridge University Press (2007).

%\bibitem{C} J.\,W.\,S.\ Cassels, \textsl{An introduction to Diophantine approximation}, Cambridge Tracts in Mathematics and Mathematical Physics {\bf 45}, Cambridge University Press (1957).

\bibitem{dani}  S.\,G.\ Dani,
\textsl{Divergent trajectories of flows on
homogeneous spaces and Diophantine approximation},
J.\ Reine Angew.\ Math.\ {\bf 359} (1985), 55--89.

%\bibitem{EH} E.\ Hebey, \textsl{Sobolev Spaces on Riemannian
%Manifolds}, Springer Science $\& $ Business Media, Oct 2 1996, 115
%pages.

%\bibitem{EKP}  M.\ Einsiedler, S.\ Kadyrov and A.\ Pohl, \textsl{Escape of mass and entropy for diagonal flows in
%real rank one situations.}, Israel J.\ Math.\ {\bf 210} (2015), no.\ 1, 245--295. 


\bibitem{DM}
S.\,G.\ Dani and G.\,A.\ Margulis,
\textsl{Limit distributions of orbits of unipotent flows and values of
  quadratic forms},
in {\em I.\,M.\ Gelfand Seminar}, pp.\ 91--137,   Adv.\ Soviet Math., {\bf 16}, Part 1, Amer.\ Math.\
Soc.,
  Providence, RI, 1993.
  
 \bibitem{EKP}  M.\ Einsiedler, S.\ Kadyrov and A.\ Pohl, \textsl{Escape of mass and entropy for diagonal flows in
real rank one situations}, Israel J.\ Math.\ {\bf 210} (2015), no.\ 1, 245--295.  

%\bibitem{EMM} A.\ Eskin, G.\ Margulis and S.\ Mozes, \textsl{Upper bounds and asymptotics in a quantitative
%version of the Oppenheim conjecture}, Ann.\ of Math.\ ({\bf 2}) 147 (1998), no.\ 1, 93--141. 
  
 \bibitem{FP} A.\ Ferguson and M.\ Pollicott, \textsl{Escape rates for Gibbs measures}, Ergodic Theory Dynam.\ Systems {\bf 32} (2012), 961--988.
\bibitem{GS}
L.\ Guan and R.\ Shi, \textsl{Hausdorff dimension of divergent trajectories on homogeneous space}, preprint  (2018), {\tt  arxiv.org/abs/1805.07444}.

%\bibitem{GK} P.\,M.\ Gruber and C.\,G.\ Lekkerkerker,
%\textsl{Geometry of Numbers}, 2nd ed., North-Holland, Amsterdam, 1987.

%\bibitem{H} D.\ Hensley, \textsl{Continued fractions, Cantor sets, Hausdorff dimension and functional analysis}, J.\ Number Theory {\bf 40} (1992), 336--358.

\bibitem{K} S.\ Kadyrov, \textsl{Exceptional sets in homogeneous spaces
and Hausdorff dimension}, Dyn.\ Syst.\ {\bf 30} (2015), no.\ 2, 149--157.



\bibitem{K1} D.\  Kleinbock,
\textsl{Nondense orbits of flows on homogeneous spaces},
Ergodic Theory Dynam.\ Systems {\bf 18} (1998), 373--396.

\bibitem{K2} \bysame, \textsl{Flows on homogeneous spaces and
Diophantine properties of matrices}, Duke Math.\ J. {\bf 95} (1998), 107--124.


%\bibitem%[KLW]
%{KLW}D.\ Kleinbock, E.\ Lindenstrauss and B.\ Weiss,
%\textsl{On fractal measures and diophantine approximation}, Selecta Math.\
%{\bf 10} (2004), 479--523.


\bibitem{KKLM}   S.\ Kadyrov, D.\ Kleinbock, E.\ Lindenstrauss  and G.\,A.\ Margulis, \textsl{Singular systems of linear forms 
and non-escape of mass in the space of lattices},  J.\ Anal.\ Math.\ {\bf 133} (2017),  253--277. 

\bibitem{KM1} D.\ Kleinbock and G.\,A.\ Margulis, \textsl{Bounded orbits of nonquasiunipotent flows on homogeneous spaces}, Sina\'i's Moscow Seminar on Dynamical Systems, 141--172, Amer.\ Math.\ Soc.\ Trans.\ Ser.\ 2, Amer.\ Math.\ Soc.\, Providence, RI, 1996.
\bibitem{KM2} \bysame, \textsl{Logarithm laws for flows on homogeneous spaces}, Invent.\ Math.\ {\bf 138} (1999), no.\ 3, 451--494.
%[21] D. Kleinbock, R. Shi, and B. Weiss, Pointwise equidistr
%ibution with an error rate and with respect 

\bibitem{KM4} \bysame, \textsl{On effective equidistribution of expanding translates of certain orbits in the space of
lattices}, in: Number Theory, Analysis and Geometry, Springer, New York, 2012, pp. 385--396.

%\bibitem{Kro}L.\ Kronecker, \textsl{Zwei S\"atse
%\"uber Gleichungen mit ganzzahligen CoefÞcienten}, J.\ Reine Angew.\ Math.\ {\bf 53}
%(1857), 173--175; see also Werke, Vol.\ 1, 103--108, Chelsea Publishing Co., New York, 1968.


%\bibitem{KTV}S.\ Kristensen, R.\ Thorn and S.L.\ Velani,
%\textsl{Diophantine approximation and badly approximable sets},
%Advances in Math.\ {\bf 203} (2006), 132--169.

%\bibitem{Ku} J.\ Kurzweil, \textsl{A contribution to the metric theory of Diophantine approximations},
%        Czechoslovak Math.\ J.\ {\bf 1} (1951), 149--178.



\bibitem{KW}
D.\ Kleinbock and B.\ Weiss, \textsl{Dirichlet's theorem on Diophantine approximation
and homogeneous flows},  J.\ Mod.\ Dyn. {\bf 4} (2008), 43--62.
%\bibitem{KW2} \bysame, %D.\  Kleinbock and B.\ Weiss,
%\textsl{Modified Schmidt games and  a conjecture of Margulis}, J.\ Mod.\ Dyn.\  {\bf 7}, no.\ 3 (2013), 429--460.

%\bibitem{KW2} \bysame, %D.\  Kleinbock and B.\ Weiss,
%\textsl{Modified Schmidt games and  a conjecture of Margulis}, J.\ Mod.\ Dyn.\  {\bf 7}, no.\ 3 (2013), 429--460.



\bibitem{PV} A.\ Pollington and S.\ Velani, \textsl{On simultaneously badly approximable numbers,}
J.\ London Math.\ Soc. (2) {\bf 66} (2002), no.\ 1, 29--40.

%\bibitem{S1}W.\,M.\ Schmidt, \textsl{On badly approximable numbers and certain games},
%Trans.\ A.M.S. {\bf 123} (1966), 27--50.

%\bibitem{S2}    \bysame,  \textsl{Badly approximable systems of
%linear forms}, J. Number Theory {\bf 1} (1969), 139--154.
\bibitem{Si1} C.\  L.\ Siegel, \textsl{A mean value theorem in geometry of numbers}, Ann. Math. 46 (1945), 340-347.
\bibitem{Si}  D.\ Simmons, \textsl{A Hausdorff measure version of the Jarn\'ik--Schmidt theorem in Diophantine approximation},   Math.\ Proc.\ Cambridge Philos.\ Soc.\ 
{\bf 164} (2018), no.\ 3. 413--459.



\bibitem{St}  A.\ Starkov, \textsl{Dynamical systems on homogeneous spaces}, Translations of Mathematical Monographs, {\bf 190}, American Mathematical Society, Providence, RI, 2000.

\bibitem{U1} M.\ Urba\'nski,\textsl{ The Hausdorff dimension of the set of points with nondense orbit under a hyperbolic dynamical system}, Nonlinearity {\bf 4} (1991), 385--397. 

\bibitem{Weg} H.\ Wegmann,  \textsl{Die Hausdorff-Dimension von
kartesischen Produktmengen in metrischen R\" aumen},  J.\ Reine
Angew.\ Math.\ {\bf 234} (1969), 163--171.







\end{thebibliography}

\end{document}